\documentclass[]{interact}

\usepackage{mathrsfs}
\usepackage[colorlinks,citecolor=blue,urlcolor=blue]{hyperref}
\usepackage{float}

\usepackage{epstopdf}% To incorporate .eps illustrations using PDFLaTeX, etc.
\usepackage[caption=false]{subfig}% Support for small, `sub' figures and tables
\usepackage[natbibapa,nodoi]{apacite}% Citation support using apacite.sty. Commands using natbib.sty MUST be deactivated first!
\theoremstyle{plain}% Theorem-like structures provided by amsthm.sty
\newtheorem{theorem}{Theorem}[section]
\newtheorem{lemma}[theorem]{Lemma}

\theoremstyle{definition}
\newtheorem{definition}[theorem]{Definition}

\theoremstyle{remark}
\newtheorem{remark}{Remark}

\begin{document}

%\articletype{ARTICLE TEMPLATE}% Specify the article type or omit as appropriate

\title{Stabilization of linear waves with inhomogeneous Neumann boundary conditions}

\author{
\name{Türker Özsarı\textsuperscript{a} and İdem Susuzlu\textsuperscript{b}}
\affil{\textsuperscript{a}Department of Mathematics, Bilkent University, Ankara, Turkey; \textsuperscript{b}Department of Mathematics, Izmir Institute of Technology, Izmir, Turkey}
}

\maketitle

\begin{abstract}
  We study linear damped and viscoelastic wave equations evolving on a bounded domain. For both models, we assume that waves are subject to an inhomogeneous Neumann boundary condition on a portion of the domain's boundary. The analysis of these models presents additional interesting features and challenges compared to their homogeneous counterparts. In the present context, energy depends on the boundary trace of velocity. It is not clear in advance how this quantity should be controlled based on the given data, due to regularity issues. However, we establish global existence and also prove uniform stabilization of solutions with decay rates characterized by the Neumann input. We supplement these results with numerical simulations in which the data do not necessarily satisfy the given assumptions for decay. These simulations provide, at a numerical level, insights into how energy could possibly change in the presence of, for example, improper data.
\end{abstract}


\section{Introduction}
\subsection{Mathematical models and notation}
In this paper, we examine the global existence and stabilization of the linear damped and viscoelastic wave equations, which are denoted as follows:
\begin{align}\label{damped}
	u_{tt}-\bigtriangleup u & + a(x)u_{t}=0
\end{align} 
and
\begin{equation}\label{lineq}u_{tt}-\bigtriangleup u + \int_{0}^{t}g(t-s)\bigtriangleup u(s)ds=0.
\end{equation}
We consider both equations on a general domain $\Omega\subset \mathbb{R}^d$ ($d\ge 1$) with a smooth boundary $\Gamma=\partial \Omega$ in an arbitrary time interval $(0,T)$.  The boundary is partitioned into two parts $\Gamma_i$, $i=0,1$ that $\Gamma=\overline{\Gamma_{0}}\cup \overline{\Gamma_{1}}$ with $\Gamma_{0}\neq \emptyset$ and $\Gamma_{0}\cap  \Gamma_{1} =\emptyset$. We assume the presence of an inhomogeneous Neumann manipulation on $\Gamma_1$:
\begin{equation}\label{neumann-bc}{\frac{\partial u }{\partial n}}\Big|_{\Gamma_{1}} = h(x,t),\end{equation}
whereas homogeneous Dirichlet boundary condition is imposed on $\Gamma_0$:
\begin{equation}\label{dirichlet-bc}u|_{\Gamma_{0}} =0.\end{equation}  The Neumann boundary input and initial data satisfy \begin{equation}h\in C^1([0,\infty);L^2(\Gamma_1))\end{equation} 
and \begin{equation}\label{initial-cond}(u(0),u_t(0))=(u^0,u^1)\in H_{\Gamma_0}^1(\Omega)\times L^2(\Omega),
\end{equation} respectively.  

Regarding the linear damped wave equation \eqref{damped}, we assume that the distributed damping coefficient $a=a(x)$ is bounded from below and above by two positive constants on $\Omega$: 
\begin{equation}\label{a(x)}
	a_{\max}\geq a(x)\geq a_{\min}>0.
\end{equation}
We associate the linear damped wave equation in \eqref{damped} with the energy functional 
\begin{equation}\label{energy-damped}
	\varepsilon_a (t):= \frac{1}{2} \| \nabla u(t)  \|_{L^2(\Omega)}^{2}+\frac{1}{2}\| u_{t}(t)\|_{L^2(\Omega)}^{2}.
\end{equation}

Regarding the viscoelastic wave equation \eqref{lineq}, we assume that (the relaxation function) $g=g(t)$ satisfies the following properties:
\begin{itemize}
	\item[(A1)] $g\in C^1([0,\infty);\mathbb{R}_{+})$ with $1-2\int_{0}^{\infty}g(s)ds>L$ for some $L\in (0,1)$,
	\item[(A2)] there exists $\xi\in C^1([0,\infty); \mathbb{R}_{+})$ satisfying
	$$g'(t)\leq -\xi(t)g(t),\ \ \left|\frac{\xi'(t)}{\xi(t)}\right| \leq k, \ \ \xi'(t)\leq 0,  \ \ \forall t >0. $$
\end{itemize}
\noindent Hypothesis (A2) has been widely used in the literature of viscoelastic wave equations; see for instance, \citet{messaoudi2008} and \citet{paul2022}. In addition, \citet{han2009}, \citet{liu2009-2}, and \citet{liu2012} use very similar assumptions but are supplemented with a few further global conditions.  The assumption (A2) elucidates the decay rate of the relaxation function \( g(t) \). This rate is intrinsically linked to the rate at which the material dissipates energy. A higher relaxation rate corresponds to a more rapid energy dissipation.
Knowing this rate offers insights into the viscoelastic behavior of the material under various loading conditions.  

The viscoelastic wave equation is associated with the (modified) energy functional
\begin{equation}\label{defmod}\varepsilon_g (t):= \frac{1}{2}\bigg(1-\int_{0}^{t}g(s)ds\bigg) \parallel \nabla u(t) \parallel_{L^2(\Omega)}^{2}+\frac{1}{2}\parallel u_{t}(t)\parallel_{L^2(\Omega)}^{2}+\frac{1}{2}(g\circ \nabla u)(t),\end{equation}
where, for $\phi:(0,t)\rightarrow L^2(\Omega)$, we set
$$(g\circ \phi)(t):= \int_{0}^{t}g(t-s)\parallel \phi(t)-\phi(s) \parallel_{L^2(\Omega)}^2 ds .$$
We denote the inner products on $L^2(\Omega)$ and $L^2(\Gamma_1)$ by $(\cdot,\cdot)_{\Omega}$ and $\langle\cdot,\cdot\rangle_{\Gamma_1}$, respectively. The Banach space pairing for $H_{\Gamma_0}^1(\Omega)$ and its dual $(H_{\Gamma_0}^1(\Omega))'$ will be denoted by $<\cdot,\cdot>$.  

We write $c$ for a (generic) positive constant in estimates. Any subscript appended to $c$ is used to emphasize a dependence, e.g., $c_T$ is a constant which depends on $T$.  In the latter case, $c_T$ may depend on some other parameters too such as domain $\Omega$, but we may not write them as subscripts if they are not crucial.

\begin{remark}We sometimes use the term ``Neumann manipulation'' instead of ``inhomogeneous Neumann boundary condition'' since the behavior of solutions can be altered depending on the choice of Neumann input. For instance, in certain frameworks such as control theory, the boundary inputs are controls acting on the open loop systems with the aim of manipulating the behavior of solutions in a prescribed sense.  Note that one can for example choose the boundary input in a certain way to pump up the energy of the system, thereby manipulating the qualitative properties of solutions. This  will be illustrated numerically in subsequent chapters.
\end{remark}

\subsection{Motivation and orientation} Studying wave models subject to external manipulations offers potential applications in engineering. 
There is extensive literature on linear and nonlinear damped and viscoelastic wave models. Most of these papers deal with the case of homogeneous boundary conditions.  However, it is important to study the effects of external manipulations on the qualitative behavior of solutions.  This boosts the physical realism by accounting for boundary noise, helps us understand the stability and features of the system, and allows us to build boundary controls that govern the dynamics in a prescribed manner.

This work sheds light on the mathematical properties of certain physical wave models where inhomogeneous Neumann boundary conditions play a crucial role.  From a physical point of view, a homogeneous Neumann condition implies no flux or flow of the wave across the boundary, where the wave is reflected back into the domain with no loss of energy. In contrast, an inhomogeneous Neumann condition involves a prescribed flux or flow across the boundary, which could represent an external source or sink of energy, momentum, or pressure at the surface. In the context of waves, this means that the boundary is not perfectly reflective, allowing energy to either enter or leave the domain, depending on the spatio-temporal values of Neumann data.   

There are tangible examples of damped and viscoelastic wave models subject to inhomogeneous Neumann boundary condition in the physical world. An example of the damped wave model is a string stretched between two points, subjected to air resistance or internal friction within the material, with one end being manipulated by an external force that varies over time. Such a force can be created for instance via an external device that pushes or pulls on one end of the string, introducing a non-zero flux of wave energy at that boundary. This external force either adds energy to or extracts energy from the system, altering the wave's behavior throughout the domain and at the boundary. 

Designing new viscoelastic materials requires understanding how these materials respond to external forces.  Recent efforts in applied mathematics have explored viscoelasticity in the presence of inhomogeneous boundary data.  For instance, \citet{eruslu18} and \citet{erusluphd} studied waves propagating in a viscoelastic solid subject to inhomogeneous Dirichlet and Neumann type manipulations.  Unlike the equation \eqref{lineq} considered here, their main equations did not involve the Laplacian.    In the works of \citet{eruslu18} and \citet{erusluphd}, the authors investigated the stability of various viscoelastic models (Zener, Maxwell, and Voigt) characterized by the material properties of the solid.  They also provided various numerical experiments highlighting differences between these viscoelastic models.  

Another physical application of the viscoelastic wave model with inhomogeneous boundary conditions could be  modeling the effect of human activities on seismic waves.  Mechanical loads such as heavy vehicles or machinery can introduce additional stress into the Earth's crust, affecting the propagation of seismic waves.

It may be assumed that the damped and viscoelastic wave equations with external manipulations can be handled as those with homogeneous boundary conditions, but this is not the case.  For instance, when $h\equiv 0$ or $\Gamma_{1}=\emptyset$, it is easy to show that the energy of the solution for both problems is nonincreasing, as homogeneous case allows for a straightforward proof that
\begin{align*}
	\varepsilon'_a(t) = -\int_{\Omega}a(x)|u_t(t)|^2 dx \leq 0.
\end{align*}
and
$$\varepsilon'_g (t)= \frac{1}{2}(g'\circ \nabla u)(t)-\frac{1}{2}g(t)\|\nabla u(t)\|_{L^2(\Omega)}^2\le 0.$$ The above inequalities indicate that the energy can be effectively controlled by initial data chosen from a physically reasonable space. However, when the evolution is subject to external manipulation, as in the present paper, the rate of energy change takes the following complex forms for the respective problems, assuming a sufficiently smooth solution:
\begin{align}\label{neumann-energydot-damped}
	\varepsilon'_a(t) = -\int_{\Omega}a(x)|u_t(t)|^2 dx +\int_{\Gamma_1} h u_t d\Gamma,
\end{align}
\begin{align}\label{e1}
	\begin{split}
		\varepsilon'_g(t) = &\frac{1}{2}(g'\circ \nabla u)(t)-\frac{1}{2}g(t)\|\nabla u(t)\|_{L^2(\Omega)}^2+\int_{\Gamma_{1}}u_{t}(t)h(t)d\Gamma\\
		&-\int_{\Gamma_{1}}u_{t}(t)\int_{0}^{t}g(t-s)h(s)dsd\Gamma.
	\end{split}
\end{align}From these expressions, it is unclear how the energy behaves, as we do not know in advance how the boundary integrals involving $u_t|_{\Gamma_1}$ are controlled in terms of given data. In particular, for a solution in the energy space with $u_t\in L^\infty(0,T;L^2(\Omega))$, a well-defined Dirichlet trace $u_t|_{\Gamma_1}$ is not guaranteed according to Sobolev trace theory. Therefore, the boundary integrals in the energy estimates require further analysis, which is addressed in the next section in several steps:  

\begin{itemize}
	\item[Step 1:] We begin by employing the \emph{dynamic extension} method to homogenize boundary conditions. Specifically, we extend the given Neumann boundary datum as the unique weak solution of the  initial boundary value problem (ibvp) associated with the classical wave equation. This approach allows us to establish regularity properties of the dynamic extension model using techniques for nonhomogeneous ibvps.  This method is more efficient than taking an arbitrary extension of boundary datum, see Remark \ref{dynext} below.
	\item [Step 2:] Next, we construct the weak solutions for the homogenized models.  For the damped wave equation, we employ the semigroup approach, while for the viscoelastic model, we first construct approximate solutions using the Faedo-Galerkin method, and then establish a priori estimates for these approximations, which ultimately lead to existence of a weak solution for the homogenized model through compactness techniques. The global unique solutions of the original models are then obtained through a reunification argument. 
	\item[Step 3:] We establish the general uniform decay estimates using the multiplier (energy) technique.  The most challenging aspect of this process is managing boundary traces, such as $u_t|_{\Gamma_1}$, which are absent in the case of homogeneous boundary conditions.  To address this challenge, we perform integration by parts in time to relax the strict regularity requirements for weak solutions.  The main results of the paper are presented in Theorem \ref{TemamMainThm} and Theorem \ref{mainresultthm}.
\end{itemize}

The theoretical results on general decay rates are also supported by numerical simulations in the one-dimensional setting. We develop the numerical solution for the wave equation with a distributed damping using an explicit method, while for the viscoelastic wave equation, we employ the Crank-Nicolson method.  Several numerical experiments are provided to offer strong physical insights into how the energy of solutions can be modified by the combined effects of damping and external manipulation.  Our simulations include cases where the boundary datum does not necessarily satisfy the assumptions needed for establishing general uniform decay rates.  Indeed, these controversial examples demonstrate that the energy of solutions can even increase or be kept away from zero through appropriate external manipulation.

\subsection{Previous studies}
Evolution equations modeling linear waves with distributed damping and viscoelasticity have been widely studied in the literature.  \citet{rauch74} proved exponential decay of solutions to hyperbolic equations on manifolds with distributed damping (see also \citealp{rauch1976}). In one dimensional setting, \citet{cox95} identified the optimal decay rates associated with distributed damping coefficients of bounded variation. \cite{shu97}  established the exact interior controllability of the the one dimensional model.  For uniform stabilization of the wave equation in higher dimensions with a distributed damping, see e.g., \cite{zua85}, \cite{kom94}, \cite{nakao97}, \cite{teb98}, and \cite{mart00}. \cite{daou12} investigated the decay of solutions with a distributed damping for those evolutions subject to an internal manipulation.  For a general overview of damped wave equations, refer to \cite{ala10}.

In the early 1970s, \cite{dafermos1970} studied the one dimensional viscoelastic wave equation with Dirichlet boundary conditions.  He proved existence, uniqueness and asymptotic stability of solutions associated with a specified history but did not provide decay rate estimates.  \cite{dassios1992} later investigated linear viscoelasticity in $\mathbb{R}^3$ and established (polynomial) decay rates for kinetic and strain energies of longitudinal and transverse components of solutions.  \cite{rivera1994} analyzed linear isotropic homogeneous viscoelastic solids of integral type on domains and proved exponential decay for exponentially decaying memory kernels on bounded domains, whereas only a polynomial decay was established on the whole space.  \cite{rivera1996} extended this analysis, showing that algebraically decaying kernels yield algebraically decaying solutions, where the rate is determined by the relaxation function. \cite{rivera2001} considered a (partially) viscoelastic wave equation where memory takes effect only on a subregion of the underlying physical medium.  They proved that energy decays exponentially if the relaxation function decays similarly, with dissipation near boundary.  \cite{rivera2003} investigated a class of abstract viscoelastic systems in which memory term has the form  $g\ast \mathcal{A}^\alpha u$ with  $\mathcal{A}$ being a strictly positive, self-adjoint operator. They showed that solutions decay polynomially when $0 < \alpha < 1$,  even with  exponentially decaying kernels.  \cite{Fab02} established that decay rates in viscoelastic wave equations with linear damping are determined by the relaxation function. \cite{messaoudi2008} expanded classical decay rates (exponential, polynomial) to more general relaxation functions. Most recently, there have been some effort in the direction of further weakening the assumptions on the relaxation function and proving general decay rates, see e.g., \cite{messaoudi2009}, \cite{mustafa-messaoudi2012} and \cite{alabau2009}, \cite{Las13} and \cite{mustafa2017}.  Qualitative behavior of solutions of viscoelastic wave equations has also been studied in other directions. For an overview of decay properties in viscoelastic wave equations with memory terms, in which temporal integration is taken over the half line, see the survey by  \cite{Vit09}.

Regarding homogeneous Neumann boundary conditions, few results exist for similar models.  \cite{mart00} is among the few that addresses locally damped wave models with homogeneous Neumann boundary conditions. He proves decay of the energy by using a method incorporating piecewise multipliers. His study is based on the construction of multipliers adapted to geometrical conditions.  However, unlike Martinez’s case, in our study with nonhomogeneous Neumann boundary conditions, energy decay or a specific decay rate is not guaranteed.  In particular, initial data that lead to exponentially or polynomially decaying solutions under homogeneous Neumann conditions do not necessarily yield solutions with similar behavior satisfying nonhomogeneous boundary conditions.  In the present work, whether decay occurs depends on the temporal behavior of the boundary manipulation, and in some cases, there may be no decay, such as when the boundary manipulation grows over time in a relevant norm.

To the best of our knowledge, for the viscoelastic model with homogeneous Neumann boundary conditions, no closely related studies exist.  However, research has been conducted with acoustic boundary conditions, see e.g., \cite{baukhatem},  and nonlinear boundary feedback or damping, see e.g.,  \cite{cavalcantidissipative},  \cite{suggested}, \cite{farida2020}, and \cite{jiali2020}, where the energy is nonincreasing and various decay rates are achieved depending on the relaxation function.  

Interestingly, the nonhomogeneous Dirichlet problem is more difficult to address than the nonhomogeneous Neumann problem studied here.  The former requires controlling a boundary term involving the outward normal derivative for which one typically establishes hidden trace estimates using special multipliers (e.g., $q\cdot \nabla u$, $q$ is a smooth vector field). These estimates are hindered by the viscoelastic term.  In contrast, for nonhomogeneous Neumann boundary conditions, the boundary term involves the trace of the velocity, which is easier to control via integration by parts and Sobolev trace estimates.  To maintain consistency, we also consider the damped wave equation with nonhomogeneous Neumann boundary conditions. It is worth noting that damped wave equations are more amenable to hidden trace estimates, so handling the corresponding problem with nonhomogeneous Dirichlet conditions is feasible but beyond the scope of this paper.

\section{Dynamic extension for the wave equation}
\label{linv}
We begin by introducing a dynamic extension method and use this for extending the given inhomogeneous Neumann boundary input as well as initial data as the solution of the classical wave equation. This approach will later help in homogenizing the linear wave models under consideration.  More rigorously, we define the dynamic extension map $E:(h,u^0,u^1)\mapsto v$ where $v=E(h,u^0,u^1)$ is defined to be the solution of the following initial boundary value problem:
\begin{align}\label{vpart}
	\left\{
	\begin{array}{ll}
		v_{tt} =\bigtriangleup v   &\text{in } \Omega\times (0,T),\\
		v(0)=u^{0},\  v_{t}(0)=u^{1}\quad &\text{in } \Omega,\\
		v =0  & \textrm{in} \ \  \Gamma_{0}\times (0,T),   \\
		{\frac{\partial v }{\partial n}} = h & \textrm{in} \ \  \Gamma_{1}\times (0,T).
	\end{array}
	\right.
\end{align}
The dynamic extension operator has the following crucial property:
\begin{lemma}\label{extlem}
	$E:(h,u^0,u^1)\rightarrow (v,v_t)$ is continuous from $$C^1([0,T];L^2(\Gamma_1))\times H_{\Gamma_0}^1(\Omega)\times L^{2}(\Omega)$$ into $$C([0,T];H_{\Gamma_0}^1(\Omega))\times C([0,T];L^2(\Omega)).$$
\end{lemma}
\begin{proof}For details, see \citep[Remark 2.3]{pitts2002}.
\end{proof}

We can further improve the regularity under stronger assumptions on the initial and boundary data.
\begin{lemma}\label{extlemreg}
	$E:(h,u^0,u^1)\rightarrow (v,v_t)$ is continuous from the subspace of $$C^2([0,T];L^2(\Gamma_1))\cap C([0,T];H^{1/2}(\Gamma_1))\times H^2(\Omega)\cap H_{\Gamma_0}^1(\Omega)\times H_{\Gamma_0}^{1}(\Omega)$$ with elements satisfying $\displaystyle{\left.\frac{\partial u^0 }{\partial n}\right|}_{\Gamma_{1}} = h(0)$ into $$C([0,T];H^2(\Omega)\cap H_{\Gamma_0}^1(\Omega))\times C([0,T];H_{\Gamma_0}^1(\Omega)).$$
\end{lemma}

\begin{proof}
	Let $v=E(h,u^0,u^1)$. Setting $z=v_t$ and taking the derivative of the main equation in \eqref{vpart} with respect to temporal variable, we find that $z$ solves the following ibvp:
	\begin{align}\label{vpartder}
		\left\{
		\begin{array}{ll}
			z_{tt} =\bigtriangleup z \quad  \text{in } \Omega\times (0,T),\\
			z(0)=u^{1},\  z_{t}(0)=\bigtriangleup u^{0}\quad \text{in } \Omega,\\
			z|_{\Gamma_{0}} =0,   \\
			{\frac{\partial z }{\partial n}|}_{\Gamma_{1}} = h_t.
		\end{array}
		\right.
	\end{align} Given $u^0\in H^2(\Omega)\cap H_{\Gamma_0}^1(\Omega)$, $u^1\in H_{\Gamma_0}^1(\Omega)$ and $h\in C^2([0,T];L^2(\Gamma_1))\cap C([0,T];H^{1/2}(\Gamma_1))$, we have $\bigtriangleup u^{0}\in L^2(\Omega)$ and $h_t\in C^1([0,T];L^2(\Gamma_1))$. By applying Lemma \ref{extlem}, we deduce $z\in C([0,T];H_{\Gamma_0}^1(\Omega))$ with $z_t\in C([0,T];L^2(\Omega))$. \eqref{vpart} leads to the following nonhomogeneous elliptic problem:
	\begin{align}\label{vpartnew}
		\left\{
		\begin{array}{ll}
			\bigtriangleup v = z_t\in C([0,T];L^2(\Omega)),\\
			v|_{\Gamma_{0}} =0,   \\
			{\frac{\partial v }{\partial n}|}_{\Gamma_{1}} = h\in C([0,T];H^{1/2}(\Gamma_1)).
		\end{array}
		\right.
	\end{align} By classical elliptic regularity results, we conclude
	$v\in C([0,T];H^2(\Omega))$.
\end{proof}

\begin{remark}\label{dynext}Dynamic extension method was formally introduced by the first author in the context of analyzing and stabilizing solutions to the defocusing nonlinear Schrödinger equation (NLS) subject to an inhomogeneous Dirichlet boundary condition (see \citealp{OThesis} and \citealp{OKL2011}). This technique was later applied to similar models, including focusing NLS, multiple nonlinearities, and inhomogeneous Neumann boundary conditions, see \cite{O2012}, \cite{O2013}, and \cite{O2015}.  
	
	The dynamic extension method is more efficient than arbitrary spatial extensions of boundary datum, as the latter approach introduces an interior source term during homogenization, which may harm the optimal regularity properties of the original wave model. For instance, assume $v$ is an arbitrary function defined on $\bar{\Omega}$ such that  $$v|_{\Gamma_{0}} =0, \ \ {\left.\frac{\partial v }{\partial n}\right|}_{\Gamma_{1}} = h.$$ By setting $w:=u-v$  we see that $w$ satisfies homogeneous boundary conditions, along with the following equations associated with the linear damped and viscoelastic wave equations, where the source terms $f_d$ and  $f_v$ are given by: 
	\begin{align}
		&w_{tt}-\bigtriangleup w + aw_t=f_d(x,t),
	\end{align}
	\begin{align}
		&w_{tt}-\bigtriangleup w + \int_{0}^{t}g(t-s)\bigtriangleup w(s)ds=f_v(x,t),
	\end{align}
	where
	\begin{align*}
		&f_d(x,t)=-v_{tt}+ \bigtriangleup v-a v_t,\\
		&f_v(x,t)=-v_{tt}+ \bigtriangleup v -\int_{0}^{t}g(t-s)\bigtriangleup v(s)ds.
	\end{align*} 
	The above source terms would not appear if the dynamic extension method were used.  Notice that when the boundary datum is rough, the arbitrary extension $v$ may lack smoothness, resulting in rough source terms $f_d$ and$ f_v$, which can harm the regularity of the original PDE model with inhomogeneous boundary conditions.  Additionally, using an arbitrary extension complicates the homogenized model, making it less amenable to multiplier techniques, more so in nonlinear problems.
\end{remark}

\section{Damped wave equation}
\subsection{Homogenization}
\label{linw_damped} Setting $w\equiv u-v=u-E(h,u_0,u_1)$, we observe that to solve the ibvp \eqref{damped}, \eqref{neumann-bc}-\eqref{dirichlet-bc}, \eqref{initial-cond}, it is enough to obtain a solution of the following initial-boundary value problem with zero initial data and homogeneous Dirichlet-Neumann boundary conditions:

\begin{equation}\label{w_part_damped}
	\left\{
	\begin{array}{l l}
		w_{tt}-\bigtriangleup w + a(x)(w_t+v_t)=0  & \textrm{in} \ \ \Omega \times (0,T), \\
		w(x,0)=0,  w_{t}(x,0)=0  &   \textrm{in} \ \ \Omega,  \\
		w|_{\Gamma_{0}} =0, \\
		{\frac{\partial w }{\partial n}}\Big|_{\Gamma_{1}} = 0.
	\end{array}
	\right. 
\end{equation}
We prove the following well-posedness result for the above problem.
\begin{lemma}\label{reg_damped_w} Let $(h,u_0,u_1)\in C^1([0,T];L^2(\Gamma_1))\times H_{\Gamma_0}^1(\Omega)\times L^{2}(\Omega)$, $v=E(h,u_0,u_1)$, and  $a\in L^{\infty}(\Omega)$, then \eqref{w_part_damped}  has a unique weak solution   $w\in C([0,T];H_{\Gamma_0}^1(\Omega))$ with $w_t\in C([0,T];L^2(\Omega)). $
\end{lemma}
\begin{proof}
	To prove the well-posedness of \eqref{w_part_damped}, we employ the semigroup approach.  We first transform (\ref{w_part_damped}) into an abstract operator theoretic form by introducing the new unknown $\tilde{w}:=\frac{\partial }{\partial t}w$. This allows us to rewrite (\ref{w_part_damped}) as
	\begin{align}\label{system1}
		\left\{
		\begin{array}{ll}
			w_{t}-\tilde{w}=0,\\
			\tilde{w}_{t}+A w+a(x)w_t =-a(x)v_{t}\in L^2(\Omega)
		\end{array}
		\right.
	\end{align}
	where $A:=  -\Delta $ and
	$ D(A)= \{ u\in H^2(\Omega) :  u|_{\Gamma_{0}}=\frac{\partial u }{\partial n}|_{\Gamma_{1}}=0 \}.$ 
	Introducing the vector function $W:=(w,\tilde{w})^{T}$, the system  (\ref{system1})  can now be written as 
	\begin{align}
		\begin{split} \label{abstrac_form_w}
			&\frac{d}{dt}W+\mathcal{A}W=F,
		\end{split}
	\end{align}
	where
	$\mathcal{A}=
	\begin{bmatrix}
		0 & -I \\
		A & a(x) \\
	\end{bmatrix}
	$ and $F=(0,-a(x)v_{t})^{T},$ moreover
	$
	W|_{t=0}=W(0)=(0,0).
	$ Let  $H:=H_{\Gamma_0}^1(\Omega)\times L^2(\Omega). $ It is well known that $\mathcal{A}$ generates a $C_0$-semigroup $(\mathscr{S}(t))_{t\geq 0}$ on $H,$ 
	see for instance \cite{Tebeu1998} whose proof for the Dirichlet case can easily be adapted to our Dirichlet-Neumann case.  Moreover,  if $F \in L^1([0,T];H)$, then $W \in C([0,T];H_{\Gamma_0}^1(\Omega)) \times C^1([0,T];L^2(\Omega))$, see e.g., \cite[Cor 2.2]{pazy}.  Note that $a\in L^{\infty}(\Omega)$ and $ v_t \in C([0,T];L^2(\Omega))$, therefore  $a(x)v_t\in L^1([0,T];L^2(\Omega))$ so that $F\in L^1([0,T];H)$.  Thus, we conclude that (\ref{w_part_damped}) admits a unique mild solution.
	
\end{proof}

\begin{lemma}\label{improved_reg_w} Let $a\in L^{\infty}(\Omega)$ and $v=E(h,u_0,u_1)$, where $$(h,u_0,u_1)\in C^2([0,T];L^2(\Gamma_1))\cap C([0,T];H^{1/2}(\Gamma_1))\times H^2(\Omega)\cap H_{\Gamma_0}^1(\Omega)\times H_{\Gamma_0}^{1}(\Omega),$$  and $\displaystyle{\left.\frac{\partial u^0 }{\partial n}\right|}_{\Gamma_{1}} = h(0)$. Then, (\ref{w_part_damped}) admits a strong solution 
	$w\in C([0,T];H^2(\Omega)\cap H_{\Gamma_0}^1(\Omega))$ with $w_t\in C([0,T];H_{\Gamma_0}^1(\Omega))$.
\end{lemma}

\begin{proof}
	By setting $z=w_t$ and taking the derivative of the main equation in (\ref{w_part_damped}) with respect to $t$, we find that $z$ solves
	\begin{align}\label{vpartder}
		\left\{
		\begin{array}{ll}
			z_{tt}-\bigtriangleup z+a(x)z_t=-a(x)v_{tt} \quad  \text{in } \Omega\times (0,T),\\
			z(0)=0,\  z_{t}(0)=-a(x)u^1 \quad \text{in } \Omega,\\
			z|_{\Gamma_{0}} =0,   \\
			{\frac{\partial z }{\partial n}|}_{\Gamma_{1}} = 0.
		\end{array}
		\right.
	\end{align} Since  $u^1\in H_{\Gamma_0}^1(\Omega)$ and  $v\in C([0,T];H^2(\Gamma_1)\cap H_{\Gamma_0}^1(\Omega))$ (see Lemma \ref{extlemreg}), we have $-a(x)v_{tt}\in L^1((0,T);L^2(\Omega))$ and $-a(x)u^1 \in L^2(\Omega)$. As in Lemma \ref{reg_damped_w}, we obtain $z\in C([0,T];H_{\Gamma_0}^1(\Omega))$ with $z_t\in C([0,T];L^2(\Omega))$. Using \eqref{w_part_damped} we can write the following  elliptic problem:
	\begin{align}\label{vpartnew}
		\left\{
		\begin{array}{ll}
			\bigtriangleup w = z_t+a(x)z+a(x)v_t \in L^2(\Omega),\\
			w|_{\Gamma_{0}} =0,   \\
			{\frac{\partial w }{\partial n}|}_{\Gamma_{1}} = 0.
		\end{array}
		\right.
	\end{align} By the classical regularity of the Dirichlet-Neumann problem for elliptic equations, we obtain 
	$w\in C([0,T];H^2(\Omega))$.
\end{proof}
\subsection{Existence and uniqueness}
We define the weak solution of the initial-boundary value problem for the linear damped wave equation as follows:
\begin{definition}
	We say $u=w+v$ is a weak solution of (\ref{damped}),(\ref{neumann-bc})-(\ref{dirichlet-bc}), (\ref{initial-cond})  if $w$ is a (weak) solution of (\ref{w_part_damped}) and $v$ is a (weak) solution of (\ref{vpart}).
\end{definition}
We deduce the following theorem from above.
\begin{theorem}
	Let $u^0\in H_{\Gamma_0}^1(\Omega)$, $u^1\in L^2(\Omega)$, and $h\in C^1([0,\infty);L^2(\Gamma_1))$, then the linear damped model given by (\ref{damped}),(\ref{neumann-bc})-(\ref{dirichlet-bc}), (\ref{initial-cond})     has a unique weak solution.
\end{theorem}

\subsection{Uniform Stabilization}
In this section, we impose assumption (\ref{a(x)}) on the damping coefficient. We prove the following theorem which provides a (temporal) upper bound for the energy of solutions from which one can also deduce uniform decay rate estimates under certain assumptions on the behavior of Neumann manipulation.
\begin{theorem}\label{TemamMainThm}Let  $(u^0,u^1)\in H_{\Gamma_0}^1(\Omega)\times L^2(\Omega)$, $h\in C^1([0,\infty);L^2(\Gamma_1))$, $\alpha\in (0,\frac{\epsilon_0}{2})$ with $\epsilon_0$ defined by \eqref{epszero}, $\displaystyle \delta\in (0,\frac{1}{2c_p})$, and	\begin{align*}
		{H}(t)=&\|u^0\|_{H_{\Gamma_0}^1(\Omega)}^{2}+\|u^1\|_{L^2(\Omega)}^{2}+\int_0^t e^{2\alpha s} \bigg( \|h(s) \|_{L^2(\Gamma_{1})}^2+\|h_t(s) \|_{L^2(\Gamma_{1})}^2 \bigg)ds\\
		&+e^{2\alpha t}\|h \|_{L^2(\Gamma_{1})}^2
		+ \|h(0) \|_{L^2(\Gamma_{1})}^2+\int_0^t \|h(s) \|_{L^2(\Gamma_{1})}^2 ds.
	\end{align*}Then, there exists a positive constant $c=c(\alpha, \delta, \Omega)\ge 0$ such that
	\begin{align}\label{energy_decay-damped_temam}
		\varepsilon_a(t)\leq ce^{-2\alpha t} e^{\delta_1t}{H}(t),\quad t\ge 0.
	\end{align}
	Moreover, the above estimate implies  uniform exponential decay of the energy if for some $\gamma>0$
	\begin{align}\label{h_temam}
		\|h(t) \|_{L^2(\Gamma_{1})}^2,  \|h_t(t) \|_{L^2(\Gamma_{1})}^2\quad =\quad \mathcal{O}\left(e^{-\gamma t}\right).
	\end{align}
\end{theorem}
\begin{proof}The proof for the case $a$ is constant and $h$ is zero can be found in \citep{temam}. The proof presented here extends the arguments in \citep{temam} to incorporate both distributed damping and inhomogeneous external manipulation.
	Let $v=u_t+\epsilon u$ with $\epsilon\in (0, \epsilon_0]$ being fixed, where \begin{equation}\label{epszero}\displaystyle \epsilon_0=\min\left\{\frac{a_{\min}}{4},\frac{a_{\min}}{2a_{\max}^2c_p}\right\}\end{equation} 
	in which $c_p$ is the Poincaré constant satisfying $$\|u\|_{L^2(\Omega)}^2\leq c_p \|\nabla u\|_{L^2(\Omega)}^2.$$ Then, $v$ satisfies the following equation:
	\begin{align}\label{v-eqn_temam}
		v_t+(a(x)-\epsilon)v-\triangle u-\epsilon(a(x)-\epsilon)u=0.
	\end{align}
	Multiplying (\ref{v-eqn_temam}) by $v$ and  integrating on $\Omega$ we obtain
	\begin{align}\label{v_multip_temam}
			\begin{split}
		\frac{1}{2}\frac{d}{dt}\int_\Omega \Big( |v|^2 &+ |\nabla u|^2 \Big)dx+\int_\Omega (a(x)-\epsilon)|v|^2 dx +\epsilon\int_\Omega |\nabla u|^2 \\
		&-\int_\Omega \epsilon(a(x)-\epsilon) uv dx= \int_{\Gamma_1}hv d\Gamma.
		\end{split}
	\end{align}
	\quad\\
	Note that, in view of Cauchy-Schwarz, Poincaré, and Young's inequalities, and using $\displaystyle \epsilon \leq \frac{a_{\min}}{4}$, we have	
	\begin{align}\label{sum_temam}
		\begin{split}
			&\int_\Omega (a(x)-\epsilon)|v|^2 dx +\epsilon\int_\Omega |\nabla u|^2 -\int_\Omega \epsilon(a(x)-\epsilon) uv dx\\
			&\quad \geq(a_{\min}-\epsilon)\int_\Omega |v|^2 dx +\epsilon\int_\Omega |\nabla u|^2-\epsilon(a_{\max}-\epsilon)\int_\Omega uv dx\\
			&\quad \geq(a_{\min}-\epsilon) \|v\|_{L^2(\Omega)}^{2}  + \epsilon \| \nabla u  \|_{L^2(\Omega)}^{2} -\epsilon(a_{\max}-\epsilon)\|u\|_{L^2(\Omega)}\|v\|_{L^2(\Omega)} \\
			&\quad \geq(a_{\min}-\epsilon) \|v\|_{L^2(\Omega)}^{2}  + \epsilon \| \nabla u  \|_{L^2(\Omega)}^{2}-(\epsilon (a_{\max}-\epsilon) \sqrt{c_p})\|\nabla u\|_{L^2(\Omega)}\|v\|_{L^2(\Omega)}\\
			&\quad \geq\frac{3a_{\min}}{4} \|v\|_{L^2(\Omega)}^{2}  + \epsilon \| \nabla u  \|_{L^2(\Omega)}^{2}-(\epsilon (a_{\max}-\epsilon)  \sqrt{c_p})\|\nabla u\|_{L^2(\Omega)}\|v\|_{L^2(\Omega)}\\
			&\quad \geq\frac{3a_{\min}}{4} \|v\|_{L^2(\Omega)}^{2}  + \epsilon \| \nabla u  \|_{L^2(\Omega)}^{2}-(\epsilon a_{\max} \sqrt{c_p})\|\nabla u\|_{L^2(\Omega)}\|v\|_{L^2(\Omega)}  \\
			&\quad \geq  \Big( \frac{3a_{\min}}{4} -\delta (\epsilon a_{\max}  \sqrt{c_p})\Big) \|v\|_{L^2(\Omega)}^{2}  +\Big( \epsilon - \frac{1}{ 4\delta}(\epsilon a_{\max}  \sqrt{c_p})\Big)\| \nabla u  \|_{L^2(\Omega)}^{2}\\ 
			&\quad \geq \frac{a_{\min}}{2}\|v\|_{L^2(\Omega)}^{2}  + \frac{\epsilon}{2}\| \nabla u  \|_{L^2(\Omega)}^{2}\ \ \ \  (\delta=\frac{a_{\min}}{4\epsilon  a_{\max}\sqrt{c_p}},\ \ \ \epsilon \leq \frac{a_{\min}}{2a_{\max}^2c_p}) \\
			&\quad \geq \frac{\epsilon}{2} \Big(\|v\|_{L^2(\Omega)}^{2}+\| \nabla u  \|_{L^2(\Omega)}^{2}  \Big).
		\end{split}
	\end{align}
	
	\noindent 	Therefore (\ref{v_multip_temam}) and (\ref{sum_temam})  yield
	\begin{align}\label{v_trace}
		\begin{split}
			&\frac{1}{2}\frac{d}{dt}\int_\Omega \Big( |v|^2 + |\nabla u|^2 \Big)dx+\frac{\epsilon}{2} \Big(\|v\|_{L^2(\Omega)}^{2}+\| \nabla u  \|_{L^2(\Omega)}^{2}  \Big)
			\leq  \int_{\Gamma_1}hvd\Gamma.\\
		\end{split}
	\end{align}
	We estimate the term at the right hand side of \eqref{v_trace} by using Young's, trace and Poincaré's inequalities:
	\begin{align}\label{rewrited_term}
		\begin{split}
			\int_{\Gamma_1}hvd\Gamma &=\int_{\Gamma_1}h(u_t+\epsilon u)d\Gamma\\
			&=\frac{d}{dt}\int_{\Gamma_1}h ud\Gamma-\int_{\Gamma_1}h_t u d\Gamma + \epsilon\int_{\Gamma_1}h  u d\Gamma \\
			&\leq \frac{d}{dt}\int_{\Gamma_1}h ud\Gamma+\eta c_p(1+\epsilon)\|\nabla u \|_{L^2(\Omega)}^{2} +\frac{\epsilon}{4\eta}\|h \|_{L^2(\Gamma_{1})}^2+\frac{1}{4\eta}\|h_t \|_{L^2(\Gamma_{1})}^2.
		\end{split}
	\end{align}
	Next, we set $$y(t)=\int_\Omega \Big( |v|^2 + |\nabla u|^2 \Big)dx.$$ Therefore, \eqref{v_trace} becomes
	\begin{align}\label{yield_temam}
		\begin{split}
			\frac{1}{2}\frac{d}{dt}y(t)+\alpha y(t)\leq  \frac{d}{dt}\int_{\Gamma_1}h ud\Gamma +\frac{1}{4\eta}
			\bigg(\epsilon \|h \|_{L^2(\Gamma_{1})}^2+\|h_t \|_{L^2(\Gamma_{1})}^2 \bigg)
		\end{split}
	\end{align}
	where $\displaystyle \alpha:= \frac{\epsilon}{2}- \eta c_p(1+\epsilon)$ and $\displaystyle \eta\in (0,\frac{\epsilon}{2c_p(1+\epsilon)})$. Then, multiplying  (\ref{yield_temam}) by $e^{2\alpha t}$  and  integrating on $(0,t)$ we obtain
	\begin{align}\label{y_ineq}
		\begin{split}
		e^{2\alpha t}y(t)\leq y(0)&+2\int_0^t e^{2\alpha s} \frac{d}{ds}\int_{\Gamma_1}h ud\Gamma ds  \\
		&+\frac{1}{2\eta}\int_0^t e^{2\alpha s} \bigg(\epsilon \|h(s) \|_{L^2(\Gamma_{1})}^2+\|h_t(s) \|_{L^2(\Gamma_{1})}^2 \bigg)ds.
		\end{split}
	\end{align}
	Similarly, we can rewrite and estimate the second term at the right hand side of \eqref{y_ineq}  as follows
	\begin{align}
		\begin{split}
			&2\int_0^t e^{2\alpha s}\frac{d}{ds}\int_{\Gamma_1}h ud\Gamma ds \\
			&\quad =2\int_0^t \frac{d}{ds}\bigg(e^{2\alpha s}\int_{\Gamma_1}h ud\Gamma\bigg)ds-2\int_0^t \frac{d}{ds}\bigg(e^{2\alpha s}\bigg)\int_{\Gamma_1}h ud\Gamma ds \\
			&\quad= 2e^{2\alpha t}\int_{\Gamma_1}h ud\Gamma- 2\int_{\Gamma_1}h(0)u(0)d\Gamma-4\alpha\int_0^t e^{2\alpha s}\int_{\Gamma_1}h ud\Gamma ds \\
			&\quad \leq \frac{1}{2\delta}e^{2\alpha t}\|h \|_{L^2(\Gamma_{1})}^2+ 2\delta c_p e^{2\alpha t}\|\nabla u \|_{L^2(\Omega)}^{2}+c_p\|\nabla u^0\| _{L^2(\Omega)}^{2}+ \|h(0) \|_{L^2(\Gamma_{1})}^2\\
			&\quad \quad +\frac{\alpha}{\delta}\int_0^t \|h(s) \|_{L^2(\Gamma_{1})}^2 ds+ 4 \delta \alpha c_p \int_0^t e^{2\alpha s}\|\nabla u(s) \|_{L^2(\Omega)}^{2}ds.
		\end{split}
	\end{align}
	We find that
	\begin{align}\label{y_ineq_2}
		(1-2\delta c_p) e^{2\alpha t}y(t)\leq 4 \delta \alpha c_p \int_0^t e^{2\alpha s} y(s) ds+ \tilde{H}(t),
	\end{align}
	where $\delta\in (0,\frac{1}{2c_p})$ is arbitrary and 
	\begin{align*}
		\tilde{H}(t)&=y(0) +\frac{1}{2\eta}\int_0^t e^{2\alpha s} \bigg(\epsilon \|h(s) \|_{L^2(\Gamma_{1})}^2+\|h_t(s) \|_{L^2(\Gamma_{1})}^2 \bigg)ds
		+\frac{1}{2\delta}e^{2\alpha t}\|h \|_{L^2(\Gamma_{1})}^2\\
		&\quad+c_p\|\nabla u^0\| _{L^2(\Omega)}^{2}+ \|h(0) \|_{L^2(\Gamma_{1})}^2+\frac{\alpha}{\delta}\int_0^t \|h(s) \|_{L^2(\Gamma_{1})}^2 ds.
	\end{align*}
	
	\noindent Note that 
	\begin{align*}
		y(t)&=\|u_t+\epsilon u\|_{L^2(\Omega)}^2+\|\nabla u\|_{L^2(\Omega)}^2\\
		&=\|u_t\|_{L^2(\Omega)}^2+\epsilon^2\|u\|_{L^2(\Omega)}^2+2\epsilon\int_\Omega u_t udx+\|\nabla u\|_{L^2(\Omega)}^2\\
		&\ge (1-\rho)\|u_t\|_{L^2(\Omega)}^2+\epsilon^2(1-\frac{1}{\rho})\|u\|_{L^2(\Omega)}^2+\|\nabla u\|_{L^2(\Omega)}^2\\
		&\ge  (1-\rho)\|u_t\|_{L^2(\Omega)}^2+\left(\epsilon^2(1-\frac{1}{\rho})+\frac{1}{c_p}\right)\|u\|_{L^2(\Omega)}^2\\
		&\ge (1-\rho)\|u_t\|_{L^2(\Omega)}^2
	\end{align*} where $\displaystyle \rho\in (\frac{\epsilon^2c_p}{1+\epsilon^2c_p},1)$ is arbitrary. Considering definition of $y(t)$, we also have $$y(t)\ge \|\nabla u\|_{L^2(\Omega)}^2\ge (1-\rho)\|\nabla u\|_{L^2(\Omega)}^2.$$ Therefore, we deduce that $y(t)\ge (1-\rho)\varepsilon_a(t)$.  Thus, Gronwall's inequality applied to \eqref{y_ineq_2} yields
	\begin{align}
		\frac{1}{1-\rho}\varepsilon_a(t)\leq y(t)\leq \frac{1}{1-2\delta c_p}e^{-2\alpha t} e^{\delta_1t}\tilde{H}(t)
	\end{align}
	where $\displaystyle \delta_1:= \frac{4 \delta \alpha c_p}{1-2\delta c_p}.$ The last statement of the lemma follows from L'Hospital's rule.
\end{proof}

\section{Viscoelastic wave equation}
\subsection{Homogenization}
In this section, we study the initial-boundary value problem for the linear viscoelastic wave equation given by \eqref{lineq}, \eqref{neumann-bc}-\eqref{dirichlet-bc}, and \eqref{initial-cond}.
\label{linw} Setting $w\equiv u-v=u-E(h,u_0,u_1)$, we observe that solving \eqref{lineq}, \eqref{neumann-bc}-\eqref{dirichlet-bc}, \eqref{initial-cond} reduces to finding a solution to the following initial boundary value problem with zero initial data and homogeneous Dirichlet-Neumann boundary conditions:
\begin{align}\label{wpart}
	\left\{
	\begin{array}{ll}
		w_{tt}-\bigtriangleup w &+ \int_{0}^{t}g(t-s)\bigtriangleup (w+v)(s)ds=0\quad  \text{in } \Omega\times (0,T),  \\
		w(x,0)&=0,\  w_{t}(x,0)=0\quad \text{in } \Omega,   \\
		w|_{\Gamma_{0}} &=0,   \\
		{\frac{\partial w   }{\partial n}}\Big|_{\Gamma_{1}} &= 0.
	\end{array}
	\right.
\end{align}
We define weak solutions for \eqref{wpart} as follows:

\begin{definition}We say a function $w\in L^\infty(0,T;H_{\Gamma_0}^1(\Omega))$ with $w_t\in L^\infty(0,T;L^2(\Omega))$ and $w_{tt}\in L^\infty(0,T;\left(H_{\Gamma_0}^1(\Omega)\right)')$ is called a weak solution of \eqref{wpart} provided
	\begin{align}\label{ap1new2}
		\begin{split}
			< w_{tt},\phi>& +(\nabla w, \nabla \phi)_{\Omega}-\int_0^t g(t-s)(\nabla w, \nabla \phi)_{\Omega}ds\\
			&=\int_0^t g(t-s)(\nabla v, \nabla \phi)_{\Omega}ds-\int_0^t g(t-s)\langle h, \phi\rangle_{\Gamma_1}ds
		\end{split}
	\end{align}for all $\phi\in H_{\Gamma_0}^1(\Omega)$ and $(w(0),w_t(0))\equiv (0,0)$.
\end{definition}

\subsubsection*{Approximate Solutions}
We begin by constructing approximate solutions of \eqref{wpart} using the Faedo–Galerkin method. To this end, let $\{e_k\}$ be a complete orthonormal basis in $L^2(\Omega)$ of eigenfunctions of Laplacian with zero mixed boundary conditions:
\begin{align}\label{basis-eq}
	-\Delta e_k &= \lambda_ke_k\quad \text{ in }\Omega, \\
	e_k&=0\quad \text{ on }\Gamma_0,\\
	{\frac{\partial e_k   }{\partial n}}&=0\quad \text{ on }\Gamma_1.
\end{align}
We now consider the following initial value problem for the system of ordinary differential equations, with unknowns $w_{N,k}$:
\begin{align}\label{ap1}
	\begin{split}
		(w_{Ntt},e_k)_{\Omega}-(\Delta w_N, e_k)_{\Omega}+\int_0^t g(t-s)(\Delta w_N, e_k)_{\Omega}ds=f_{N,k}, k=1,...,N,
	\end{split}
\end{align}

\begin{align}\label{ap2}
	w_{N,k}(0)=w'_{N,k}(0)=0,
\end{align} where
$w_{N}(t)\equiv \sum_{k=1}^N w_{N,k}(t)e_k$ and
\begin{align}
	\begin{split}
		f_{N,k}(t)&=-\int_0^t g(t-s)(\Delta v_N, e_k)_{\Omega}ds\\
		&=\int_0^t g(t-s)(\nabla v_N, \nabla e_k)_{\Omega}ds-\int_0^t g(t-s)\langle h_N, e_k\rangle_{\Gamma_1}ds.
	\end{split}
\end{align}
In this formulation, $v_N$ denotes a smooth approximation of $v$, constructed by applying the strong regularity result from Lemma \ref{extlemreg}. More precisely, we replace the boundary datum $h\in C^1([0,T];L^2(\Gamma_1))$ by $h_N\in C^2([0,T];L^2(\Gamma_1))\cap C([0,T];H^{1/2}(\Gamma_1))$ and the initial data
$(u^0,u^1)\in H_{\Gamma_0}^1(\Omega)\times L^2(\Omega)$ with $(u_{0N},u_{1N})\in H^2(\Omega)\cap H_{\Gamma_0}^1(\Omega)\times H_{\Gamma_0}^1(\Omega)$ satisfying the additional compatibility condition $\displaystyle{\left.\frac{\partial u_{N0} }{\partial n}\right|}_{\Gamma_{1}} = h_N(0)$. These, approximations are constructed such that $h_N\rightarrow h$ in $C^1([0,T];L^2(\Gamma_1))$ and $(u_{0N},u_{1N})\rightarrow (u^0,u^1)$ in $H_{\Gamma_0}^1(\Omega)\times L^2(\Omega)$.   Then, we solve \eqref{vpart} with data $(h_N,u_{N0},u_{N1})$, obtaining $v_N\in C([0,T];H^2(\Omega))$ with $v_{Nt}\in C([0,T];H_{\Gamma_0}^1(\Omega))$.\\
\eqref{ap1}-\eqref{ap2} can also be rewritten using the properties of $e_k$ as:
\begin{align}\label{apode}
	\begin{split}
		&w''_{N,k}+\lambda_k w_{N,k}-\lambda_k \int_0^t g(t-s)w_{N,k}ds= f_{N,k},\\
		&w_{N,k}=w'_{N,k}=0, k=1,...,N.
	\end{split}
\end{align}

By the theory of systems of ordinary differential equations, there is a (continuous in time) solution $w_N$, defined on an interval $[0,t_N)$, for each $N$.   These will be referred to as the approximate solutions.

\subsubsection*{A priori estimates}
Here, we derive a priori estimates for $w_N$ that are uniform with respect to $N$ and $t$. Multiplying \eqref{ap1} by $\partial_tw_{N,k}$ and summing over $k$ from $1$ to $N$, we obtain
\begin{align}\label{bypartwN}
	\begin{split}
		&\frac{1}{2}\frac{d}{dt}\int_{\Omega} |w_{Nt}(t)|^{2} dx + \frac{1}{2} \frac{d}{dt}\int_{\Omega}|\nabla w_N(t)|^2dx\\
		&-\int_{0}^{t}g(t-s)\int_{\Omega}\nabla w_N(s)\nabla w_{Nt}(t)dxds-\int_{0}^{t}g(t-s)\int_{\Omega}\nabla v_N(s)\nabla w_{Nt}(t)dxds\\
		&+\int_{0}^{t}g(t-s)\int_{\Gamma_1}h_N(s)w_{Nt}(t)d\Gamma ds =0.
	\end{split}
\end{align}

\noindent We rewrite third, fourth and fifth terms of (\ref{bypartwN}) as follows:

\begin{align}\label{est1wN}
	\begin{split}
		&\int_{0}^{t}g(t-s)\int_{\Omega}\nabla w_N(s)\nabla w_{Nt}(t)dxds\\
		=&-\frac{1}{2}\frac{d}{dt}\int_{0}^{t} g(t-s)\int_{\Omega}(\nabla w_N(s)-\nabla w_N(t))^2dx ds\\
		&+\frac{1}{2}\int_{0}^{t}g'(t-s)\int_{\Omega}(\nabla w_N(s)-\nabla w_N(t))^2dxds\\
		&+\frac{1}{2}\frac{d}{dt}\int_{0}^{t}g(s)\int_{\Omega}|\nabla w_N(t)|^2dxds-\frac{1}{2}g(t)\int_{\Omega}|\nabla w_N(t)|^2dx,
	\end{split}
\end{align}

\begin{align}\label{est3wN}
	\begin{split}
		&\int_{0}^{t}g(t-s)\int_{\Omega}\nabla v_N(s)\nabla w_{Nt}(t)dxds\\
		=&-\frac{1}{2}\frac{d}{dt}\int_{0}^{t} g(t-s)\int_{\Omega}(\nabla v_N(s)-\nabla w_N(t))^2dx ds\\
		&-g(0)\int_{\Omega}(\nabla v_N(t)-\nabla w_N(t))^2dx\\
		&+\frac{1}{2}\int_{0}^{t}g'(t-s)\int_{\Omega}(\nabla v_N(s)-\nabla w_N(t))^2dxds\\
		&+\frac{1}{2}\frac{d}{dt}\int_{0}^{t}g(s)\int_{\Omega}|\nabla w_N(t)|^2dxds-\frac{1}{2}g(t)\int_{\Omega}|\nabla w_N(t)|^2dx,
	\end{split}
\end{align}

\begin{align}\label{est2wN}
	\begin{split}
		\int_{0}^{t}&g(t-s)\int_{\Gamma_1}h_N(s)w_{Nt}(t)d\Gamma ds=\frac{d}{dt}\int_{0}^{t}g(t-s)\int_{\Gamma_1} h_N(s)w_N(t)d\Gamma ds\\
		&-g(0)\int_{\Gamma_1} h_N(t)w_N(t)d\Gamma
		-\int_{0}^{t} g'(t-s)\int_{\Gamma_1} h_N(s)w_N(t)d\Gamma ds.
	\end{split}
\end{align}

Using identities (\ref{est1wN})- (\ref{est2wN}) in (\ref{bypartwN}) and integrating over $(0,t)$, we obtain

\begin{align}
	\begin{split}
		&\frac{1}{2}\int_{\Omega} |w_{Nt}(t)|^{2} dx + \frac{1}{2}\bigg(1-2\int_{0}^{t}g(s)ds \bigg)\int_{\Omega}|\nabla w_N(t)|^2dx\\
		&+\int_{0}^{t}g(t-s)\int_{\Gamma_1}w_N(t)h_N(s)d\Gamma ds-\int_{0}^{t}\int_{0}^{s}g'(s-\tau)\int_{\Gamma_1}w_N(s)h_N(\tau)d\Gamma d\tau ds\\
		&-g(0)\int_{0}^{t}\int_{\Gamma_1}w_N(s)h_N(s)d\Gamma ds
		+\frac{1}{2} \bigg((g\circ \nabla w_N)(t)\\
		&+\frac{1}{2}\int_{0}^{t} g(t-s)\int_{\Omega}(\nabla v_N(s)-\nabla w_N(t))^2dx ds\\
		&- \frac{1}{2}\int_{0}^{t}\int_{0}^{s}g'(s-\tau) \int_{\Omega}|\nabla v_N(\tau)-\nabla w_N(s)|^2 dx d\tau ds\\
		&+\int_{0}^{t}g(s)\int_\Omega|\nabla w_N(s)|^2dxds+ g(0)\int_{0}^{t}\int_{\Omega}|\nabla v_N(s)-\nabla w_N(s)|^2 dxds\\
		& -\frac{1}{2}\int_{0}^{t}(g'\circ \nabla w_N)(s)ds \bigg)=0.
	\end{split}
\end{align}

Note that some of the terms at the left hand side of above identity are nonnegative. Dropping those terms, we get the estimate 
\begin{align}\label{becomeswN}
	\begin{split}
		\frac{1}{2}\int_{\Omega}& |w_{Nt}(t)|^{2} dx + \frac{1}{2}\bigg(1-2\int_{0}^{t}g(s)ds \bigg)\int_{\Omega}|\nabla w_N(t)|^2dx\\
		 \leq & -\int_{0}^{t}g(t-s)\int_{\Gamma_1}w_N(t)h_N(s)d\Gamma ds\\
		&+\int_{0}^{t}\int_{0}^{s}g'(s-\tau)\int_{\Gamma_1}w_N(s)h_N(\tau)d\Gamma d\tau ds+g(0)\int_{0}^{t}\int_{\Gamma_1}w_N(s)h_N(s)d\Gamma ds.
	\end{split}
\end{align}
Then, applying $\epsilon-$Young's, trace and Poincaré inequalities together with the fact that $g'$ has a negative sign, (\ref{becomeswN}) becomes
\begin{align}\label{gwN}
	\begin{split}
		\frac{1}{2}&\|w_{Nt}(t)\|_{L^2(\Omega)}^{2} + \frac{1}{2}\bigg(1-2\int_{0}^{t}g(s)ds-2\varepsilon c\int_{0}^{t}g(s)ds\bigg)\|\nabla w_N(t)\|_{L^2(\Omega)}^2   \\
		&\leq \frac{1}{4\varepsilon} \int_{0}^{t}g(t-s) \Big \| h_N(s)\Big\|_{L^2(\Gamma_{1})}^2 ds-\frac{1}{4\varepsilon}\int_{0}^{t}\int_{0}^{s}g'(s-\tau) \| h_N(\tau)\|_{L^2(\Gamma_{1})}^2d\tau ds \\&+ \frac{1}{4\varepsilon}g(0) \int_{0}^{t}\| h_N(s)\|_{L^2(\Gamma_{1})}^2 ds+\varepsilon c\int_{0}^{t}g(s)\|\nabla w_N(s)\|_{L^2(\Omega)}^2 ds\\
		&+ \varepsilon cg(0) \int_{0}^{t}\| \nabla w_N(s)\|_{L^2(\Omega)}^2 ds.
	\end{split}
\end{align}

\noindent We set $\mathcal{E}_{N}(t)\equiv \|w_{Nt}(t)\|_{L^2(\Omega)}^{2} + \|\nabla w_N(t)\|_{L^2(\Omega)}^2$ and
\begin{align*}
	H_N(t)\equiv & \frac{1}{4\varepsilon} \int_{0}^{t}g(t-s) \Big \| h_N(s)\Big\|_{L^2(\Gamma_{1})}^2 ds-\frac{1}{4\varepsilon}\int_{0}^{t}\int_{0}^{s}g'(s-\tau) \| h_N(\tau)\|_{L^2(\Gamma_{1})}^2d\tau ds\\
	+ &  \frac{1}{4\varepsilon}g(0) \int_{0}^{t}\| h_N(s)\|_{L^2(\Gamma_{1})}^2 ds.
\end{align*}
Now, taking $\varepsilon>0$ small enough and using $g(s)<g(0)$, we get the estimate

$$\mathcal{E}_N(t)\le cH_N(t)+c\int_{0}^{t}\| \nabla w_N(s)\|_{L^2(\Omega)}^2 ds\le cH_N(t)+c\int_{0}^{t}\mathcal{E}_N(t) ds,$$ where $c$ is a positive constant that depends on $\varepsilon, g(0),$ and  $\Omega$. Then, by using Gronwall's inequality we obtain a uniform in $t$ estimate over the interval $[0,t_N)$ given by
\begin{equation}\label{aprioriest}
	\mathcal{E}_N(t)\le cH_N(t_N)\exp(ct_N)\text{ for } t\in [0,t_N).
\end{equation}

\subsubsection*{Weak solution}
The first consequence of the estimate \eqref{aprioriest} is that we can extend $w_N$ globally in time for each $N$. Moreover, the sequence $\{w_N\}$ is bounded in $L^\infty(0,T;H_{\Gamma_0}^1(\Omega))$ and $\{w_{Nt}\}$ is bounded in $L^\infty(0,T;L_2(\Omega))$ for any $T>0$. Next, using \eqref{ap1}, \eqref{aprioriest} and the strong convergence properties of $v_N$ and $h_N$, we get
\begin{align*}
	&\|w_{Ntt}(t)\|_{\left(H_{\Gamma_0}^1(\Omega)\right)'}\equiv \sup_{\phi\neq 0}\frac{|(w_{Ntt},\phi)_{\Omega}|}{\|\phi\|_{H^1_{\Gamma_0}(\Omega)}^2}\\
	&\le \|w_N\|_{H_{\Gamma_0}^1(\Omega)} + \int_0^t g(t-s)\Big(\| w_N\|_{H_{\Gamma_0}^1(\Omega)}+\| v_N\|_{H_{\Gamma_0}^1(\Omega)}  +  \|h_N\|_{L^2(\Gamma_1)}\Big)ds\le c_T
\end{align*} for $t\in [0,T]$. Hence, $\{w_{Ntt}\}$ is a bounded sequence in $L^\infty(0,T;\left(H_{\Gamma_0}^1(\Omega)\right)').$
It follows from above boundedness arguments there exists a subsequence of $\{w_N\}$ (still denoted by $\{w_N\}$) and there exists $w\in L^\infty([0,T];H_{\Gamma_0}^1(\Omega))$ with derivatives $w_t\in L^\infty([0,T];L_2(\Omega))$ and $w_{tt}\in L^\infty([0,T];(H_{\Gamma_0}^1(\Omega))')$ such that
\begin{equation*}
	\left\{
	\begin{array}{lll}
		w_N \longrightarrow w & \mbox{weakly-* in}  & L^\infty([0,T];H_{\Gamma_0}^1(\Omega)),\\
		w_{Nt}\longrightarrow w_t &  \mbox{weakly-* in} &  L^\infty([0,T];L^2(\Omega)),\\
		w_{Ntt}\longrightarrow w_{tt} &   \mbox{weakly-* in} & L^\infty([0,T];(H_{\Gamma_0}^1(\Omega))').
	\end{array}\right.
\end{equation*}

\noindent From \eqref{ap1} and aforementioned convergence properties, we derive
\begin{align}\label{ap1new}
	\begin{split}
		&\int_{0}^{T}\left(<w_{tt},\phi>+(\nabla w, \nabla \phi)_{\Omega}-\int_0^t g(t-s)(\nabla w, \nabla \phi)_{\Omega}ds\right)dt\\
		&=\int_{0}^{T}\left(\int_0^t g(t-s)(\nabla v, \nabla \phi)_{\Omega}ds-\int_0^t g(t-s)\langle h, \phi\rangle_{\Gamma_1}ds\right)dt
	\end{split}
\end{align} for all $\phi\in L^2(0,T;H_{\Gamma_0}^1(\Omega))$. In particular, the variational formulation \eqref{ap1new2} in the definition of the weak solution for \eqref{wpart} holds true. From \eqref{ap1new}, we also have

\begin{align}\label{ap1newa}
	\begin{split}
		&\int_{0}^{T}\left(-<w_{t},\phi_t>+(\nabla w, \nabla \phi)_{\Omega}-\int_0^t g(t-s)(\nabla w, \nabla \phi)_{\Omega}ds\right)dt\\
		&=\int_{0}^{T}\left(\int_0^t g(t-s)(\nabla v, \nabla \phi)_{\Omega}ds-\int_0^t g(t-s)\langle h, \phi\rangle_{\Gamma_1}ds\right)dt+(w_t(0),\phi(0))
	\end{split}
\end{align} for $\phi\in C^1([0,T];H_{\Gamma_0}^1(\Omega))$ with $\phi(T)=0.$ On the other hand, from \eqref{ap1} and $w_{Nt}(0)\equiv 0$, we obtain
\begin{align}\label{ap1-intt}
	\begin{split}
		\int_0^T\left(-(w_{Nt},\phi_t)_{\Omega}+(\nabla w_N, \nabla \phi)_{\Omega}-\int_0^t g(t-s)(\nabla w_N, \nabla\phi)_{\Omega}ds\right)dt\\
		=\int_{0}^{T}\left(\int_0^t g(t-s)(\nabla v_N, \phi)_{\Omega}ds-\int_0^t g(t-s)\langle h_N, \phi\rangle_{\Gamma_1}ds\right).
	\end{split}
\end{align} Passing to limit in $N$, we get
\begin{align}\label{ap1-intt2}
	\begin{split}
		\int_0^T\left(-(w_{t},\phi_t)_{\Omega}+(\nabla w, \nabla\phi)_{\Omega}-\int_0^t g(t-s)(\nabla w, \nabla\phi)_{\Omega}ds\right)dt\\
		=\int_{0}^{T}\left(\int_0^t g(t-s)(\nabla v, \phi)_{\Omega}ds-\int_0^t g(t-s)\langle h, \phi\rangle_{\Gamma_1}ds\right).
	\end{split}
\end{align} Comparing \eqref{ap1-intt2} and \eqref{ap1newa} and the fact that $\phi$ is arbitrary, we conclude that $w_t(0)\equiv 0.$ A similar argument, after integrating by parts in $t$ again, shows that $w(0)\equiv 0.$ Hence, we just proved that $(w(0),w_t(0))\equiv (0,0)$, ensuring that the initial conditions are satisfied. We just established the existence of a weak solution for $w$-problem. The uniqueness of weak solution for $w$-problem is easy to see.  Indeed, it is enough to show that the only weak solution of $w$ problem with $v\equiv 0$ and $h\equiv 0$ is the zero solution. But when  $v\equiv 0$ and $h\equiv 0$, one has the estimate
$$\|w_{t}(t)\|_{L^2(\Omega)}^{2} + \|\nabla w(t)\|_{L^2(\Omega)}^2\le 0,$$ which implies $w\equiv 0$ as the initial data are zero.
\subsection{Existence and uniqueness}
We define the weak solution of the initial-boundary value problem for the linear viscoelastic wave equation given by \eqref{lineq}, \eqref{neumann-bc}-\eqref{dirichlet-bc}, and \eqref{initial-cond} as follows:
\begin{definition}
	We say $u=w+v$ is called a weak solution of \eqref{lineq}, \eqref{neumann-bc}-\eqref{dirichlet-bc}, \eqref{initial-cond}  if $w$ is a (weak) solution of (\ref{wpart}) and $v$ is a solution of (\ref{vpart}).
\end{definition}
The results of Section \ref{linv} and Section \ref{linw} yield the following theorem.
\begin{theorem}[Existence and uniqueness for the linear model]
	Let $u^0\in H_{\Gamma_0}^1(\Omega)$, $u^1\in L^2(\Omega)$, and $h\in C^1([0,\infty);L^2(\Gamma_1))$, then the linear viscoelastic model given by \eqref{lineq}, \eqref{neumann-bc}-\eqref{dirichlet-bc}, \eqref{initial-cond} has a unique weak solution.
\end{theorem}
\subsection{Uniform stabilization}
In this section, we establish uniform decay rate estimates for the solutions of the viscoelastic wave equation with external manipulation by using the energy method. Our calculations are only given formally for simplicity, but all steps can be justified by first working with the approximate solutions constructed earlier and then passing to the limit.

\subsubsection*{Modified energy}
The following lemma analyzes the rate of change of the modified energy for a weak solution.
\begin{lemma} If $u$ is a weak solution of \eqref{lineq}, \eqref{neumann-bc}-\eqref{dirichlet-bc}, \eqref{initial-cond}, then its modified energy satisfies the estimate
	\begin{align}\label{e3}
		\begin{split}
			&\varepsilon'_g (t) \leq \frac{1}{2}(g'\circ \nabla u)(t)-\frac{1}{2}g(t)\|\nabla u(t)\|_{L^2(\Omega)}^2  + \epsilon c(1+2 g(0))\|\nabla u(t)\|_{L^2(\Omega)}^2\\&\, +\frac{1}{4\epsilon}g(0)\| h(t)\|_{L^2(\Gamma_{1})}^2 -\frac{1}{4\epsilon}\int_0^t g'(t-s)\|h(s)\|_{L^2(\Gamma_{1})}^2 ds\\ &\,  +\frac{1}{4\epsilon}\|h_{t}(t)\|_{L^2(\Gamma_{1})}^2+\frac{d}{dt}\bigg( \int_{\Gamma_{1}}u(t)h(t)d\Gamma-\int_{\Gamma_{1}}u(t)\int_{0}^{t}g(t-s)h(s)dsd\Gamma \bigg )
		\end{split}
	\end{align} for $\epsilon>0$.
\end{lemma}

\begin{proof}
	Using the definition of $\varepsilon_g$ in \eqref{defmod}, differentiating with respect to $t$, using the main equation, and applying divergence theorem in $x$, we obtain \eqref{e1}. Note that third and last terms at the right hand side of \eqref{e1} can be rewritten as
	\begin{align}\label{uthgamma1}
		\int_{\Gamma_{1}}u_t(t)h(t)d\Gamma=\frac{d}{dt}\int_{\Gamma_{1}}u(t)h(t)d\Gamma-\int_{\Gamma_{1}}u(t)h_t(t)d\Gamma
	\end{align}
	and
	\begin{align}\label{guthgamma1}
		\begin{split}
			\int_{0}^{t} & g(t-s)\int_{\Gamma_{1}}u_t(t)h(s)d\Gamma ds\\
			=&\int_{0}^{t}\frac{d}{dt}\bigg(g(t-s)\int_{\Gamma_{1}}u(t)h(s)d\Gamma\bigg) ds -\int_{0}^{t}g'(t-s)\int_{\Gamma_{1}}u(t)h(s)d\Gamma ds \\
			=& \frac{d}{dt}\int_{0}^{t}g(t-s)\int_{\Gamma_{1}}u(t)h(s)d\Gamma ds-g(0)\int_{\Gamma_{1}}u(t)h(t)d\Gamma\\
			& -\int_{0}^{t}g'(t-s)\int_{\Gamma_{1}}u(t)h(s)d\Gamma ds.
		\end{split}
	\end{align}
	Using $\epsilon$-Young's, trace and Poincaré inequalities, second term at the right hand side of \eqref{uthgamma1} can be estimated as
	\begin{align}\label{com1}
		\int_{\Gamma_{1}}u(t)h_{t}(t)d\Gamma \leq \epsilon c \|\nabla u(t)\|_{L^2(\Omega)}^2 + \frac{1}{4\epsilon}\|h_t(t)\|_{L^2(\Gamma_1)}^2.
	\end{align}
	Similarly, second and third terms at the right hand side of \eqref{guthgamma1} are estimated as
	\begin{align}\label{com2}
		g(0)\int_{\Gamma_{1}}u(t)h(t)d\Gamma \leq \epsilon c g(0) \|\nabla u(t)\|_{L^2(\Omega)}^2 + \frac{1}{4\epsilon}g(0)\|h(t)\|_{L^2(\Gamma_1)}^2
	\end{align}
	and
	\begin{align}\label{com3}
		\begin{split}
			&\int_{\Gamma_{1}}u(t)\int_{0}^{t}g'(t-s)h(s)dsd\Gamma \\
			&\,\leq -\epsilon c \bigg(\int_0^tg'(t-s)ds\bigg)\|\nabla u(t)\|_{L^2(\Omega)}^2 -\frac{1}{4\epsilon}\int_0^tg'(t-s)\|h(s)\|_{L^2(\Gamma_1)}^2ds\\
			&\,\leq \epsilon cg(0) \|\nabla u(t)\|_{L^2(\Omega)}^2-\frac{1}{4\epsilon}\int_0^tg'(t-s)\|h(s)\|_{L^2(\Gamma_1)}^2ds.
		\end{split}
	\end{align}
	The result follows by using \eqref{com1}-\eqref{com3} in \eqref{e1}.
\end{proof}
\subsubsection*{Perturbed energy}
We define (as in \citealp{messaoudi2008}) a perturbed energy functional given by
\begin{align}\label{fdef}
	F(t)\equiv \varepsilon_g (t)+\epsilon_{1}\psi(t)+\epsilon_{2}\chi(t)
\end{align} where $\epsilon_{1}$ and $\epsilon_{2}$ are positive constants and
$$\psi(t)\equiv  \xi(t)\int_{\Omega} u u_{t}dx,$$
$$\chi(t)\equiv -\xi(t)\int_{\Omega}u_{t}\int_{0}^{t}g(t-s)\Big(u(t)-u(s)\Big)dsdx.$$
The following two lemmas are easy to prove.

\begin{lemma}\label{lempoin} If $u \in L^{2}(0,T;H_{\Gamma_0}^1(\Omega))$, then
	\begin{align} \label{tool}
		\int_{\Omega}\bigg(\int_{0}^{t}g(t-s)(u(t)-u(s))ds\bigg)^2 dx \leq c(1-L)(g\circ \nabla u)(t)
	\end{align} and
	\begin{align} \label{toolbdr}
		\int_{\partial \Omega}\bigg(\int_{0}^{t}g(t-s)(u(t)-u(s))ds\bigg)^2 dx \leq c(1-L)(g\circ \nabla u)(t)
	\end{align}
	for a.a. $t\in [0,T]$.
\end{lemma}
\begin{proof}
	Proof of \eqref{tool} is given in \citep{messaoudi2008} when $\Gamma_1=\emptyset$ but the proof easily extends to the case $\Gamma_1\neq \emptyset$ as Poincaré inequality is still valid since $\Gamma_0\neq \emptyset$. \eqref{toolbdr} follows from Cauchy-Schwarz, trace and Poincaré inequalities together with the assumption on the relaxation function.
\end{proof}
\begin{lemma}[\citealp{messaoudi2008}] For $\epsilon_{1},\epsilon_{2}> 0$ small enough, the inequalities
	\begin{align}\label{masu}
		\alpha_{1}F(t)\leq\varepsilon_g(t)\leq \alpha_{2} F(t)
	\end{align}
	hold for some positive constants $\alpha_{1}$ and  $\alpha_{2}$.
\end{lemma}

Next, we prove an estimate on the rate of change of $\psi$.
\begin{lemma} Under assumptions (A1) and (A2) on the relaxation function, $\psi$ satisfies
	\begin{align}\label{pissi}
		\begin{split}
			\psi'(t)&\le -\frac{L(2-L)}{4}\xi(t)\|  \nabla u(t)\|_{L^2(\Omega)}^2+c\xi(t)\bigg(\|  u_{t}(t)\|_{L^2(\Omega)}^2+(g\circ \nabla u)(t)\\
			&\quad+\| h(t)\|_{L^2(\Gamma_{1})}^2 +\int_{0}^{t}g(t-s)\| h(s)\|_{L^2(\Gamma_{1})}^2ds\bigg).
		\end{split}
	\end{align}
\end{lemma}

\begin{proof}Using the main equation and the divergence theorem, we get
	\begin{align}\label{pisi}
		\begin{split}
			\psi'(t)&=\xi(t)\int_{\Gamma_{1}}h(t)u(t)d\Gamma -\xi(t)\int_{\Omega}|\nabla u(t)|^2 dx\\
			&-\xi(t)\int_{\Gamma_{1}}u(t)\int_{0}^{t} g(t-s)h(s)dsd\Gamma +\xi(t)\int_{\Omega}|u_{t}(t)|^2dx \\
			&+\xi(t)\int_{\Omega}\nabla u(t)\int_{0}^{t}g(t-s)\nabla u(s)dsdx+\xi'(t)\int_{\Omega}u(t)u_{t}(t)dx.
		\end{split}
	\end{align}
	Now, we will estimate the terms at the right hand side of (\ref{pisi}). Young's, trace and Poincaré inequalities applied to first and third terms give rise to
	\begin{align}
		\int_{\Gamma_{1}}h(t)u(t)d\Gamma \leq \frac{1}{4\eta} \| h(t)\|_{L^2(\Gamma_{1})}^2+ \eta c \|  \nabla u(t)\|_{L^2(\Omega)}^2
	\end{align}
	and
	\begin{align}
		\begin{split}
		\int_{\Gamma_{1}} & u(t)\int_{0}^{t} g(t-s)h(s)dsd\Gamma\\
		& \leq \eta c \bigg( \int_{0}^{t}g(t-s)ds \bigg) \|  \nabla u(t)\|_{L^2(\Omega)}^2
		+\frac{1}{4\eta}\int_{0}^{t}g(t-s)\| h(s)\|_{L^2(\Gamma_{1})}^2ds
		\end{split}
	\end{align} for $\eta>0$.
	We estimate fifth term at the right hand side of (\ref{pisi}) by observing that
	\begin{align}\label{pisi1}
		\begin{split}
			&\int_{\Omega}\nabla u(t)\int_{0}^{t}g(t-s)\nabla u(s)dsdx  \\
			&\leq\frac{1}{2} \int_{\Omega}|\nabla u|^2 dx+\frac{1}{2} \int_{\Omega}\bigg(\int_{0}^{t}g(t-s)\Big(|\nabla u(s)- \nabla u(t) |+ | \nabla u(t)| \Big)ds\bigg)^2dx.
		\end{split}
	\end{align}
	From Cauchy-Schwarz inequality and $\int_{0}^{t}g(s)ds<1-L$ with $L\in (0,1)$ (see (A1)), we estimate the last term above as
	\begin{align}\label{pisi2}
		\begin{split}
			&\int_{\Omega}\bigg( \int_{0}^{t}g(t-s)\Big(|\nabla u(s)- \nabla u(t)|+|\nabla u(t)| \Big )ds\bigg)^2dx\\
			&\quad\leq (1+\frac{1}{\eta})(1-L)(g\circ \nabla u)(t)+(1+\eta)(1-L)^2 \|  \nabla u(t)\|_{L^2(\Omega)}^2.
		\end{split}
	\end{align}
	Combining above estimates and using the inequality
	\begin{align}
		\int_{\Omega}u(t)u_t(t)dx\leq \alpha c\|  \nabla u(t)\|_{L^2(\Omega)}^2+\frac{1}{4\alpha}\|  u_{t}(t)\|_{L^2(\Omega)}^2, \ \ \ \alpha >0,
	\end{align}
	we arrive at
	\begin{align*}
		\begin{split}
			&\psi'(t)\leq  \bigg(1+\frac{1}{4\alpha} k\bigg)\xi(t)\|u_{t}(t)\|_{L^2(\Omega)}^2+\frac{1}{2}(1+\frac{1}{\eta})(1-L)\xi(t)(g\circ \nabla u)(t)\\
			&-\frac{1}{2}\bigg( 1-(1+\eta)(1-L)^2-2k\alpha c-\eta  c \Big(1+\int_{0}^{t}g(t-s)ds\Big)\bigg)\xi(t)\|  \nabla u(t)\|_{L^2(\Omega)}^2\\
			&+\frac{1}{4\eta}\xi(t) \| h(t)\|_{L^2(\Gamma_{1})}^2 +\frac{1}{4\eta}\xi(t)\int_{0}^{t}g(t-s)\| h(s)\|_{L^2(\Gamma_{1})}^2ds,
		\end{split}
	\end{align*} from which the result of lemma follows with sufficiently small choices of $\eta$ and $\alpha$.
\end{proof}
Now, we estimate the rate at which the functional $\chi$ changes.
\begin{lemma}Under assumptions (A1) and (A2) on the relaxation function, $\chi$ satisfies
	\begin{align}\label{kayy}
		\begin{split}
			\chi'(t) \leq & c\xi(t)\Big(-(g'\circ \nabla u)(t)+(g\circ \nabla u)(t)+(g\circ_{\Gamma_{1}} h)(t)+\| h(t)\|_{L^2(\Gamma_{1})}^2\Big)\\
			&+\xi(t)\delta\left(1+2(1-L)^2\right)\|\nabla u(t)\|_{L^2(\Omega)}^2\\
			&+\xi(t)\left(\delta(1+k)-\int_{0}^{t}g(s)ds\right)\|u_t(t)\|_{L^2(\Omega)}^2
		\end{split}
	\end{align} for $\delta>0$.
\end{lemma}
\begin{proof}Using the main equation and divergence theorem, we obtain
	\begin{align}\label{kay}
		\begin{split}
			\chi&'(t)=-\xi(t)\int_{\Gamma_1}h(t)\int_{0}^{t} g(t-s)(u(t)-u(s)) ds d\Gamma\\
			&+\xi(t)\int_{\Omega} \nabla u(t) \int_{0}^{t}\Big(g(t-s)(\nabla u(t)-\nabla u(s))ds \Big)dx \\
			&+\xi(t)\int_{\Gamma_1}\bigg(\int_{0}^{t} g(t-s) h(s) ds  \bigg)\bigg(\int_{0}^{t} g(t-s)(u(t)-u(s))ds\bigg)d\Gamma\\
			& -\xi(t)\int_{\Omega}\bigg(\int_{0}^{t} g(t-s) \nabla u(s)ds  \bigg)\bigg(\int_{0}^{t} g(t-s)(\nabla u(t)-\nabla u(s))ds\bigg)dx \\
			&-\xi(t)\int_{\Omega}u_t(t)\int_{0}^{t}g'(t-s)(u(t)-u(s))dsdx\\
			&-\xi'(t)\int_{\Omega}u_t(t)\int_{0}^{t}g(t-s)(u(t)-u(s))dsdx-\xi(t)\Big(\int_{0}^{t}g(s)ds  \Big)\int_{\Omega}|u_t(t)|^2dx.
		\end{split}
	\end{align}
	Using Young's inequality and Lemma \ref{lempoin}, we estimate first term at the RHS of \eqref{kay} as
	\begin{align}\label{kay1}
		\begin{split}
			\int_{\Gamma_{1}}& h(t) \int_{0}^{t}\Big(g(t-s)(u(t)-u(s))ds \Big)d\Gamma \leq \delta \| h(t)\|_{L^2(\Gamma_{1})}^2+ c\frac{(1-L)}{4\delta}(g\circ \nabla u)(t)
		\end{split}
	\end{align} for $\delta>0$.
	Using Young's inequality and the assumption on the relaxation function, second term at the RHS of \eqref{kay} is estimated as
	\begin{align}
		\begin{split}
			\int_{\Omega}\nabla u(t) \int_{0}^{t}\Big(g(t-s)(\nabla u(t)- \nabla u(s))ds \Big)dx\\
			\leq \delta \|  \nabla u(t)\|_{L^2(\Omega)}^2+ \frac{1-L}{4\delta}(g\circ \nabla u)(t).
		\end{split}
	\end{align}
	Regarding  third, fourth, fifth and sixth terms at the RHS of \eqref{kay}, we have the estimates
	\begin{multline}
		\int_{\Gamma_1}\bigg(\int_{0}^{t} g(t-s) h(s) ds  \bigg)\bigg(\int_{0}^{t} g(t-s)(u(t)- u(s))ds\bigg)d\Gamma \\ \leq 2\delta(1-L)(g\circ_{\Gamma_{1}} h)(t)+2\delta(1-L)^{2} \| h(t)\|_{L^2(\Gamma_{1})}^2+c\frac{(1-L)}{4\delta}(g\circ \nabla u)(t),
	\end{multline}
	
	\begin{multline}
		\int_{\Omega}\bigg(\int_{0}^{t} g(t-s) \nabla u(s)ds  \bigg)\bigg(\int_{0}^{t} g(t-s)(\nabla u(t)-\nabla u(s))ds\bigg)dx
		\\ \leq \Big(2\delta +\frac{1}{4\delta} \Big)(1-L)(g\circ \nabla u)(t)+ 2\delta (1-L)^2\|  \nabla u(t)\|_{L^2(\Omega)}^2,
	\end{multline}
	
	\begin{align}
		\int_{\Omega}u_t\int_{0}^{t}g'(t-s)(u(t)-u(s))dsdx\leq \delta \|  u_{t}(t)\|_{L^2(\Omega)}^2-\frac{cg(0)}{4\delta}(g'\circ \nabla u)(t),
	\end{align}
	and
	\begin{align}\label{kay2}
		\int_{\Omega}u_t\int_{0}^{t}g(t-s)(u(t)-u(s))dsdx\leq  \delta \|  u_{t}(t)\|_{L^2(\Omega)}^2+c\frac{(1-L)}{4\delta}(g\circ \nabla u)(t).
	\end{align}
	The assertion of lemma is established  by combining (\ref{kay})-(\ref{kay2}).
\end{proof}

The main result of the paper is stated as follows.

\begin{theorem}\label{mainresultthm}
	Let $(u^0,u^1) \in H_{\Gamma_0}^1(\Omega)\times  L^{2}(\Omega)$, $h\in C^1([0,\infty), L^2(\Gamma_1))$  be given. Under assumptions (A1) and (A2) on the relaxation function,
	the modified  energy functional satisfies
	\begin{align}\label{modenest}
		\varepsilon_g(t)\leq\alpha_3 e^{(\frac{2\epsilon c}{1-2\epsilon c}-1)\int_0^t(\kappa\xi(s)-c_\epsilon)ds}\tilde{H}(t),
	\end{align}
	where $\alpha_3$, $c_\epsilon$ and  $\tilde{H}(t)$ are given in (\ref{al3}),(\ref{c_eps}) and (\ref{Htilde}), respectively.
	
	Moreover,  if $\xi(t)$ is bounded below by a positive constant and the boundary datum satisfies the conditions
	\begin{align}\label{decayh}
		\|h(t) \|_{L^2(\Gamma_{1})}^2, \|h_t(t)\|_{L^2(\Gamma_{1})}^2\quad =\quad \mathcal{O}\left(e^{-\lambda\int_0^t \xi(s)ds}\right),
	\end{align}
	where $\lambda$ is positive constant (can be arbitrarily small), then, \eqref{modenest} becomes a decay estimate for the modified  energy, where the decay rate is characterized by the  decay rate of the relaxation function $g$ and the boundary datum $h$.
\end{theorem}
\begin{proof}
	Since $g$ is positive and $g(0)>0$ then for any $t_0 > 0 $, we have
	
	\begin{align*}
		\int_0^t g(s)ds \geq \int_0^{t_0} g(s)ds=g_0, \ \ \ \ \forall t >t_0.
	\end{align*}
	By using (\ref{fdef}), (\ref{e3}), (\ref{pissi}) and (\ref{kayy}),  we obtain for $t\geq t_0$

	\begin{align}\label{f}
		\begin{split}
			F'(t)&\leq - \alpha \xi(t)\int_{\Omega} u_t^2 dx -\beta \xi(t)\int_{\Omega} |\nabla u|^2 dx +\Big(\epsilon c(1+ 2g(0))-\frac{1}{2}g(t)\Big) \int_{\Omega} |\nabla u|^2 dx \\ &\quad + \gamma (g'\circ \nabla u)(t)+\zeta \xi(t) (g\circ \nabla u)(t) +H(t)\\ &\quad +\frac{d}{dt}\bigg( \int_{\Gamma_{1}}u(t)h(t)d\Gamma-\int_{\Gamma_{1}}u(t)\int_{0}^{t}g(t-s)h(s)dsd\Gamma \bigg )
		\end{split}
	\end{align}
	where
	\begin{align*}
		\alpha := \epsilon_2(g_0-\delta(1+k))-c\epsilon_1,\\
		\beta:=\epsilon_1 \frac{L(2-L)}{4} -\epsilon_2 \delta \Big( 1+2(1-L)^2 \Big),\\
		\zeta:=c\epsilon_1+c\epsilon_2,
		\gamma :=\frac{1}{2}-c\epsilon_2 M,
	\end{align*}
	and
	\begin{align*}
		H(t)&:= \bigg(\frac{1}{4\epsilon}g(0)+ c\epsilon_{1}\xi(t)+c\epsilon_{2}\xi(t)\bigg)\| h(t)\|_{L^2(\Gamma_{1})}^2+\frac{1}{4\epsilon}\| h_{t}(t)\|_{L^2(\Gamma_{1})}^2\\& \quad+c\epsilon_{1}\xi(t)\int_{0}^{t}g(t-s)\| h(s)\|_{L^2(\Gamma_{1})}^2 ds  -\frac{1}{4\epsilon}\int_0^tg'(t-s)\| h(s)\|_{L^2(\Gamma_{1})}^2ds\\
		&+c\epsilon_{2}\xi(t)(g\circ_{\Gamma_{1}} h)(t).
	\end{align*}
	Since $g$ is positive,  $$\epsilon c(1+ 2g(0))-\frac{1}{2}g(t)\leq  \epsilon c(1+ 2g(0)).$$
	Let us denote
	\begin{align}\label{c_eps}
		c_\epsilon:=\epsilon c(1+ 2g(0)).
	\end{align}
	On the other hand, using  $g'(t)\leq - \xi (t)g(t)$ and the fact that $\xi(t)$ is nonincreasing, we obtain
	\begin{align}
		\gamma (g'\circ \nabla u)(t) +\zeta \xi(t) (g\circ \nabla u)(t)=-(\gamma-\zeta)\xi(t)(g\circ \nabla u)(t).
	\end{align}
	Here we choose $\epsilon_1, \epsilon_2$ such that
	$\alpha , \beta > 0,$ and $\gamma-\zeta > 0 $,
	and also set $\kappa:= \min \{\alpha, \beta, (\gamma-\zeta) \}$. Then, (\ref{f}) yields
	\begin{align}\label{fenq}
		\begin{split}
			F'(t)\leq &\Big(c_\epsilon -  \kappa \xi(t)\Big) F(t)\\
			&+\frac{d}{dt}\bigg( \int_{\Gamma_{1}}u(t)h(t)d\Gamma-\int_{\Gamma_{1}}u(t)\int_{0}^{t}g(t-s)h(s)dsd\Gamma \bigg )+ H(t).
		\end{split}
	\end{align}
	Now, we multiply (\ref{fenq}) by $e^{ \int_{0}^{t}(\kappa\xi(s)-c_\epsilon) ds}$ and integrate over $(0,t)$, obtaining
	\begin{align}\label{fenq2}
		\begin{split}
			\int_{0}^{t}&e^{ \int_{0}^{s}(\kappa\xi(\tau)-c_\epsilon)d\tau}\Big(F'(s)+(\kappa\xi(s)-c_\epsilon) F(s)\Big)ds\\
			\leq &\int_{0}^{t}e^{\int_{0}^{s}(\kappa\xi(\tau)-c_\epsilon)d\tau} \frac{d}{ds}\bigg(\int_{\Gamma_1}h(s)u(s)d\Gamma \bigg)ds \\ 
			& -\int_{0}^{t}e^{\int_{0}^{s}(\kappa\xi(\tau)-c_\epsilon)d\tau} \frac{d}{ds}\bigg(\int_{\Gamma_1}u(s)\int_{0}^{s}g(s-\tau)h(\tau)d\tau d\Gamma  \bigg)ds\\
			&+ \int_{0}^{t}e^{\int_{0}^{s}(\kappa\xi(\tau)-c_\epsilon)d\tau}H(s)ds.
		\end{split}
	\end{align}
	Then, (\ref{fenq2}) yields
	
	\begin{align}\label{ex}
		\begin{split}
			F(t)e^{\int_{0}^{t}(\kappa\xi(s)-c_\epsilon)ds}&\leq F(0)+\int_{0}^{t}e^{\int_{0}^{s}(\kappa\xi(\tau)-c_\epsilon)d\tau} \frac{d}{ds}\bigg(\int_{\Gamma_1}h(s)u(s)d\Gamma \bigg)ds \\ & \quad-\int_{0}^{t}e^{\int_{0}^{s}(\kappa\xi(\tau)-c_\epsilon)d\tau} \frac{d}{ds}\bigg(\int_{\Gamma_1}u(s)\int_{0}^{s}g(s-\tau)h(\tau)d\tau d\Gamma  \bigg)ds\\ &
			\quad +\int_{0}^{t}e^{\int_{0}^{s}(\kappa\xi(\tau)-c_\epsilon)d\tau}H(s)ds.
		\end{split}
	\end{align}
	Rewriting the second and the third terms  of (\ref{ex}) on the right hand side as follows
	\begin{align}\label{est1}
		\begin{split}
			\int_{0}^{t}e^{\int_{0}^{s}(\kappa\xi(\tau)-c_\epsilon)d\tau} &\frac{d}{ds}\bigg(\int_{\Gamma_1}h(s)u(s)d\Gamma \bigg)ds \\
			& = e^{\int_{0}^{t}(\kappa\xi(s)-c_\epsilon)ds}\int_{\Gamma_1}h(t)u(t)d\Gamma -\int_{\Gamma_1}u^0 h(0)d\Gamma\\
			& \quad-\int_{0}^{t}(\kappa \xi(s)-c_\epsilon)e^{\int_{0}^{s}(\kappa\xi(\tau)-c_\epsilon)d\tau}\int_{\Gamma_1}h(s)u(s)d\Gamma ds,
		\end{split}
	\end{align}
	
	\begin{align}\label{est2}
		\begin{split}
			\int_{0}^{t}&e^{\int_{0}^{s}(\kappa\xi(\tau)-c_\epsilon)d\tau} \frac{d}{ds}\bigg(\int_{\Gamma_1}u(s)\int_{0}^{s}g(s-\tau)h(\tau)d\tau d\Gamma  \bigg)ds\\
			&=e^{\int_{0}^{t}(\kappa\xi(s)-c_\epsilon)ds}\int_{\Gamma_1}u(t)\int_{0}^{t}g(t-s)h(s)ds d\Gamma \\ &\quad- \int_{0}^{t}(\kappa\xi(s)-c_\epsilon)e^{\int_{0}^{s}(\kappa\xi(\tau)-c_\epsilon)d\tau}\int_{\Gamma_1}u(s)\int_{0}^{s}g(s-\tau)h(\tau)d\tau d\Gamma ds
		\end{split}
	\end{align}
	and setting $\tilde{F}(t):=F(t)e^{\int_{0}^{t}(\kappa\xi(s)-c_\epsilon)ds}$, we derive
	
	\begin{align}\label{final}
		\begin{split}
			\tilde{F}(t)&\leq F(0)+ e^{\int_{0}^{t}(\kappa\xi(s)-c_\epsilon)ds}\int_{\Gamma_1}h(t)u(t)d\Gamma -\int_{\Gamma_1}u^0 h(0)d\Gamma\\
			& \quad-\int_{0}^{t}(\kappa \xi(s)-c_\epsilon)e^{\int_{0}^{s}(\kappa\xi(\tau)-c_\epsilon)d\tau}\int_{\Gamma_1}h(s)u(s)d\Gamma ds\\& \quad +e^{\int_{0}^{t}(\kappa\xi(s)-c_\epsilon)ds}\int_{\Gamma_1}u(t)\int_{0}^{t}g(t-s)h(s)ds d\Gamma \\ &\quad- \int_{0}^{t}(\kappa\xi(s)-c_\epsilon)e^{\int_{0}^{s}(\kappa\xi(\tau)-c_\epsilon)d\tau}\int_{\Gamma_1}u(s)\int_{0}^{s}g(s-\tau)h(\tau)d\tau d\Gamma ds \\ & \quad +\int_{0}^{t}e^{\int_{0}^{s}(\kappa\xi(\tau)-c_\epsilon)d\tau}H(s)ds .
		\end{split}
	\end{align}
	Applying $\epsilon-$Young's, trace and Poincaré  inequalities, substituting $F(0)$, and using the fact that the quantity $$e^{\int_{0}^{t}(\kappa\xi(s)-c_\epsilon)ds} \| \nabla u(t)\|_{L^2(\Omega)}^2$$ is contained in $\tilde{F}(t)$, from (\ref{final}), we get
	\begin{align}\label{final2}
		\begin{split}
			(1-2\epsilon c)\tilde{F}(t)&\leq 2\epsilon c \int_{0}^{t}(\kappa\xi(s)-c_\epsilon)\tilde{F}(s) ds +\frac{1}{2}c\|\nabla u^0\|_{L^2(\Omega)}^2+\frac{1}{2}\epsilon_{1}\xi(0)\|u^0\|_{L^2(\Omega)}^2\\
			& \quad  +\frac{1}{2}\Big(1+\epsilon_{1}\xi(0) \Big)\|u^1\|_{L^2(\Omega)}^2    +\int_{0}^{t}e^{\int_{0}^{s}(\kappa\xi(\tau)-c_\epsilon)d\tau}H(s)ds \\
			& \quad +\frac{1}{4\epsilon}e^{\int_{0}^{t}(\kappa\xi(s)-c_\epsilon)ds}\|h(t)\|_{L^2(\Gamma_{1})}^2
			+ \frac{1}{4\epsilon} e^{\int_{0}^{t}(\kappa\xi(s)-c_\epsilon)ds} \tilde{h}(t) \\
			& \quad +\frac{1}{4\epsilon}\int_{0}^{t}(\kappa\xi(s)-c_\epsilon)e^{\int_{0}^{s}(\kappa\xi(\tau)-c_\epsilon)d\tau}\|h(s)\|_{L^2(\Gamma_{1})}^2ds\\
			& \quad  +\frac{1}{4\epsilon}  \int_{0}^{t}(\kappa\xi(s)-c_\epsilon)e^{\int_{0}^{s}(\kappa\xi(\tau)-c_\epsilon)d\tau}\tilde{h}(s)ds
		\end{split}
	\end{align}
	where $$\tilde{h}(t):= \frac{1}{4\epsilon}\Big( (1-L)^{2} +1\Big) \| h(t)\|_{L^2(\Gamma_{1})}^2+\frac{1}{4\epsilon}(1-L)(g\circ_{\Gamma_1}h)(t).$$
	Setting
	\begin{align}\label{Htilde}
		\begin{split}
			\tilde{H}(t)&:=\int_{0}^{t}e^{\int_{0}^{s}(\kappa\xi(\tau)-c_\epsilon)d\tau}H(s)ds +\frac{1}{4\epsilon}e^{\int_{0}^{t}(\kappa\xi(s)-c_\epsilon)ds}\|h(t)\|_{L^2(\Gamma_{1})}^2  \\ & \quad+\frac{1}{4\epsilon}\int_{0}^{t}(\kappa\xi(s)-c_\epsilon)e^{\int_{0}^{s}(\kappa\xi(\tau)-c_\epsilon)d\tau}\|h(s)\|_{L^2(\Gamma_{1})}^2ds\\
			& \quad  + \frac{1}{4\epsilon} e^{\int_{0}^{t}(\kappa\xi(s)-c_\epsilon)ds} \tilde{h}(t)+\frac{1}{4\epsilon}  \int_{0}^{t}(\kappa\xi(s)-c_\epsilon)e^{\int_{0}^{s}(\kappa\xi(\tau)-c_\epsilon)d\tau}\tilde{h}(s)ds+ C_0
		\end{split}
	\end{align}
	where $C_0$ is the positive constant defined by
	\begin{align*}C_0:=\|h^0\|_{L^2(\Gamma_{1})}^2 & +
	\frac{1}{2}\Big(1+c^2 \Big)\|\nabla u^0\|_{L^2(\Omega)}^2+\frac{1}{2}\epsilon_{1}\xi(0)\|u^0\|_{L^2(\Omega)}^2\\
	& +\frac{1}{2}\Big(1+\epsilon_{1}\xi(0) \Big)\|u^1\|_{L^2(\Omega)}^2,
	\end{align*}
	(\ref{final2}) becomes
	\begin{align}
		\tilde{F}(t)\leq \Big(\frac{2\epsilon c}{1-2\epsilon c} \Big) \int_{0}^{t}(\kappa\xi(s)-c_\epsilon) \tilde{F}(s) ds +\Big(\frac{1}{1-2\epsilon c}\Big)\tilde{H}(t).
	\end{align}
	Then, using  Gronwall's Inequality, we obtain
	\begin{align}
		\tilde{F}(t) \leq \Big(\frac{1}{1-2\epsilon c}\Big)   e^{(\frac{2\epsilon c}{1-2\epsilon c})  \int_0^t (\kappa\xi(s)-c_\epsilon)ds}\tilde{H}(t).
	\end{align}
	From the above estimate, we infer
	\begin{align} \label{ss}
		\varepsilon_g(t) \leq \alpha_2 F(t) &\leq \alpha_3 e^{(\frac{2\epsilon c}{1-2\epsilon c}-1)\int_0^t(\kappa\xi(s)-c_\epsilon)ds}\tilde{H}(t),
	\end{align}
	where
	\begin{align}\label{al3}
		\alpha_3:= \alpha_2\Big(\frac{1}{1-2\epsilon c}\Big).
	\end{align}

	Moreover, if $\xi(t)$ is bounded from below by a positive constant, then we can choose $\epsilon$ sufficiently small that $\kappa\xi(s)-c_\epsilon> 0.$ Therefore, (\ref{decayh}) and (\ref{ss}) lead to decay of the energy with decay rates characterized by the  decay rate of the relaxation function $g$ and the boundary datum $h$.
\end{proof}

\section{Numerical experiments for 1D wave equation with a distributed damping}
\subsection{Numerical Solution} \label{NumSolDamped}
In this section, we shall construct the numerical solution of  the following 1D  linear  damped wave equation by using an explicit method.
\begin{align}
	&  u_{tt}-u_{xx} +a(x)u_t=0, \quad x\in (0,L), t\in(0,T)\label{main-damped-num}\\
	&   u(x,0)=\varphi(x), \label{i1-damped-num}\\
	& 	u_{t}(x,0)=\psi(x), \label{i2-damped-num}\\
	& 	u(0,t)=f(t),\label{dirichlet-damped-num}\\
	& 	u_x(L,t)= h(t).\label{neumann-damped-num}
\end{align}
We replace the temporal and spatial domains by a finite number of mesh points:
\begin{equation}\label{tmesh}0=t_0<t_1<...<t_{{N_t}-1}=T,
\end{equation}
\begin{equation}\label{xmesh}0=x_0<x_1<...<x_{{N_x}-1}=L.
\end{equation}
To have uniformly distributed mesh points we introduce constant mesh spacings  $\triangle t>0$ and $\triangle x>0$. Therefore, we simply have 
\begin{equation}x_i=i\triangle x, \ \  i=0,…,N_x-1\end{equation} and \begin{equation}t_n=n\triangle t,\ \  n=0,…,N_t-1.\end{equation}
The solution $u=u(x,t)$ is sought at the mesh points. We will denote the approximation of the exact solution at the mesh point $(x_i,t_n)$ by $u_i^n$, where $i=0,…,N_x-1$ and $n=0,…,N_t-1$. 

We first define the two vectors ($U^0=[U^0_i]_{1\le i\le N_x-2}^T, U^1=[U^1_i]_{1\le i\le N_x-2}^T$) by setting
$U_i^0=\varphi(x_i)$ and
$ U_i^1=U_i^0+  \psi(x_i)\triangle t, $ where the latter is motivated from the numerical version of (\ref{i2-damped-num}), i.e.,
$\displaystyle \frac{U_i^1-U_i^0}{\triangle t}= \psi(x_i).$  Using an explicit method and second order finite difference scheme in time and space for the interior mesh points only, i.e, for $i = 1, ..., N_{x}-2,$ $n=1, ..., N_t-2$, in view of the numerical version of (\ref{main-damped-num}) 
\begin{align}
	\begin{split}
		\frac{ U_{i}^{n+1}-2 U_{i}^{n}+ U_{i}^{n-1}}{(\triangle t)^2}-\frac{ U_{i+1}^{n}-2 U_{i}^{n}+ U_{i-1}^{n}}{(\triangle x)^2}+a_i \frac{U_{i}^{n+1}- U_{i}^{n}}{\triangle t}=0.
	\end{split}
\end{align}
Now, organizing the terms with respect to time steps, we obtain
\begin{align}\label{int_soln_formula-damped-num}
	U_{i}^{n+1}=\Big( \alpha_i U_{i+1}^{n}+ \beta_i U_{i}^{n}+ \alpha_i U_{i-1}^{n}\Big)- \zeta_i U_{i}^{n-1},
\end{align}
where
$\displaystyle
\alpha_i:=\frac{r^2}{1+a_i\triangle t},
\beta_i:=\frac{2-2r^2 +a_i\triangle t}{1+a_i\triangle t},
\zeta_i:=\frac{1}{1+a_i\triangle t},
r:=\frac{\triangle t}{\triangle x}.
$

	We need to incorporate boundary data into the construction of the numerical solution. To this end, we use the information associated with (\ref{dirichlet-damped-num}), and thereby define the quantities
	\begin{equation}\label{U0Dirichlet}U_0^n=f(t_n), \quad n=0,...,N_t-1.\end{equation}
	Writing $i = 1$ in (\ref{int_soln_formula-damped-num}), we obtain
	\begin{align}\label{for_0-damped-num}
U_{1}^{n+1}=\Big( \alpha_1 U_{2}^{n}+ \beta_1 U_{1}^{n}+ \alpha_1 U_{0}^{n}\Big)- \zeta_1 U_{1}^{n-1}.
\end{align}

Now, we numerically construct the Neumann boundary data by using the backward $O(h^2)$ approximation, which is defined by
$$U_x\approx  \frac{U_{i-2}^n-4 U_{i-1}^n+3U_{i}^n}{2 \triangle x}.$$
Therefore, the numerical version of (\ref{neumann-damped-num}) leads us to define the quantities
\begin{align}\label{neumanndata2-damped-num}
U_{N_x-1}^n= \frac{1}{3}( 2\triangle x h(t_n)- U_{N_x-3}^n+4U_{N_x-2}^n).
\end{align}
Writing $i = N_x-2$ in (\ref{int_soln_formula-damped-num}), we obtain
\begin{align}\label{for_L-damped-num}
U_{N_x-2}^{n+1}=\Big( \alpha_{N_x-2} U_{N_x-1}^{n}+ \beta_{N_x-2} U_{N_x-2}^{n}+ \alpha_{N_x-2} U_{N_x-3}^{n}\Big)- \zeta_{N_x-2} U_{N_x-2}^{n-1}.
\end{align}
In  (\ref{for_L-damped-num}), $U_{N_x-1}^{n}$ on the RHS is unknown but it can be eliminated by
substituting $U_{N_x-2}^{n},U_{N_x-3}^{n}$ into  (\ref{for_L-damped-num}) by using (\ref{neumanndata2-damped-num}):
\begin{align}\label{forLbecomes-damped-num}
	\begin{split}
U_{N_x-2}^{n+1}=&\Big((\beta_{N_x-2}+\frac{4}{3}\alpha_{N_x-2})U_{N_x-2}^{n} + \frac{2}{3}\alpha_{N_x-2}  U_{N_x-3}^{n}+\frac{2}{3}\alpha_{N_x-2} \triangle x h(t_n)\Big)\\
&-\zeta_{N_x-2} U_{N_x-2}^{n-1}.
\end{split}
\end{align}

We define the following matrices
\begin{align*}
{K} =
\begin{bmatrix}
	\beta_1   &\alpha_1    &              &           &                                   & \textrm{\huge0} \\
	\alpha_2                &\beta_2    &\alpha_2       &           &                                   & \\
	&\ddots    &\ddots        &\ddots     &                                   & \\
	&          &\ddots        & \ddots    &\ddots                             & \\
	&          &              &\alpha_{N_x-3}    &\beta_{N_x-3}               &\alpha_{N_x-3} \\
	\textrm{\huge0}         &          &              &           &\frac{2}{3}\alpha_{N_x-3}                 &\beta_{N_x-2}+\frac{4}{3}\alpha_{N_x-2} 
\end{bmatrix}\end{align*} \text{ and }\begin{align*}\zeta=\begin{bmatrix}
	\zeta_1   &0    &              &           &                                   & \textrm{\huge0} \\
	0                &\zeta_2    &0       &           &                                   & \\
	&\ddots    &\ddots        &\ddots     &                                   & \\
	&          &\ddots        & \ddots    &\ddots                             & \\
	&          &              &0    &\zeta_{N_x-3}               &0 \\
	\textrm{\huge0}         &          &              &           &0                 &\zeta_{N_x-2}
\end{bmatrix},
\end{align*} and the following vectors
\begin{equation}
\mathbf{D}_n=
\begin{bmatrix}
	\alpha_1 f(t_n) &   0 & ...  &  0
\end{bmatrix}^T\text{ and } 	\mathbf{N}_n=
\begin{bmatrix}
	0&
	.&
	.&
	.&
	0&
	\frac{2}{3}\alpha_{N_x-2} \triangle x h_{N_x-1}^n
\end{bmatrix}^T.
\end{equation}We set
\begin{align}\label{solnmatrixform-damped-num}
\begin{split}
	U^{n+1}=KU^{n}-{\zeta} U^{n-1}+\mathbf{D}_n+\mathbf{N}_n, \quad n=1,...,N_t-2.
\end{split}
\end{align}
Finally, the numerical solution of the initial-boundary value problem is defined by the matrix
$$[u_{in}]_{0\le i\le N_x-1, 0\le n\le N_t-1},$$ where $u_{in}=U_i^n$, $0\le i\le N_x-1, 0\le n\le N_t-1.$ Note that $U_i^n$ can be found by using \eqref{solnmatrixform-damped-num} for $i=1,...,N_x-2$ and $n=0,...,N_t-1$ while $U_0^n$ and $U_{N_x-1}^n$ can be found by using \eqref{U0Dirichlet} and \eqref{neumanndata2-damped-num} for $n=0,...,N_t-1$.

\subsection{Numerical Simulations}

In this section, we present numerical examples involving external Neumann manipulations with varying qualitative behaviours. These include cases where $h$ exhibits exponential decay, oscillation, asymptotic constancy, or growth. We analyze behavior of solutions subject to different types of distributed damping, such as $a(x)=e^{-x}, \sin^2(x), (x+1)^2$.  To explore these scenarios, we consider the following initial-boundary value problem:
\begin{align}\label{nummodel-damped}
	\left\{
	\begin{array}{ll}
		u_{tt}-u_{xx} + a(x)u_{t}=0    \\
		u(x,0)=\cos(x/5),\ \ \  u_{t}(x,0)=0, \\
		u(0,t) =0,\ \ \   u_x(L,t)=h(t) ,   \\
	\end{array}
	\right.
\end{align}
where $x \in[0,10\pi].$ 
\subsubsection*{Case 1: $a(x)=e^{-x}$}
In the following simulation we take $h(t)=e^{-t}$ (an exponentially decaying manipulation which satisfies assumptions of Theorem \ref{TemamMainThm}).

\begin{figure}[H]
	\centering
	\subfloat[Energy Plot, $h(t)=e^{-t}$ and $a(x)=e^{-x}$. \label{h=exp(-t)-damped}]{%
		\resizebox*{7cm}{!}{\includegraphics{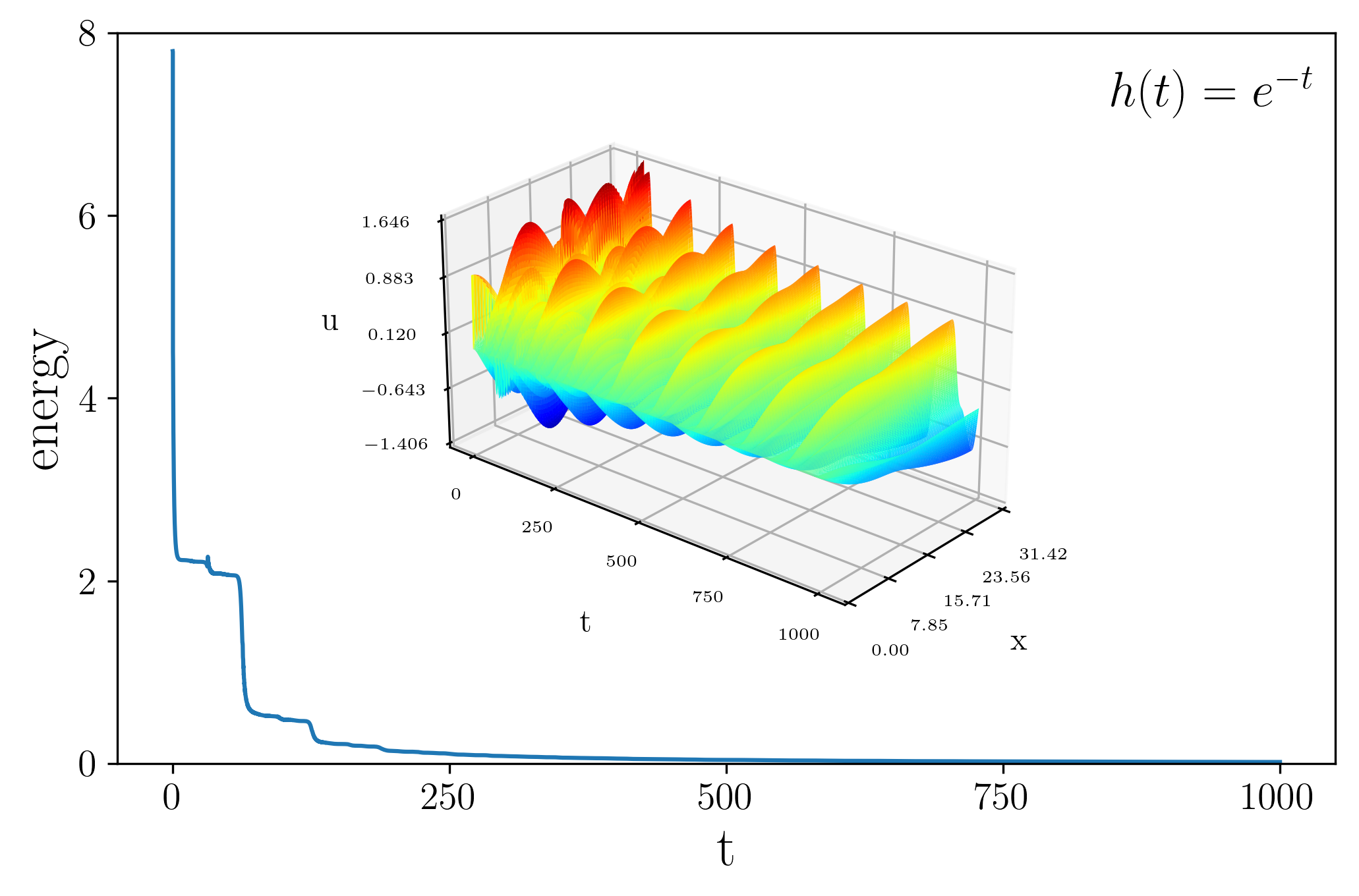}}}\hspace{5pt}
	\subfloat[Energy Plot, $h(t)=\sin(\frac{t}{5})$ and $a(x)=e^{-x}$ \label{h=sinli-damped}]{%
		\resizebox*{7cm}{!}{\includegraphics{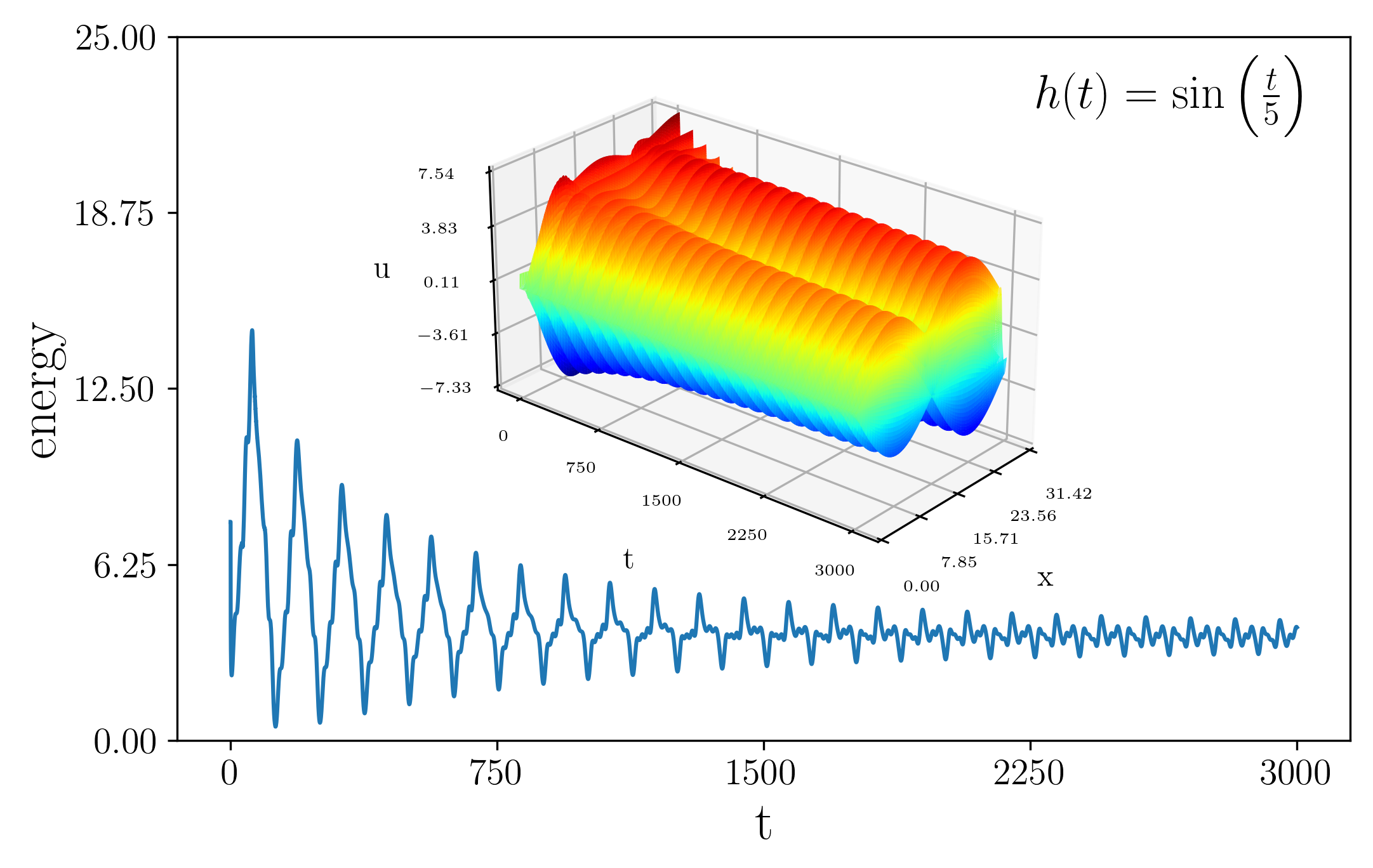}}}\\
	\subfloat[Energy Plot, $h(t)=\frac{5t}{t+1}$ and $a(x)=e^{-x}$.\label{h=5tli-damped}]{%
		\resizebox*{7cm}{!}{\includegraphics{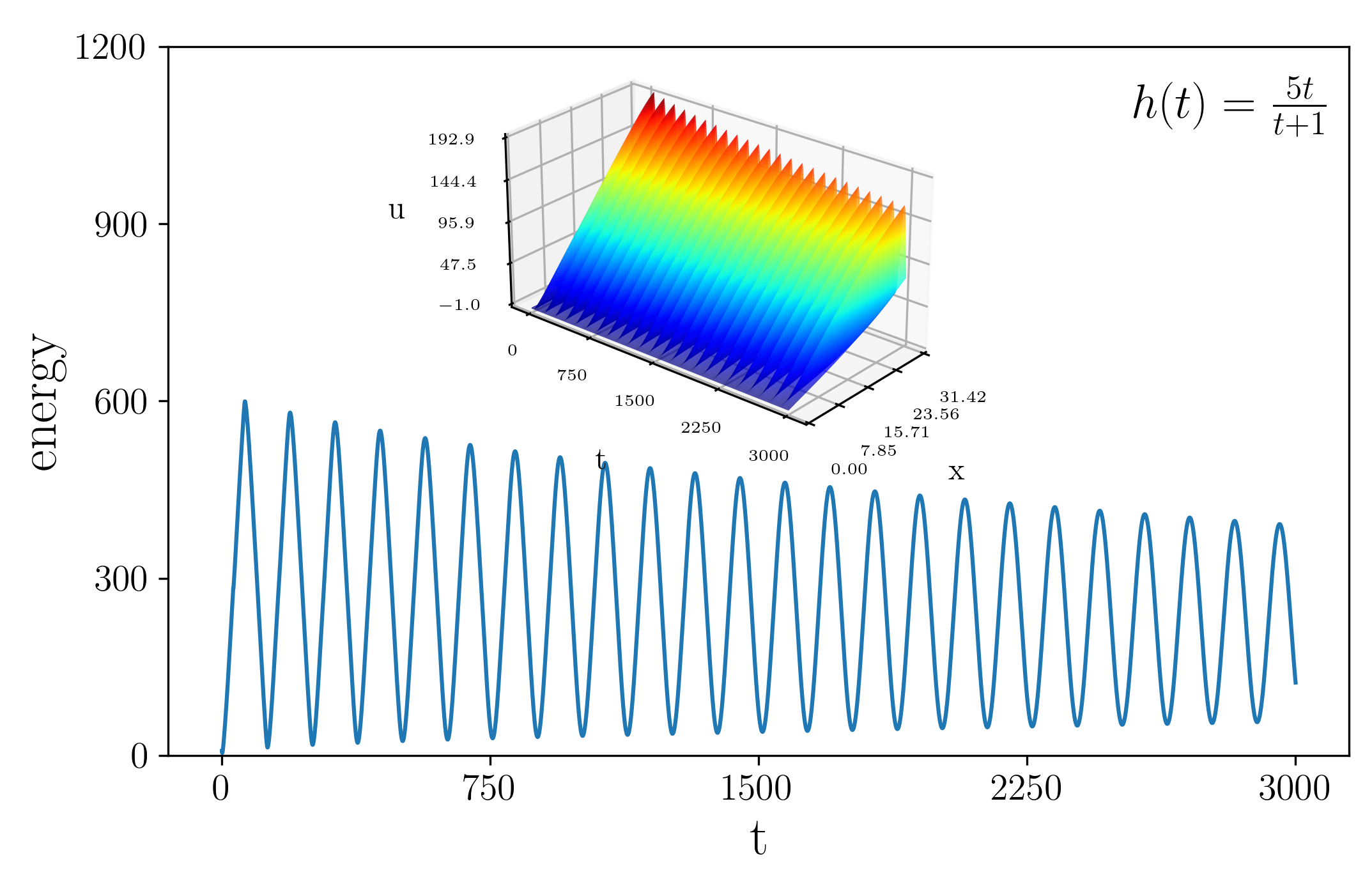}}}
	\subfloat[Energy Plot, $h(t)=\sqrt{t}$ and $a=e^{-x}$ . \label{h=sqrt-damped}]{%
		\resizebox*{7cm}{!}{\includegraphics{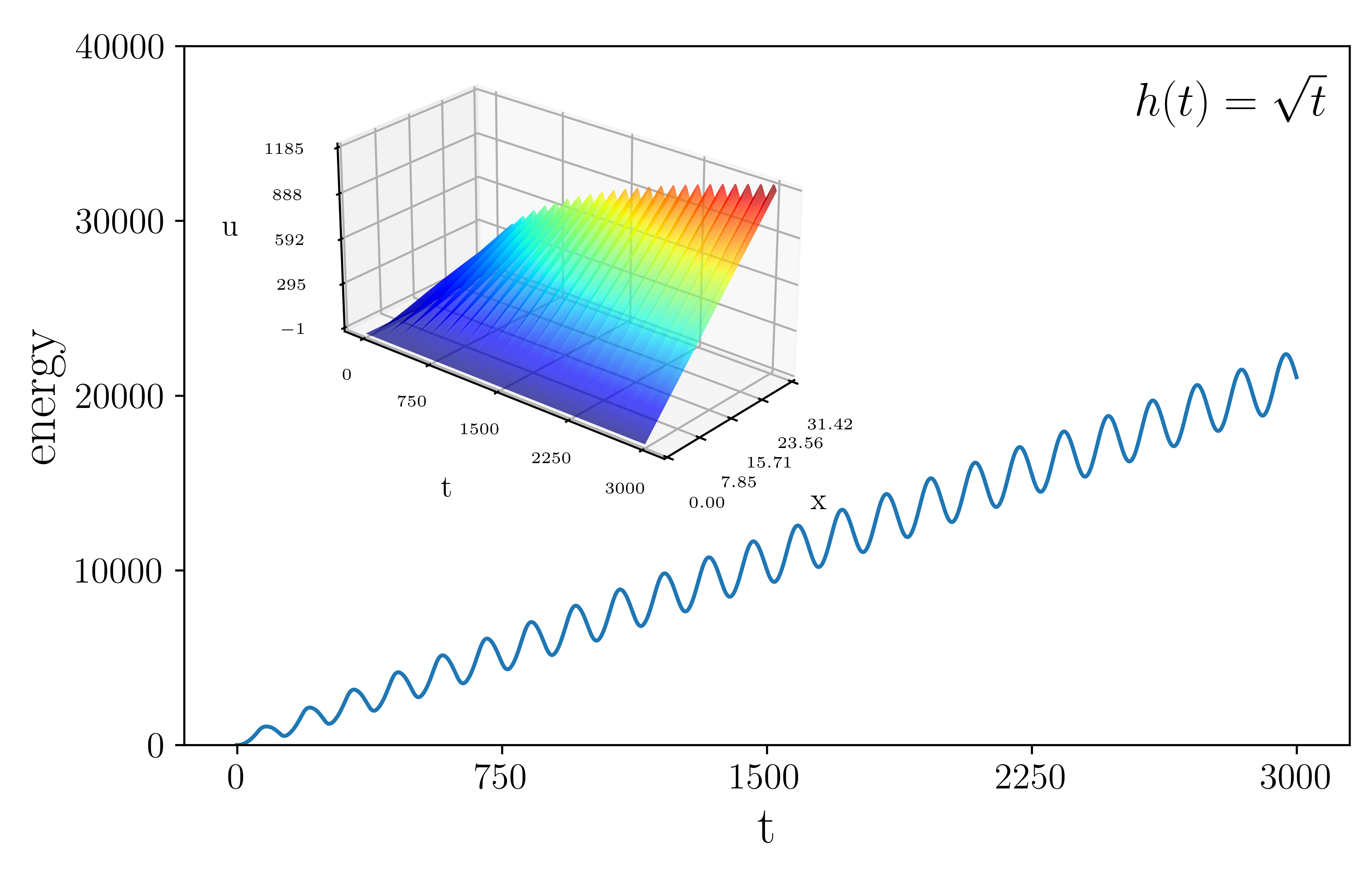}}}\\
	\caption{Linear waves subject to external Neumann manipulation and spatially decaying damping.} \label{sample-figure1}
\end{figure}

%\begin{figure}
%	\centering
%	\includegraphics[scale=0.7]{hexptdamped.png}
%	\caption{
%		Energy of solution plotted against time for $h(t)=e^{-t}$ and $a(x)=e^{-x}$.
%	}\label{h=exp(-t)-damped}
%\end{figure}

\noindent Figure \ref{h=exp(-t)-damped} demonstrates that the energy decays since $h(t)$ decays, with the damping effect clearly visible. This behavior is aligned with the estimate of Theorem \ref{TemamMainThm}.
%\textbf{Case 2:}
Next, we present three examples where the  external manipulation does not decay, violating assumption \eqref{h_temam} of Theorem \ref{TemamMainThm}. 

%\begin{figure}
%	\centering
%	\includegraphics[scale=0.7]{hsinlidamped.png}
%	\caption{
%		Energy of solution plotted against time for $h(t)=\sin(\frac{t}{5})$ and $a(x)=e^{-x}$ .
%	}\label{h=sinli-damped}
%\end{figure}

\noindent In Figure \ref{h=sinli-damped}, where $h(t)$ is an oscillating function, the energy also oscillates around a certain level. However, the amplitude of oscillations gradually diminish due to the damping, showing a stabilization effect.

%\begin{figure}
%	\centering
%	\includegraphics[scale=0.7]{h5tlidamped.png}
%	\caption{
%		Energy of solution plotted against time for $h(t)=\frac{5t}{t+1}$ and $a(x)=e^{-x}$.
%	}\label{h=5tli-damped}
%\end{figure}

\noindent Figure \ref{h=5tli-damped} shows that, when $h(t)$ asymptotically approaches a constant, so does the average of its energy. Similar to Figure \ref{h=sinli-damped}, the amplitude of oscillations decrease, but in this case the  damping  appears less effective in regard to how fast it can stabilize the wave.

%\textbf{Case 3:} In this case, we consider an external Neumann manipulation which is increasing in time. 

%\begin{figure}
%	\centering
%	\includegraphics[scale=0.7]{hsqrtdamped.png}
%	\caption{
%		Energy of solution plotted against time for $h(t)=\sqrt{t}$ and $a=e^{-x}$ .
%	} \label{h=sqrt-damped}
%\end{figure}

\noindent In Figure \ref{h=sqrt-damped}, we observe that an increasing Neumann manipulation allows energy to be injected into the system from the boundary, causing the damping effect to vanish.

%$\bullet$ Secondly, we examine the scenario where the damping coefficient is an oscillating function.\\
\subsubsection*{Case 2: $a(x)=\sin^2(x)$}
In this section, we examine the scenario where the damping coefficient is a nonnegative oscillating function of the spatial variable. Here, the damping coefficient $a=a(x)$ no longer satisfies the lower bound assumption \eqref{a(x)}. Despite this,  it remains effective in stabilizing solutions (around a certain, possibly nonzero level), provided Neumann input does not (asymptotically) exhibit temporal growth, as seen in Figure \ref{sinkare,hdecay}\,-\,Figure \ref{a=sinkare,h=5tli}.

\begin{figure}[H]
	\centering
	\subfloat[Energy Plot, $h(t)=e^{-t}$ and $a(x)=\sin^2(x)$. \label{sinkare,hdecay}]{%
		\resizebox*{7cm}{!}{\includegraphics{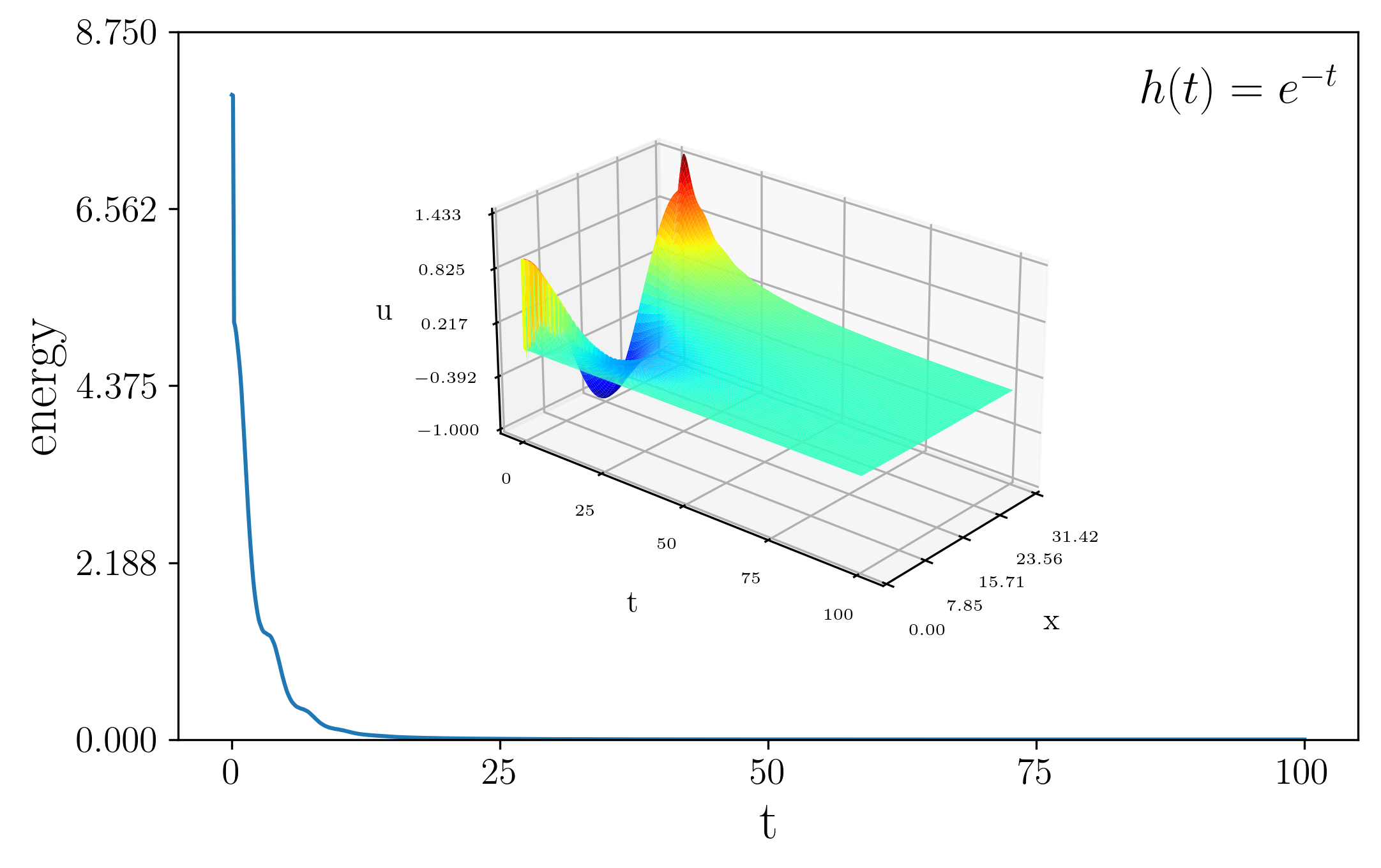}}}\hspace{5pt}
	\subfloat[Energy Plot, $h(t)=\sin(\frac{t}{5})$ and $a(x)=\sin^2(x)$ \label{a=sin2, h=sin}]{%
		\resizebox*{7cm}{!}{\includegraphics{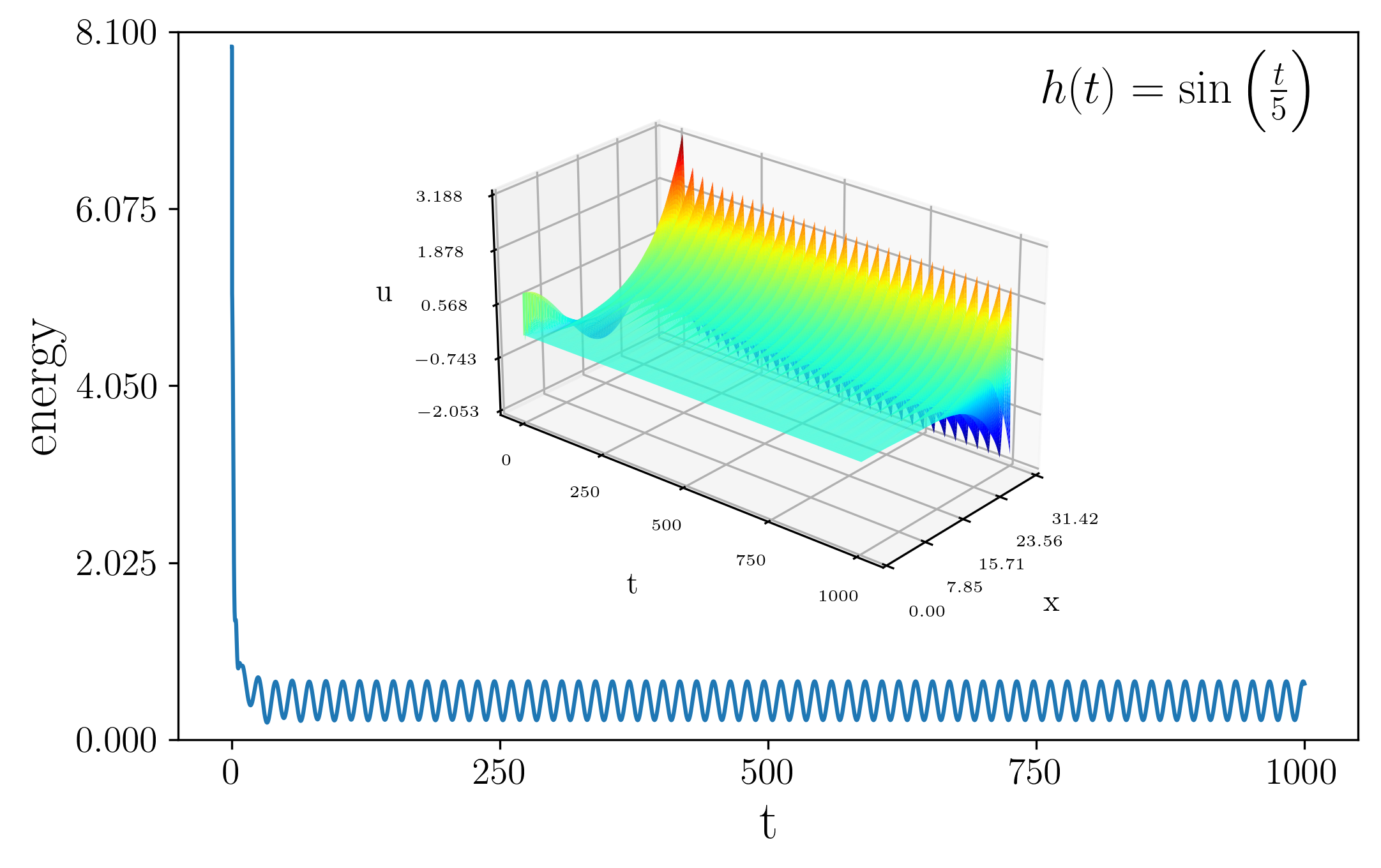}}}\\
	\subfloat[Energy Plot, $h(t)=\frac{5t}{t+1}$ and $a(x)=\sin^2(x)$. \label{a=sinkare,h=5tli}]{%
		\resizebox*{7cm}{!}{\includegraphics{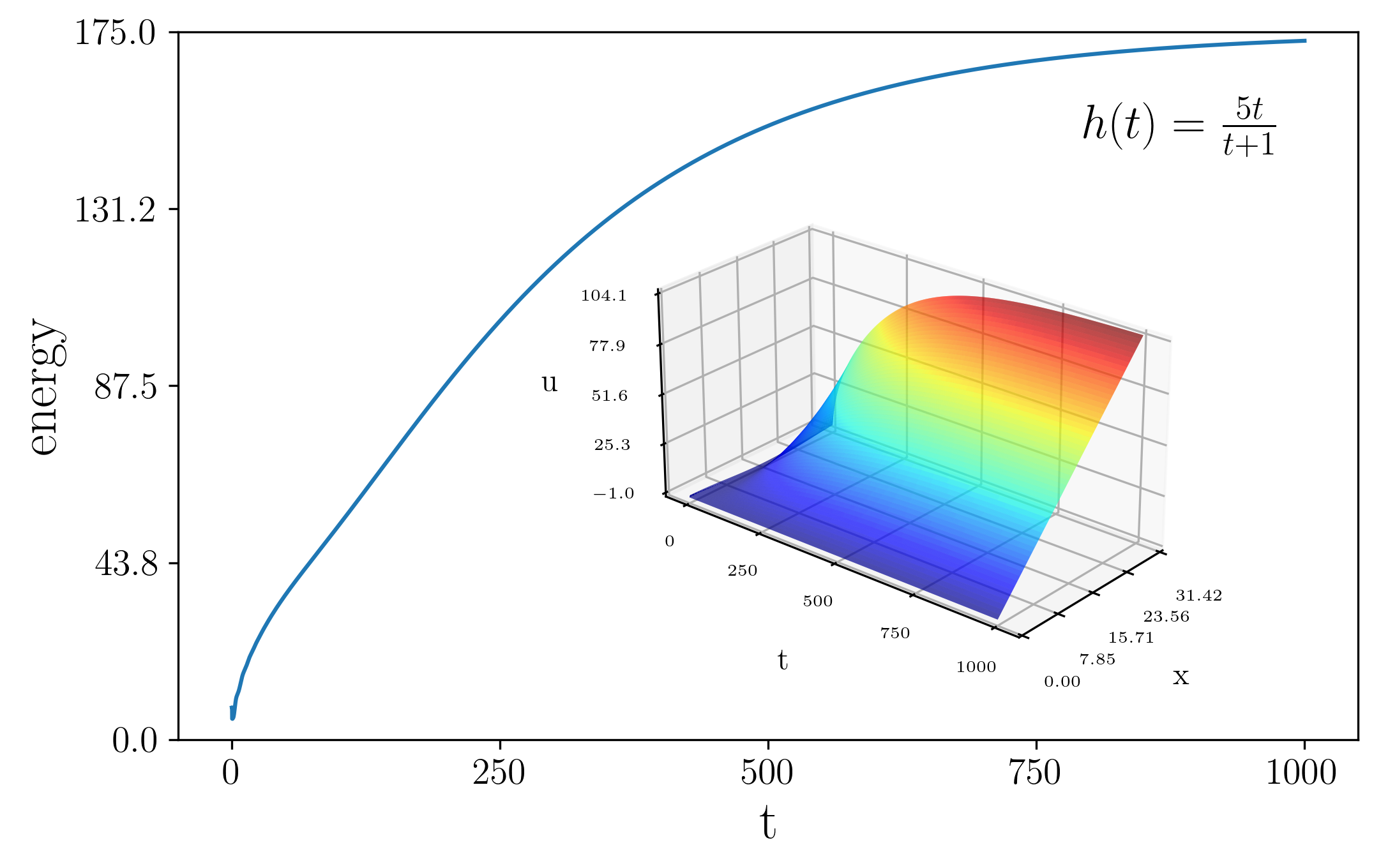}}}
	\subfloat[Energy Plot, $h(t)=\sqrt{t}$ and $a=\sin^2(x)$ . \label{a=sinkare,h=sqrt}]{%
		\resizebox*{7cm}{!}{\includegraphics{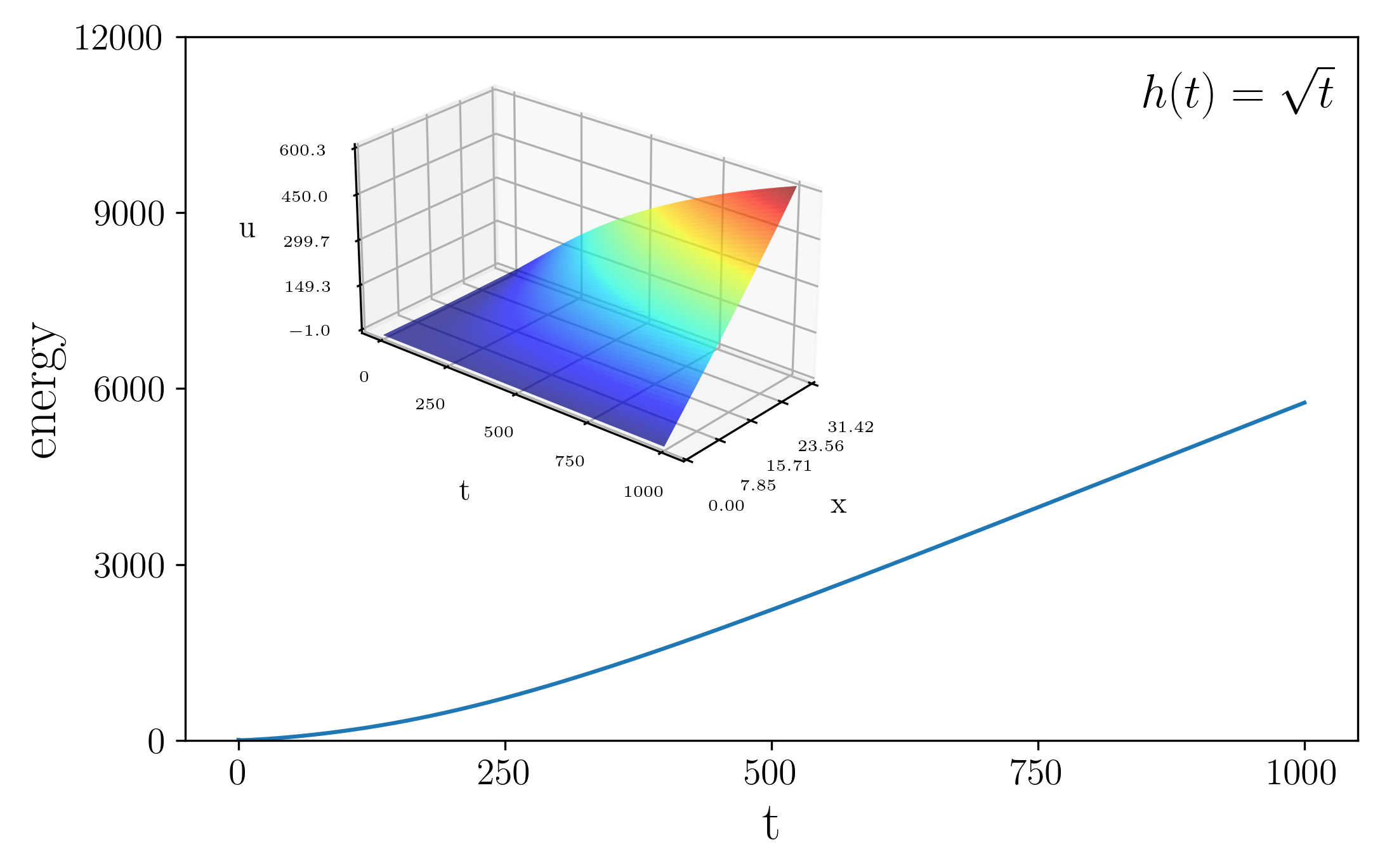}}}\\
	\caption{Linear waves subject to external Neumann manipulation and periodic damping.} \label{sample-figure2}
\end{figure}

%\begin{figure}
%	\centering
%	\includegraphics[scale=0.65]{asinkarehdecay.png}
%	\caption{
%		Energy of solution plotted against time for $h(t)=e^{-t}$ and $a(x)=\sin^2(x)$.
%	}\label{sinkare,hdecay}
%\end{figure}

\noindent Comparing Figure \ref{sinkare,hdecay} and Figure \ref{h=exp(-t)-damped}, we observe that the new choice of the damping coefficient leads to faster energy decay. 

%\begin{figure}
%	\centering
%	\includegraphics[scale=0.68]{asinkarehsinli.png}
%	\caption{
%		Energy of solution plotted against time for $h(t)=\sin(\frac{t}{5})$ and $a(x)=\sin^2(x)$ .
%	}\label{a=sin2, h=sin}
%\end{figure}

\noindent In Figure \ref{a=sin2, h=sin}, the energy initially decreases and then oscillates around a nonzero level, the amplitude of oscillations do not seem to get smaller in contrast with Figure \ref{h=sinli-damped}.

%\begin{figure}
%	\centering
%	\includegraphics[scale=0.6]{asinkareh5tli.png}
%	\caption{
%		Energy of solution plotted against time for $h(t)=\frac{5t}{t+1}$ and $a(x)=\sin^2(x)$.
%	}\label{a=sinkare,h=5tli}
%\end{figure}

\noindent Figure \ref{a=sinkare,h=5tli} shows that, when $h(t)$ asymptotically approaches a constant, so does the energy.  This property is similar to that of Figure \ref{h=5tli-damped} except that the oscillatory behavior observed in Figure \ref{h=5tli-damped} does not occur here.

%\begin{figure}
%	\centering
%	\includegraphics[scale=0.6]{asinkarehsqrt.png}
%	\caption{
%		Energy of solution plotted against time for $h(t)=\sqrt{t}$ and $a(x)=\sin^2(x)$.
%	}\label{a=sinkare,h=sqrt}
%\end{figure}

\noindent Figure \ref{a=sinkare,h=sqrt} exhibits energy growth similar to Figure \ref{h=sqrt-damped}, when $h(t)$ is an increasing function,. However, the oscillatory behavior observed in Figure  \ref{h=sqrt-damped} does not appear here.

%$\bullet$ Lastly, we examine the scenario where the damping coefficient is a polynomial.\\
\subsubsection*{Case 3: $a(x)=(x+1)^2$}
Here, we examine the scenario where the damping coefficient is a quadratic polynomial.  We repeat the experiments and find that the resulting behaviors are very similar to those of Case 2. However, in terms of the energy levels, the damping effect seems to be much stronger compared to Case 2. 

\begin{figure}[H]
	\centering
	\subfloat[Energy Plot, $h(t)=e^{-t}$, $a(x)=(x+1)^2$. \label{a_large=(x+1)kare,h_decay}]{%
		\resizebox*{7cm}{!}{\includegraphics{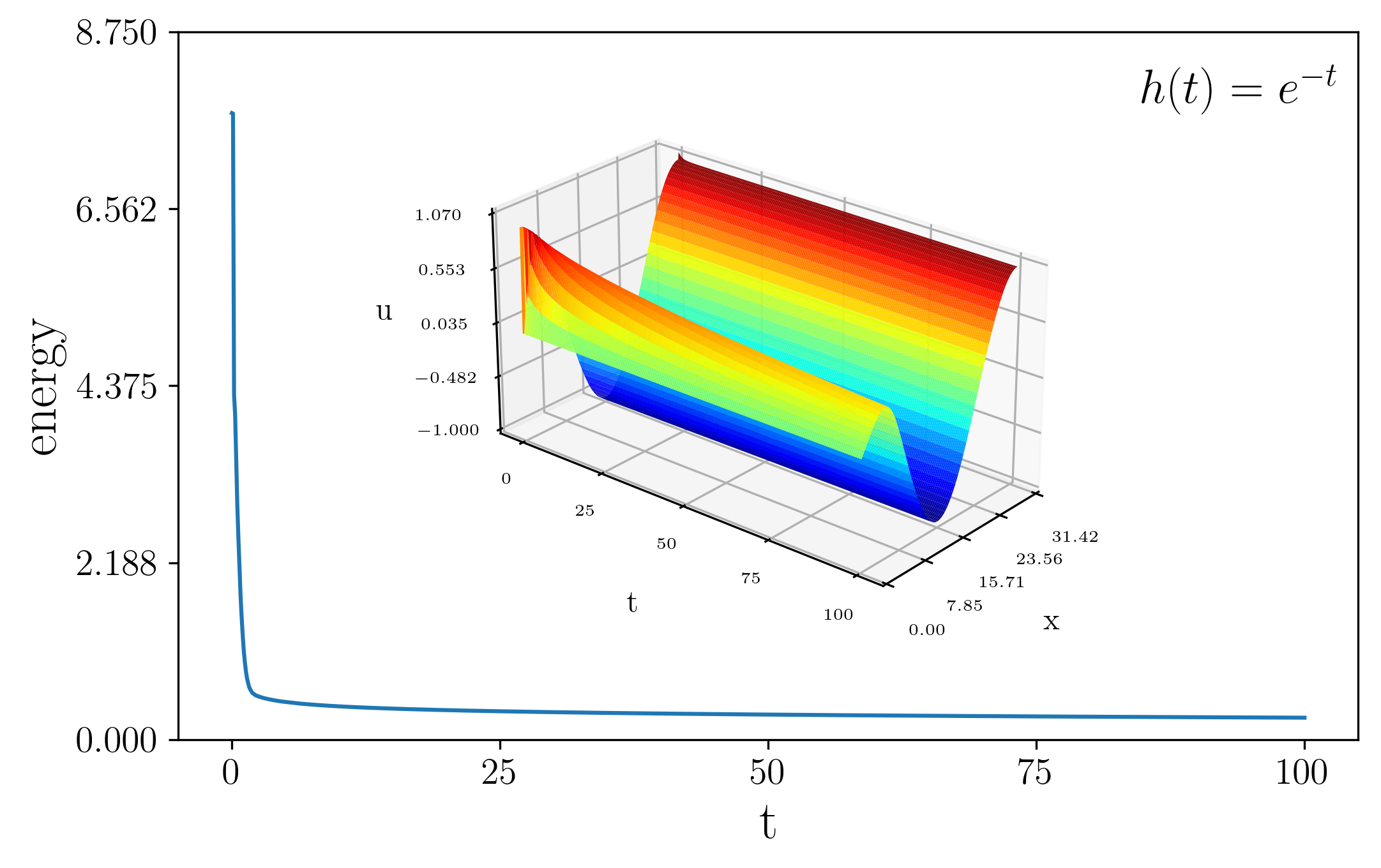}}}\hspace{5pt}
	\subfloat[Energy Plot, $h(t)=\sin(\frac{t}{5})$, $a(x)=(x+1)^2$ \label{}]{%
		\resizebox*{7cm}{!}{\includegraphics{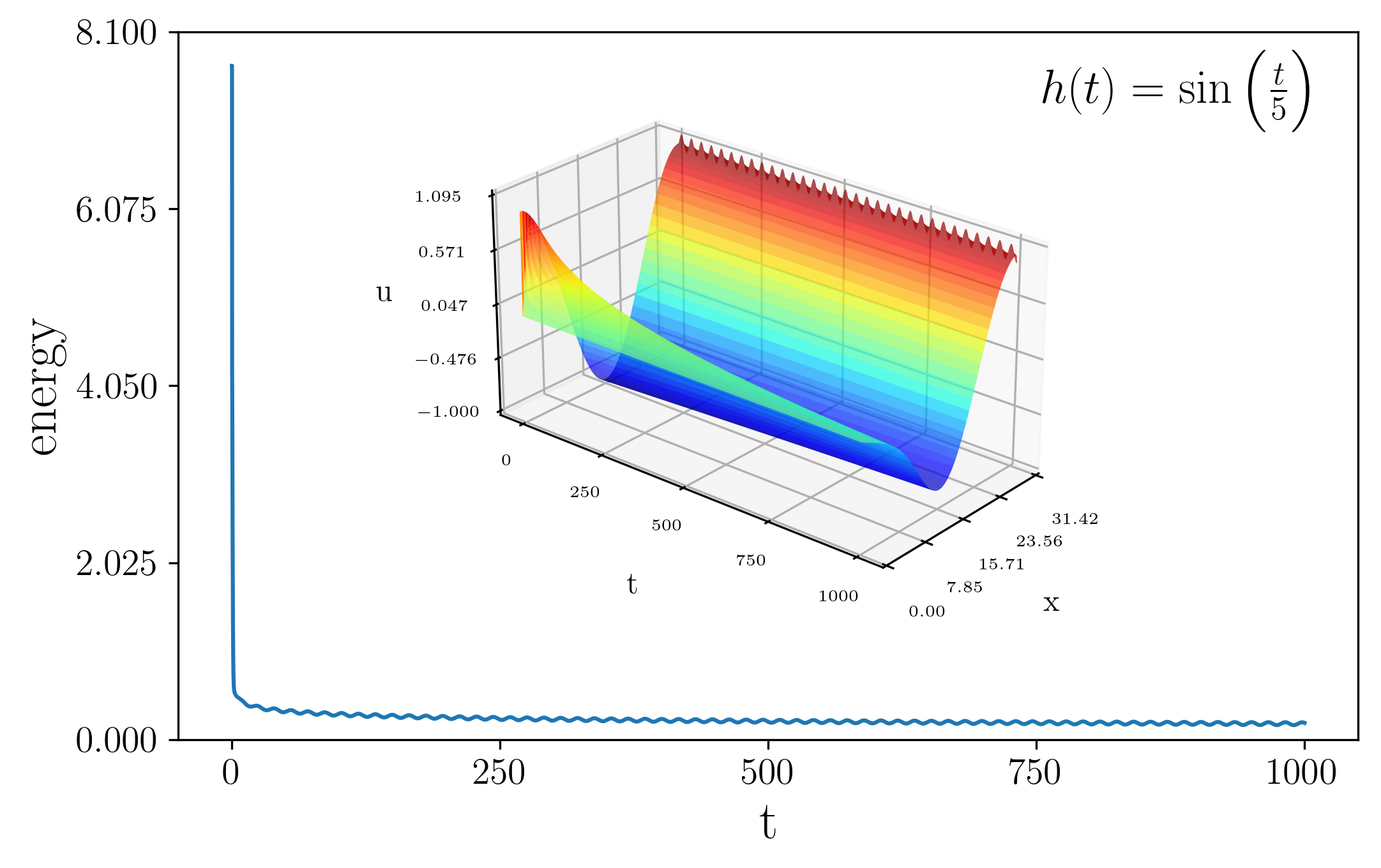}}}\\
	\subfloat[Energy Plot, $h(t)=\frac{5t}{t+1}$, $a(x)=(x+1)^2$. \label{a_large=(x+1)kare,h=5t}]{%
		\resizebox*{7cm}{!}{\includegraphics{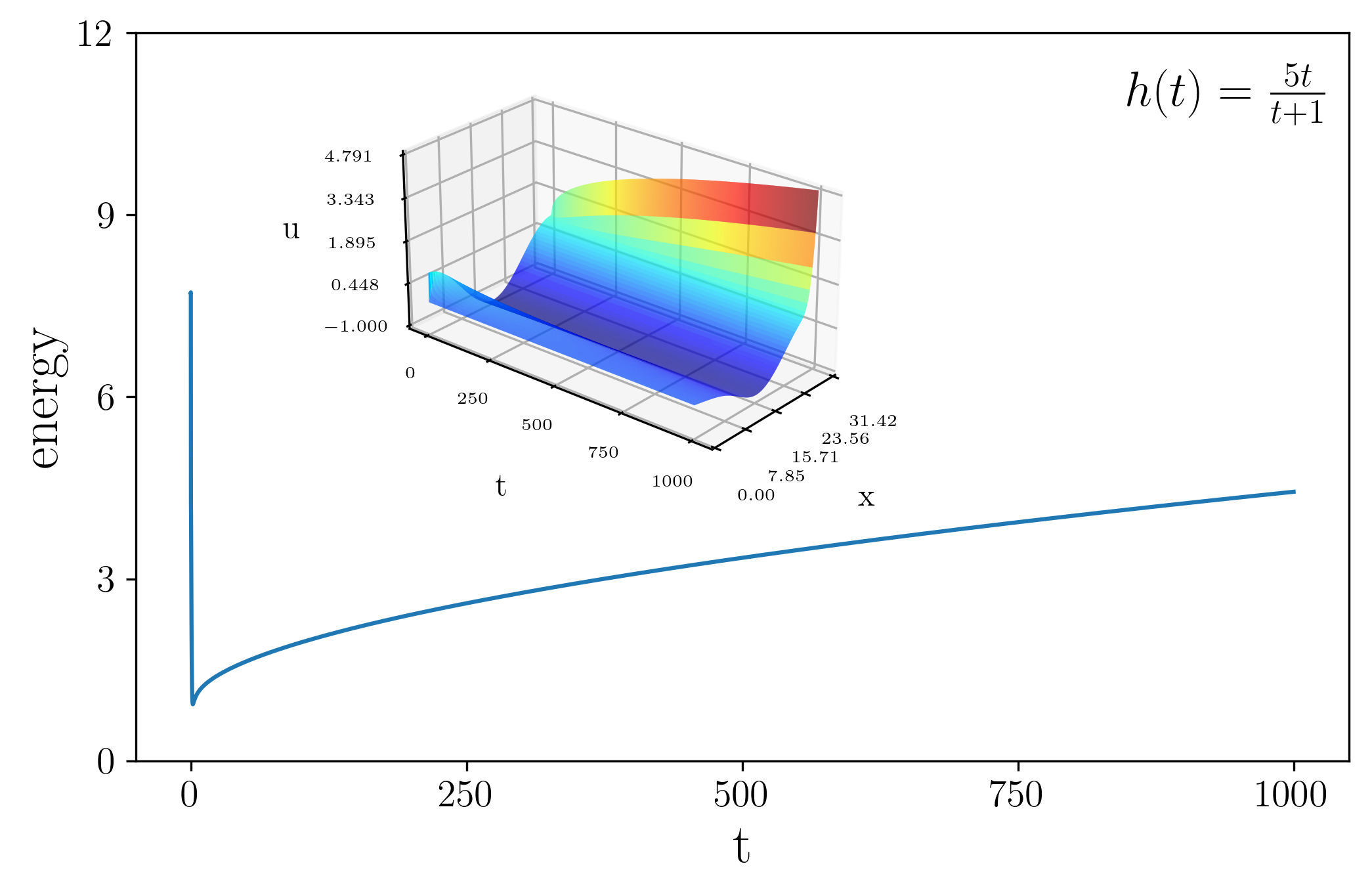}}}
	\subfloat[Energy Plot, $h(t)=\sqrt{t}$, $a(x)=(x+1)^2$ . \label{a_large,h=sqrt}]{%
		\resizebox*{7cm}{!}{\includegraphics{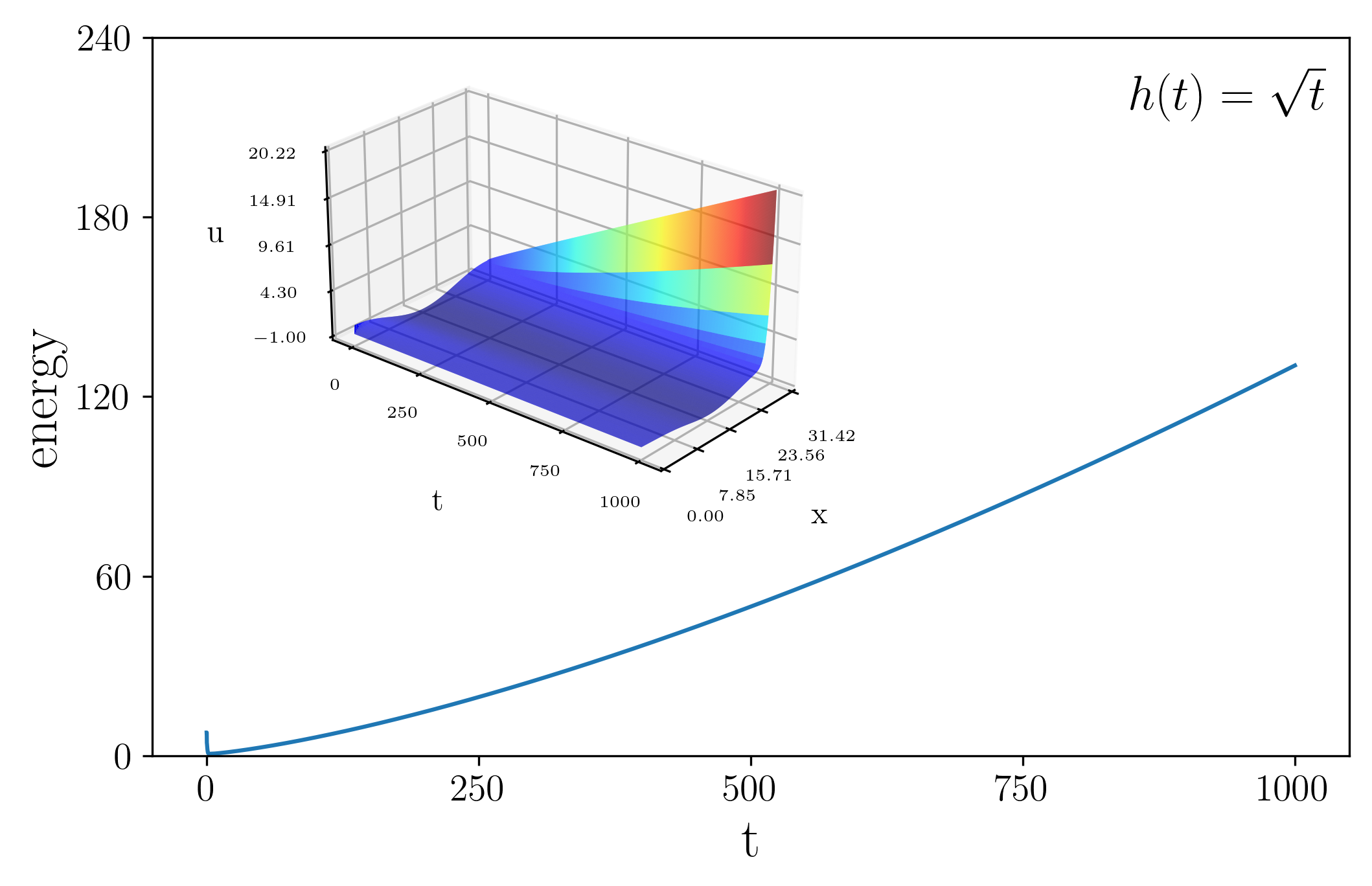}}}\\
	\caption{Linear waves subject to external Neumann manipulation and quadratic damping.} \label{sample-figure3}
\end{figure}

%\begin{figure}
%	\centering
%	\includegraphics[scale=0.7]{alargex1karehdecay.png}
%	\caption{
%		Energy of solution plotted against time for $h(t)=e^{-t}$ and $a(x)=(x+1)^2$.
%	}\label{a_large=(x+1)kare,h_decay}
%\end{figure}
%
%\begin{figure}
%	\centering
%	\includegraphics[scale=0.7]{alargex1karehsinli.png}
%	\caption{
%		Energy of solution plotted against time for $h(t)=\sin(\frac{t}{5})$ and $a(x)=(x+1)^2$.
%	}\label{}
%\end{figure}
%
%\begin{figure}
%	\centering
%	\includegraphics[scale=0.7]{alargex1kareh5t.png}
%	\caption{
%		Energy of solution plotted against time for $h(t)=\frac{5t}{t+1}$ and $a(x)=(x+1)^2$.
%	}\label{a_large=(x+1)kare,h=5t}
%\end{figure}
%
%\begin{figure}
%	\centering
%	\includegraphics[scale=0.7]{alargex1karehsqrt.png}
%	\caption{
%		Energy of solution plotted against time for $h(t)=\sqrt{t}$ and $a(x)=(x+1)^2$.
%	}\label{a_large,h=sqrt}
%\end{figure}

\section{Numerical experiments for 1D wave equation with viscoelasticity}
\subsection{Numerical solution}
In this section we construct the numerical 
solution of the following 1D  linear viscoelastic wave equation  by using the Crank-Nicolson method. We consider the following ibvp
\begin{align}
	&u_{tt}=F(t,u_{xx}):=u_{xx}- \int_{0}^{t}g(t-s)u_{xx}(s)ds,\quad x\in (0,L), t\in(0,T)\label{main}\\
	&u(x,0)=\varphi(x),\label{i1}\\
	&u_{t}(x,0)=\psi(x),\label{i2}\\
	&u(0,t) =f(t),\label{dirichlet}\\
	&u_x(L,t)= h(t).\label{neumann}\\
\end{align}
We adapt to the same space-time discretization of the space-time domain $[0,L]\times [0,T]$ that we defined in Section \ref{NumSolDamped}.  The solution $u=u(x,t)$ will be sought at the mesh points. We introduce the mesh function $u_i^n$, which approximates the exact solution at the mesh point $(x_i,t_n)$ for $i=0,…,N_x-1$ and $n=0,…,N_t-1$.

We define two vectors ($U^0=[U^0_i]_{1\le i\le N_x-2}^T, U^1=[U^1_i]_{1\le i\le N_x-2}^T$) by setting
$U_i^0=\varphi(x_i)$ and
$ U_i^1=U_i^0+  \psi(x_i)\triangle t, $ where the latter is motivated from the numerical version of (\ref{i2}), i.e.,
$\displaystyle \frac{U_i^1-U_i^0}{\triangle t}= \psi(x_i).$

Setting $v=u_t$ and $v_t=F$ and applying the Crank-Nicolson method to both equations, after some algebraic manipulations, one can obtain the following discretization for the function satisfying \eqref{main}:
\begin{align}\label{crank}
	\frac{U_i^{n+1}- 2U_i^n+U_i^{n-1}}{(\triangle t)^2}=\frac{1}{4} \bigg( F_i^{n+1}+2F_i^{n}+F_i^{n-1}\bigg),
\end{align} where $F_i^n=F$ evaluated for $i,n$ and $U_i^n.$
Before applying the Crank-Nicolson method, we replace the integral term by using the  trapezoid rule with its discrete form as follows:
\begin{align}
	\begin{split}
	\int_{0}^{t}g(t-s)u_{xx}(s)ds\approx & \frac{\triangle t}{2}\Big(g(0)u_{xx}(t_n)+g(t_n)u_{xx}(0)\Big)\\
	 &+ \sum_{m=1}^{n-1}g(t_n-t_m) u_{xx}(t_m)\triangle t.
	 \end{split}
\end{align}
Then, applying Crank-Nicolson method for the interior mesh points only, i.e, for  $i=1,…,N_x-2$, $n=1,...,N_t-2$,  (\ref{main}) can be written as follows:
\begin{align}\label{main-num}
	\begin{split}
		\frac{ U_i^{n+1}-2 U_i^{n}+ U_i^{n-1}}{(\triangle t)^2}=& \frac{1}{4} \frac{ U_{i+1}^{n+1}-2 U_i^{n+1}+ U_{i-1}^{n+1}}{(\triangle x)^2}+\frac{1}{2}\frac{ U_{i+1}^{n}-2 U_i^{n}+ U_{i-1}^{n}}{(\triangle x)^2}\\
		&+\frac{1}{4}\frac{ U_{i+1}^{n-1}-2 U_i^{n-1}+ U_{i-1}^{n-1}}{(\triangle x)^2} - I,
	\end{split}
\end{align}
where
\begin{align}\label{integralterm}
	\begin{split}
		I&:=\frac{\triangle t}{8} \bigg(g(0) \frac{ U_{i+1}^{n+1}-2 U_i^{n+1}+ U_{i-1}^{n+1}}{(\triangle x)^2}+g(t_{n+1}) \frac{ U_{i+1}^{0}-2 U_i^{0}+ U_{i-1}^{0}}{(\triangle x)^2}                     \bigg)\\
		&+ \frac{\triangle t}{4}\overbrace{g(t_{n+1}-t_{n})}^{g(t_1)}\frac{ U_{i+1}^{n}-2 U_i^{n}+ U_{i-1}^{n}}{(\triangle x)^2}\\
		&+ \frac{\triangle t}{4}\overbrace{g(t_{n+1}-t_{n-1})}^{g(t_2)}\frac{ U_{i+1}^{n-1}-2 U_i^{n-1}+ U_{i-1}^{n-1}}{(\triangle x)^2}\\
		&+\frac{\triangle t}{4} \bigg(g(0) \frac{ U_{i+1}^{n}-2 U_i^{n}+ U_{i-1}^{n}}{(\triangle x)^2}+g(t_n) \frac{ U_{i+1}^{0}-2 U_i^{0}+ U_{i-1}^{0}}{(\triangle x)^2} \bigg)  \\
		&+\frac{\triangle t}{2}\overbrace{g(t_n-t_{n-1})}^{g(t_1)}\frac{ U_{i+1}^{n-1}-2 U_i^{n-1}+ U_{i-1}^{n-1}}{(\triangle x)^2}\\
		&+\frac{\triangle t}{8} \bigg(g(0) \frac{ U_{i+1}^{n-1}-2 U_i^{n-1}+ U_{i-1}^{n-1}}{(\triangle x)^2}+g(t_{n-1}) \frac{ U_{i+1}^{0}-2 U_i^{0}+ U_{i-1}^{0}}{(\triangle x)^2} \bigg)\\
		&+\frac{\triangle t}{4}\sum_{m=1}^{n-2}\Big( g(t_{n+1}-t_m)+2g(t_n-t_m)+g(t_{n-1}-t_m)  \Big)\frac{U_{i+1}^{m}-2 U_i^{m}+ U_{i-1}^{m}} {(\triangle x)^2}.
	\end{split}
\end{align}
Solving for the $(n+1)$'th time step  in (\ref{main-num}), we obtain
\begin{align}\label{interiorsoln form}
	\begin{split}
		&-(1-a)U_{i+1}^{n+1}+2\Big(1+\frac{2}{r^2}-a\Big)U_i^{n+1}-(1-a)U_{i-1}^{n+1}=\\
		&2\bigg((1-b)U_{i+1}^{n}-2\Big(1-\frac{2}{r^2}-b\Big)U_i^{n}+(1-b)U_{i-1}^{n}\bigg)\\
		&+\bigg((1-c)U_{i+1}^{n-1}-2\Big(1+\frac{2}{r^2}-c\Big)U_i^{n-1}+(1-c)U_{i-1}^{n-1}\bigg)\\
		&- \delta \Big( g(t_{n-1})+2g(t_n)+g(t_{n+1})\Big)(U_{i+1}^{0}-2 U_i^{0}+ U_{i-1}^{0})\\
		&-2\delta \sum_{m=1}^{n-2} \Big( g(t_{n-1}-t_m)+2g(t_n-t_m)+g(t_{n+1}-t_m)\Big)(U_{i+1}^{m}-2 U_i^{m}+ U_{i-1}^{m}),
	\end{split}
\end{align}
where 
$
r :=\frac{\triangle t}{\triangle x},
\delta :=\frac{1}{2}\triangle t, 
a := \delta g(0), 
b := \delta(g(0)+g(t_1)),
c := \delta (g(0)+4g(t_1)+2g(t_2)).
$

We incorporate boundary data into the construction of the numerical solution by setting
\begin{equation}\label{U0Dirichlet-damped}U_0^n=f(t_n), \quad n=0,...,N_t-1.
\end{equation}
Writing $i = 1$ in (\ref{interiorsoln form}), we obtain
\begin{align}\label{i=1}
	\begin{split}
		&-(1-a)U_{2}^{n+1}+2\Big(1+\frac{2}{r^2}-a\Big)U_{1}^{n+1}-(1-a)U_{0}^{n+1}\\
		&=2\bigg((1-b)U_{2}^{n}-2\Big(1-\frac{2}{r^2}-b\Big)U_{1}^{n}+(1-b)U_{0}^{n}\bigg)\\
		&+\bigg((1-c)U_{2}^{n-1}-2\Big(1+\frac{2}{r^2}-c\Big)U_{1}^{n-1}+(1-c)U_{0}^{n-1}\bigg)\\
		&- \delta \Big( g(t_{n-1})+2g(t_n)+g(t_{n+1})\Big)(U_{2}^{0}-2 U_{1}^{0}+ U_{0}^{0})\\
		&-2\delta \sum_{m=1}^{n-2} \Big( g(t_{n-1}-t_m)+2g(t_n-t_m)+g(t_{n+1}-t_m)\Big)(U_{2}^{m}-2 U_{1}^{m}+ U_{0}^{m}).
	\end{split}
\end{align}
Now, we will numerically construct the Neumann boundary data by using the backward $O(h^2)$ approximation, which is defined by
$$U_x\approx  \frac{U_{i-2}^n-4 U_{i-1}^n+3U_{i}^n}{2 \triangle x}.$$
Therefore, the numerical version of (\ref{neumann}) leads us to define the quantities
\begin{align}\label{neumanndata}
	U_{N_x-1}^n= \frac{1}{3}( 2\triangle x h(t_n)- U_{N_x-3}^n+4U_{N_x-2}^n).
\end{align}
Writing $i=N_x-2$ in (\ref{interiorsoln form}) we obtain
\begin{align}\label{forL}
	\begin{split}
		&-(1-a)U_{N_x-1}^{n+1}+2\Big(1+\frac{2}{r^2}-a\Big)U_{N_x-2}^{n+1}-(1-a)U_{N_x-3}^{n+1}\\
		&=2\bigg((1-b)U_{N_x-1}^{n}-2\Big(1-\frac{2}{r^2}-b\Big)U_{N_x-2}^{n}+(1-b)U_{N_x-3}^{n}\bigg)\\
		&+\bigg((1-c)U_{N_x-1}^{n-1}-2\Big(1+\frac{2}{r^2}-c\Big)U_{N_x-2}^{n-1}+(1-c)U_{N_x-3}^{n-1}\bigg)\\
		&- \delta \Big( g(t_{n-1})+2g(t_n)+g(t_{n+1})\Big)(U_{N_x-1}^{0}-2 U_{N_x-2}^{0}+ U_{N_x-3}^{0})\\
		&-2\delta \sum_{m=1}^{n-2} \Big( g(t_{n-1}-t_m)+2g(t_n-t_m)+g(t_{n+1}-t_m)\Big)(U_{N_x-1}^{m}-2 U_{N_x-2}^{m}+ U_{N_x-3}^{m}).
	\end{split}
\end{align}
In (\ref{forL}), $U_{N_x-1}^{n+1},U_{N_x-1}^{n},U_{N_x-1}^{n-1},U_{N_x-1}^{0},U_{N_x-1}^{m}$ are unknown but can be eliminated by using (\ref{neumanndata}) so that we derive
\begin{align*}\label{forLbecomes}
	\begin{split}
		& \Big(\frac{4}{r^2}+\frac{2}{3}(1-a)\Big)U_{N_x-2}^{n+1}-\frac{2}{3}(1-a)U_{N_x-3}^{n+1}\\
		&=2\bigg(\Big(\frac{4}{r^2}-\frac{2}{3}(1-b)\Big)U_{N_x-2}^{n}+\frac{2}{3}(1-b)U_{N_x-3}^{n}\bigg)\\
		&+\bigg( -\Big(\frac{4}{r^2}+\frac{2}{3}(1-c)\Big)U_{N_x-2}^{n-1}+\frac{2}{3}(1-c)U_{N_x-3}^{n-1}\bigg)\\
		&- \delta \Big( g(t_{n-1})+2g(t_n)+g(t_{n+1})\Big) \Big(\frac{2}{3}( \triangle x h_{N_x-1}^0-U_{N_x-2}^{0}+ U_{N_x-3}^{0} )\Big)\\
		&-2\delta \sum_{m=1}^{n-2} \Big( g(t_{n-1}-t_m)+2g(t_n-t_m)+g(t_{n+1}-t_m)\Big)\Big(\frac{2}{3}( \triangle h_{N_x-1}^m-U_{N_x-2}^{m}+ U_{N_x-3}^{m} )\Big)\\
		&+\frac{2 \triangle x}{3} \Big( (1-a)h_{N_x-1}^{n+1}+2(1-b)h_{N_x-1}^{n}+(1-c)h_{N_x-1}^{n-1} \Big).
	\end{split}
\end{align*}
We introduce the following matrices
\begin{align*}
	K =
	\begin{bmatrix}
		2(1+\frac{2}{r^2}-a)   &-(1-a)    &              &           &                                   & \textrm{\huge0} \\
		-(1-a)                 &\ddots    &\ddots        &           &                                   & \\
		&\ddots    &\ddots        &\ddots     &                                   & \\
		&          &\ddots        & \ddots    &\ddots                             & \\
		&          &              &-(1-a)     &2(1+\frac{2}{r^2}-a)               &-(1-a) \\
		\textrm{\huge0}         &          &              &           &-\frac{2}{3} (1-a)                 &(\frac{4}{r^2}+\frac{2}{3}(1-a))
	\end{bmatrix},
\end{align*}

\begin{align*}
	K' =
	\begin{bmatrix}
		-2(1-\frac{2}{r^2}-b)  & (1-b)    &              &           &                                     & \textrm{\huge0} \\
		(1-b)                  &\ddots    &\ddots        &           &                                     & \\
		&\ddots    &\ddots        &\ddots     &                                     & \\
		&          &\ddots        & \ddots    &\ddots                               & \\
		&          &              &(1-b)      &-2(1-\frac{2}{r^2}-b)                &(1-b) \\
		\textrm{\huge0}         &          &              &           &\frac{2}{3}(1-b)                     &(\frac{4}{r^2}-\frac{2}{3}(1-b) )
	\end{bmatrix},
\end{align*}

\begin{align*}
	K'' =
	\begin{bmatrix}
		-2(1+\frac{2}{r^2}-c)  & (1-c)    &              &           &                                     & \textrm{\huge0} \\
		(1-c)                   &\ddots    &\ddots        &           &                                     & \\
		&\ddots    &\ddots        &\ddots     &                                     & \\
		&          &\ddots        & \ddots    &\ddots                               & \\
		&          &              &(1-c)      &-2(1+\frac{2}{r^2}-c )                &(1-c)  \\
		\textrm{\huge0}         &          &              &           &\frac{2}{3}(1-c)                    &(\frac{4}{r^2}+\frac{2}{3}(1-c))
	\end{bmatrix},
\end{align*}

\begin{align*}
	K''' =
	\begin{bmatrix}
		-2                & 1        &              &           &                  & \textrm{\huge0} \\
		1                &\ddots    &\ddots        &           &                  & \\
		&\ddots    &\ddots        &\ddots     &                  & \\
		&          &\ddots        & \ddots    &\ddots            & \\
		&          &              &1    &-2           & 1 \\
		\textrm{\huge0}   &          &              &           &\frac{2}{3}       &-\frac{2}{3}
	\end{bmatrix}.
\end{align*} and the following vectors
\begin{equation}
	\mathbf{D}^n=
	\begin{bmatrix}
		{{D_1^n}} &
		0 &
		. &
		.&
		.&
		0
	\end{bmatrix}^T
\end{equation}
where
\begin{align*}
	D_1^n=\bigg( (4-a-&2b-c)-\delta (g(t_{n-1})+2g(t_n)+g(t_{n+1}))\\
	&-2\delta \sum_{m=1}^{n-2}\Big( g(t_{n-1}-t_m)+2g(t_n-t_m)+g(t_{n+1}-t_{m})\Big)\bigg) f(t^n).
\end{align*}
and
\begin{equation}
	\mathbf{N}^n=
	\begin{bmatrix}
		0&
		.&
		.&
		.&
		0&
		N_{N_x-2}^n
	\end{bmatrix},
\end{equation}
where
\begin{align*}
	N_{N_x-2}^n&=\frac{2 \triangle x}{3} \Big( (1-a)h_{N_x-1}^{n+1}+2(1-b)h_{N_x-1}^{n}+(1-c)h_{N_x-1}^{n-1} \Big)\\
	&- \frac{2\delta \triangle x}{3} \Big( g(t_{n-1})+2g(t_n)+g(t_{n+1})\Big) h_{N_x-1}^0\\
	&- \frac{4\delta \triangle x}{3}  \sum_{m=1}^{n-2} \Big( g(t_{n-1}-t_m)+2g(t_n-t_m)+g(t_{n+1}-t_m)\Big)  h_{N_x-1}^m.
\end{align*}
We define $U^n$'s recursively through the following equation:
\begin{align}\label{solnmatrixform}
	\begin{split}
		K&U^{n+1}=2K'U^{n}+K''U^{n-1}-\delta K'''\Big( g(t_{n-1})+2g(t_n)+g(t_{n+1})\Big) U^0 \\
		&- 2\delta  \sum_{m=1}^{n-2}K''' \Big( g(t_{n-1}-t_m)+2g(t_n-t_m)+g(t_{n+1}-t_m)\Big)U^m +  \mathbf{D}^n+ \mathbf{N}^n,
	\end{split}
\end{align} where $n=1,...,N_t-2.$
Finally, the numerical solution of the initial-boundary value problem is defined by the matrix
$$[u_{in}]_{0\le i\le N_x-1, 0\le n\le N_t-1},$$ where $u_{in}=U_i^n$, $0\le i\le N_x-1, 0\le n\le N_t-1.$ Note that $U_i^n$ can be found by using \eqref{solnmatrixform} for $i=1,...,N_x-2$ and $n=0,...,N_t-1$ while $U_0^n$ and $U_{N_x-1}^n$ can be found by using \eqref{U0Dirichlet-damped} and \eqref{neumanndata} for $n=0,...,N_t-1$.

\subsection{Numerical Simulations}
In this section, we present numerical examples with qualitative behaviors to provide deeper physical insight into how the energy of solutions can be influenced by the combined effects of the relaxation term and external manipulation. We also explore cases where the boundary data do not necessarily meet the assumptions required for establishing uniform decay rates. We consider the following initial-boundary value problem:

\begin{align}\label{nummodel}
	\left\{
	\begin{array}{ll}
		u_{tt}-u_{xx} &+ \int_{0}^{t}g(t-s)u_{xx}(s)ds=0    \\
		u(x,0)&=\sin(x/5)   ,\  u_{t}(x,0)=\cos(x/5) \\
		u(0,t) &=0  ,   \\
		u_x(L,t)&= h(t)
	\end{array}
	\right.
\end{align}
where $x\in[0,10\pi].$   We first consider the numerical solution when there is no external manipulation  (i.e., $h(t)=0$) to compare the dynamics later with the case of a nonzero external manipulation.

%
%\begin{figure}[!htb]
%\centering
%{%
%	\includegraphics[scale=0.6]{1ht0.png}
%	}\hspace{5pt}
%\caption{Energy of solution plotted against time for $g(t)=e^{-(t+1)}$ and $h(t)=0$ .} \label{h=0}
%\end{figure}

Figure \ref{h=0} shows that under homogeneous boundary conditions, energy decay is characterized by the decay property of the relaxation term. Compare this with Figure \ref{g=0}, where we plot the solution to the classical wave equation with no memory term, leading to conservation of energy.

\begin{figure}[H]
	\centering
	\includegraphics[scale=0.5]{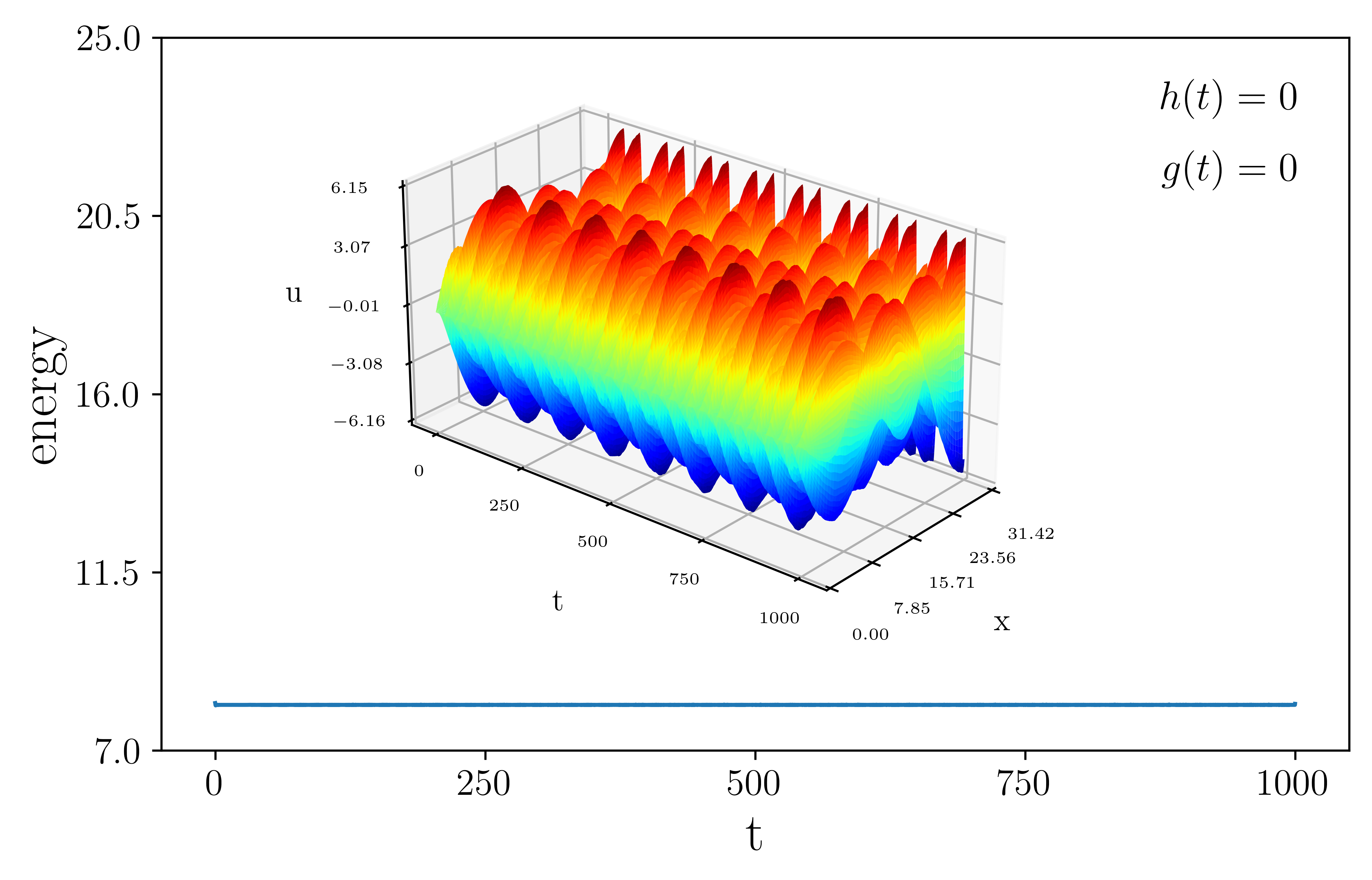}
	\caption{
		Classical wave solution, $g(t)=h(t)=0$.
	}\label{g=0}
\end{figure}
Next, we examine various types of Neumann manipulations, including cases where $h(t)$ exponentially decays, oscillates, remains asymptotically constant, or grows. We analyze the behavior of solutions under different scenarios, with the relaxation function taking one of these three forms: $$g(t)=e^{-(t+1)}, (t+1)^{-10}, \frac{0.2}{\ln^2(t+2) (t+1)}.$$

\subsubsection*{Case 1: $g(t)=e^{-(t+1)}$}
%\textbf{Case 1:}
We present simulations in Figure \ref{exp(-t)}-Figure \ref{exp(-0.01t)} where the external Neumann manipulation is an exponentially decaying function satisfying assumptions of Theorem  \ref{mainresultthm}.

\begin{figure}[H]
	\centering
	\subfloat[Energy Plot, $h(t)=0$, $g(t)=e^{-(t+1)}$. \label{h=0}]{%
		\resizebox*{7cm}{!}{\includegraphics{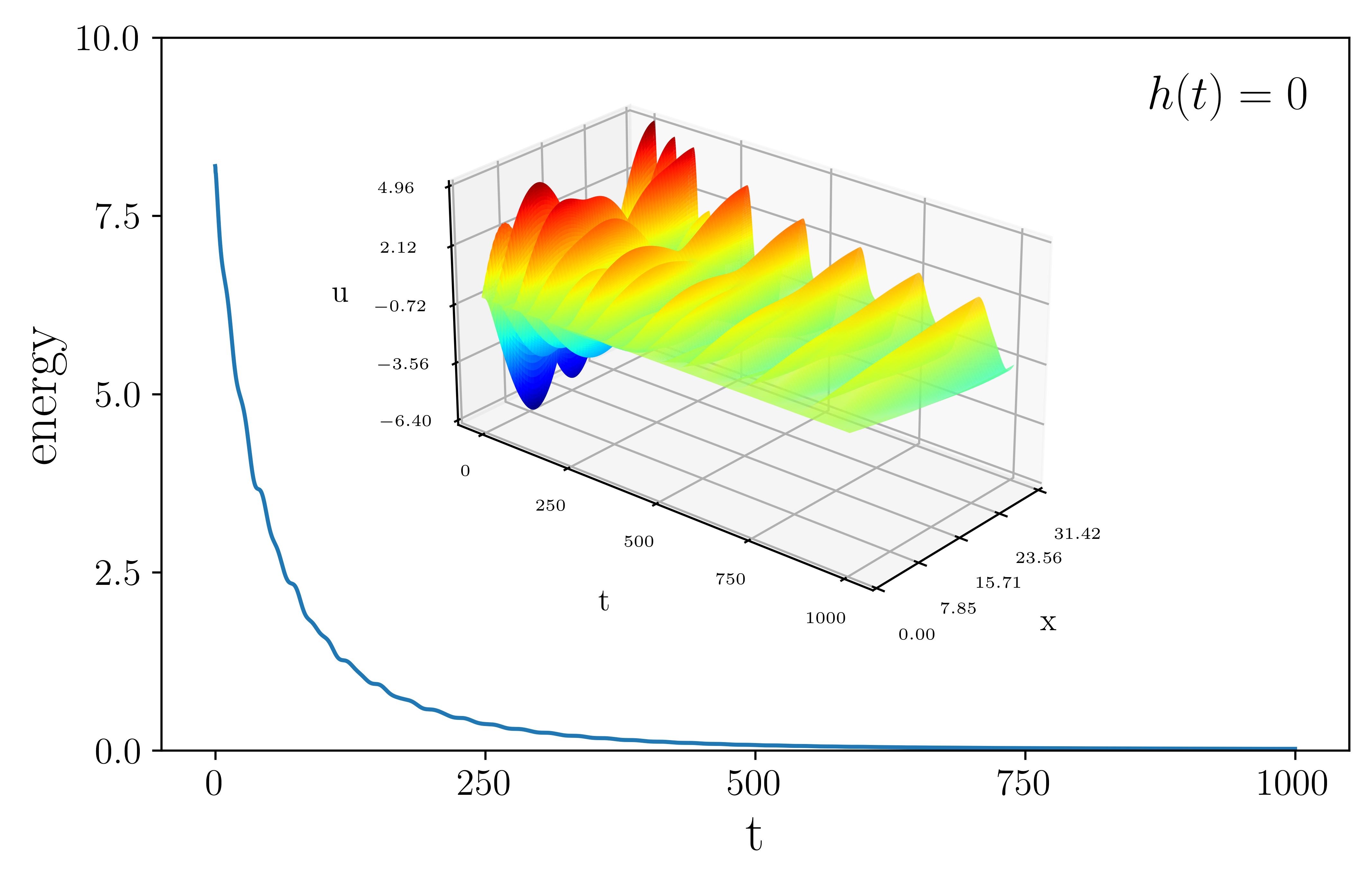}}}\hspace{5pt}
	\subfloat[Energy Plot, $h(t)=e^{-t}$, $g(t)=e^{-(t+1)}$\label{exp(-t)} ]{%
		\resizebox*{7cm}{!}{\includegraphics{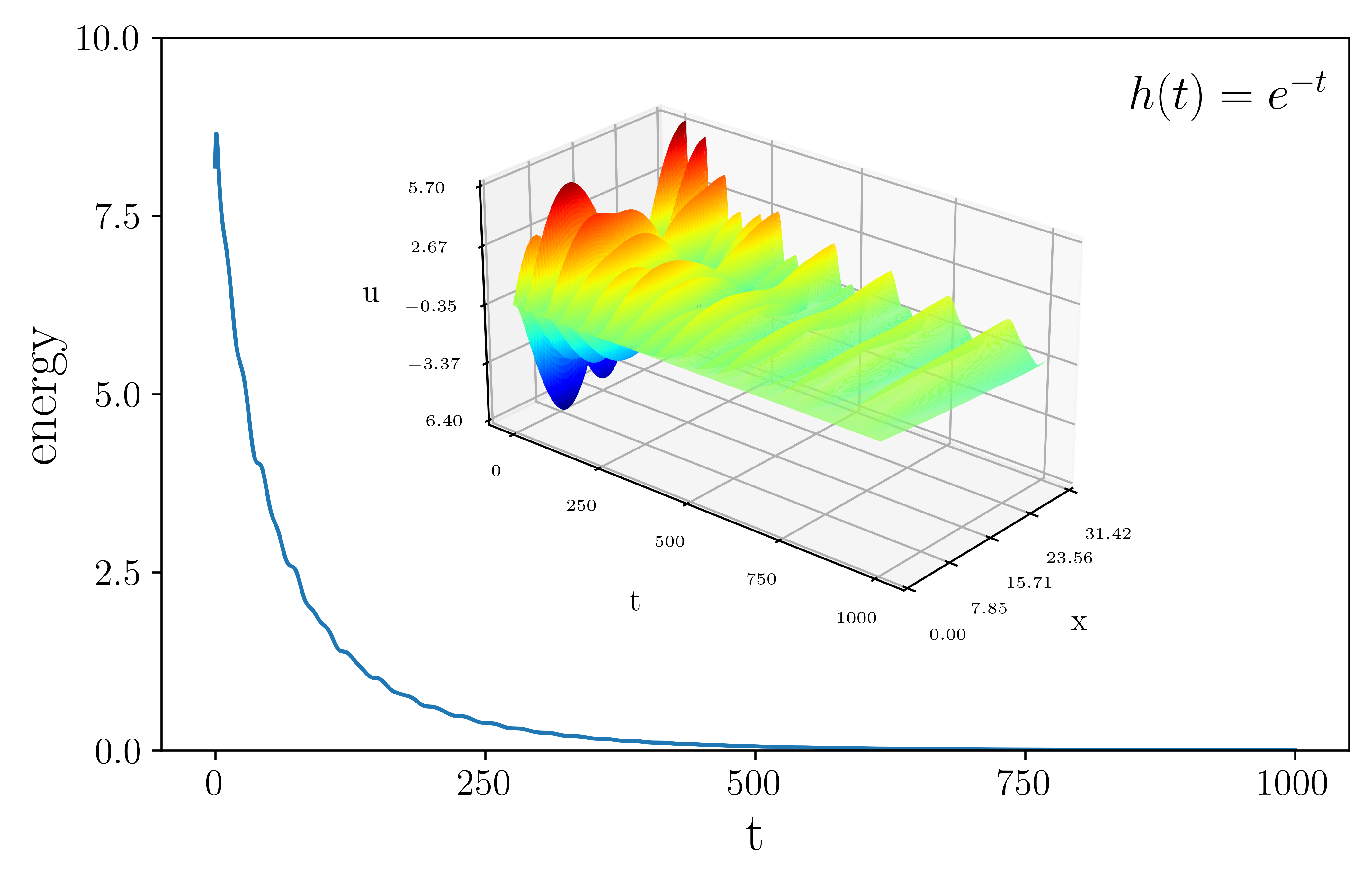}}}\\
	\subfloat[Energy Plot, $h(t)=e^{-0.01t}$, $g(t)=e^{-(t+1)}$. \label{exp(-0.01t)}]{%
		\resizebox*{7cm}{!}{\includegraphics{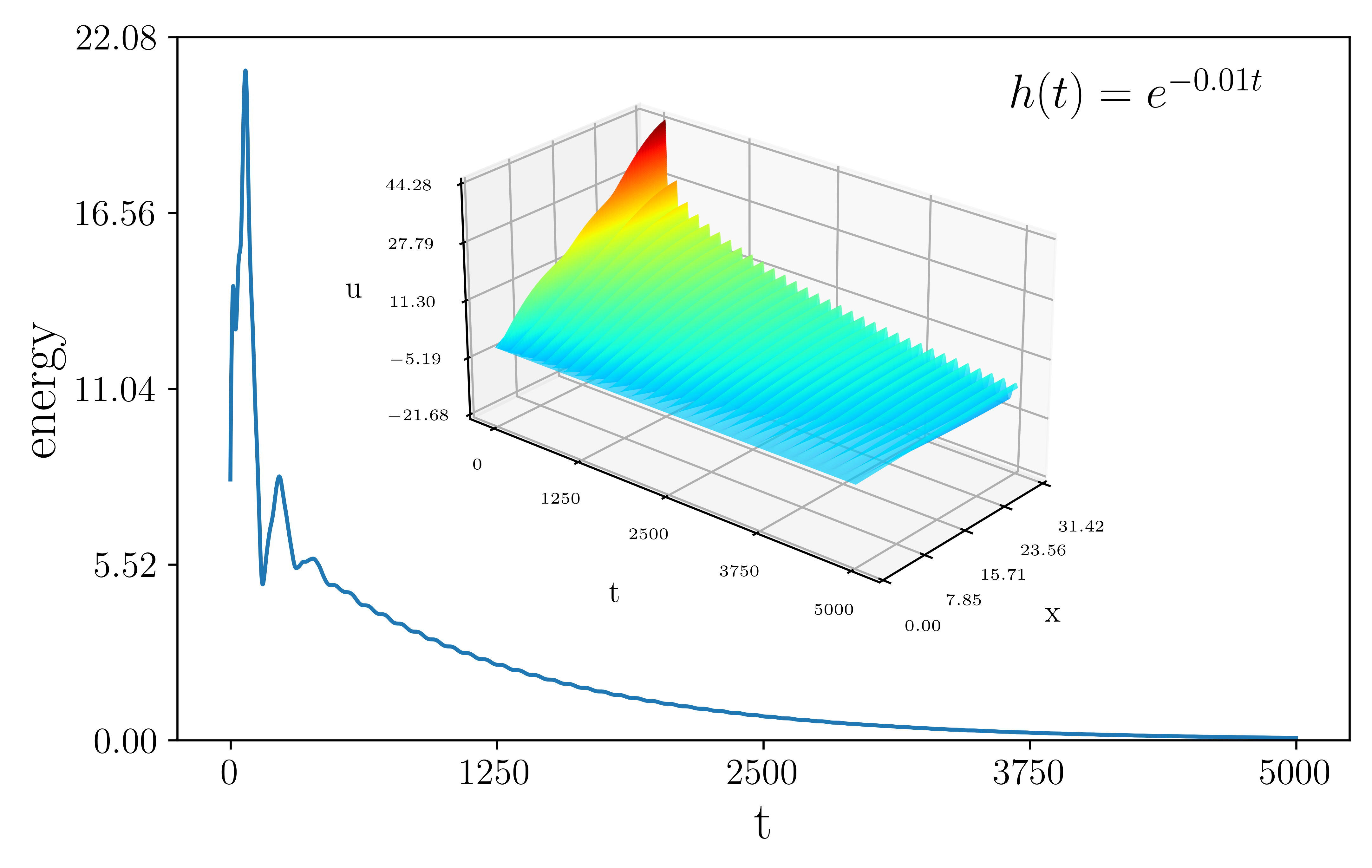}}}
	\subfloat[Energy Plot, $h(t)=\sin(t/5)$, $g(t)=e^{-(t+1)}$. \label{sin}]{%
		\resizebox*{7cm}{!}{\includegraphics{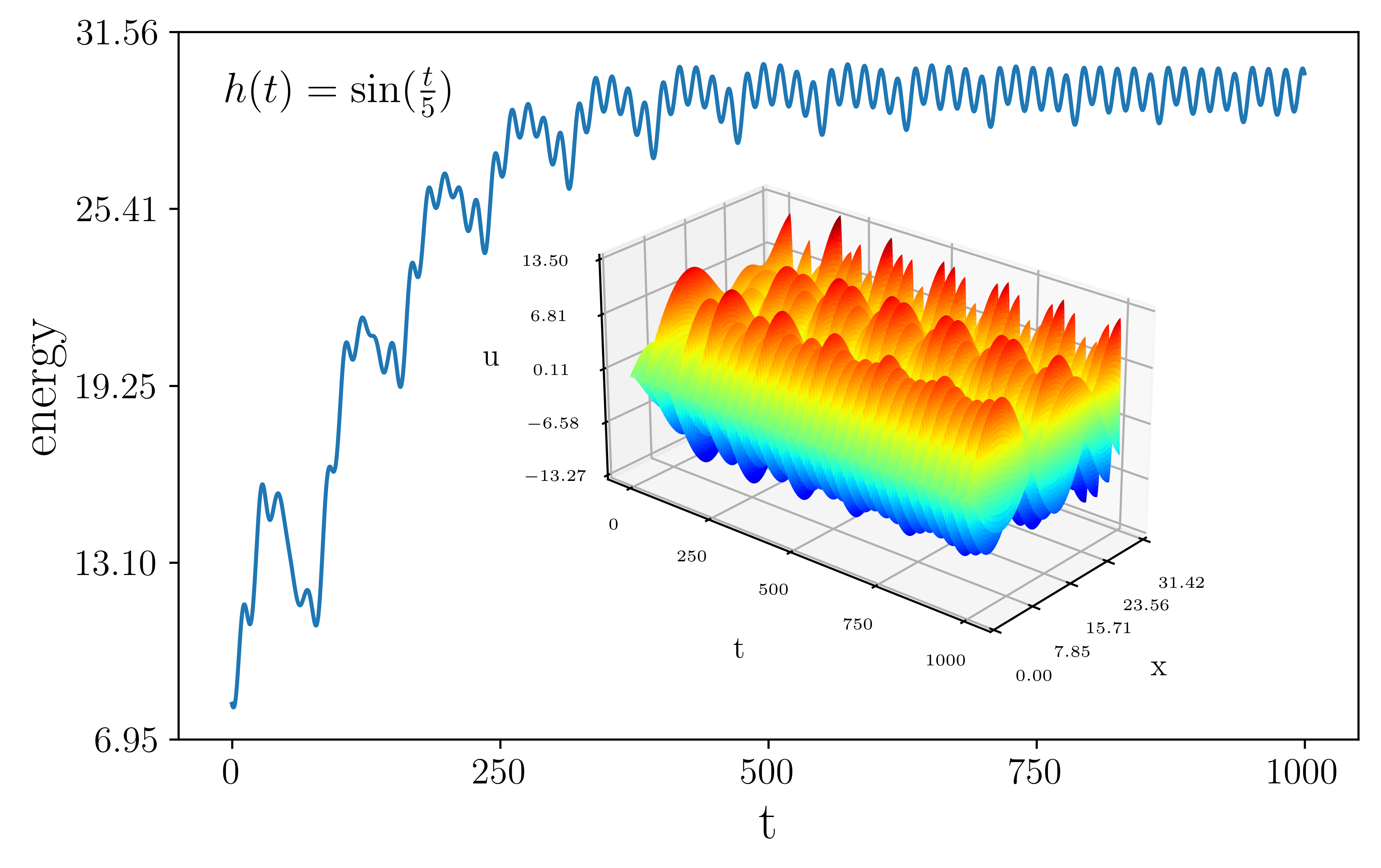}}}\\
			\subfloat[Energy Plot, $h(t)=\frac{5t}{t+1}$, $g(t)=e^{-(t+1)}$. \label{5t}]{%
			\resizebox*{7cm}{!}{\includegraphics{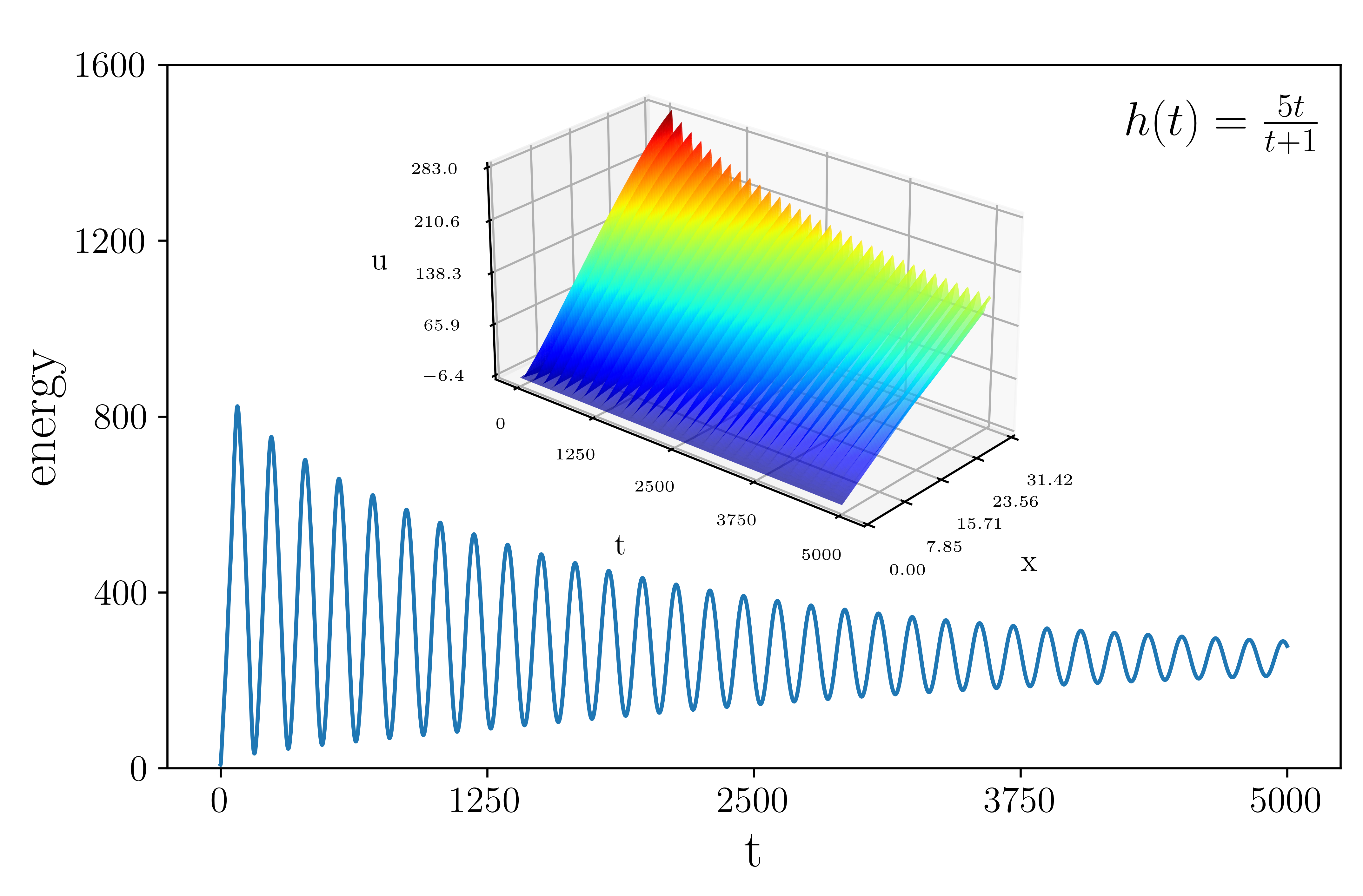}}}
				\subfloat[Energy Plot, $h(t)=\sqrt{t}$, $g(t)=e^{-(t+1)}$. \label{sqrt}]{%
				\resizebox*{7cm}{!}{\includegraphics{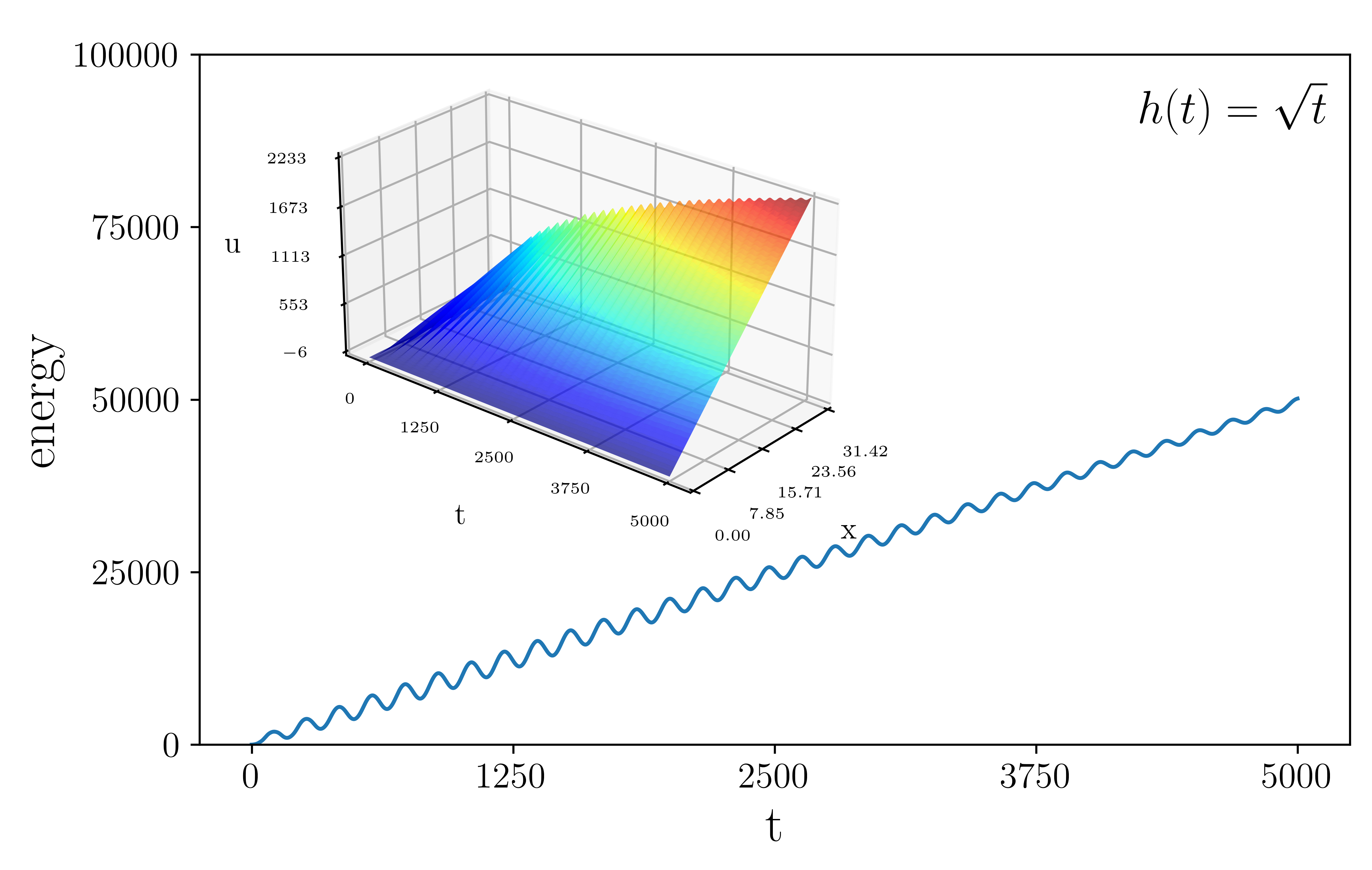}}}\\
	\caption{Viscoelastic waves subject to Neumann manipulation and exponentially decaying relaxation.} \label{sample-figure5}
\end{figure}

%\begin{figure}[!htb]
%	\centering
%	\includegraphics[scale=0.6]{3expt.png}
%	\caption{
%		Energy of solution plotted against time for $g(t)=e^{-(t+1)}$ and $h(t)=e^{-t}$.
%	}\label{exp(-t)}
%\end{figure}
%
%\begin{figure}[!htb]
%	\centering
%	\includegraphics[scale=0.6]{4exp001t.png}
%	\caption{
%		Energy of solution plotted against time for $g(t)=e^{-(t+1)}$ and  $h(t)=e^{-0.01t}$.
%	}\label{exp(-0.01t)}
%\end{figure}

\noindent As we see from Figure \ref{exp(-t)} and Figure \ref{exp(-0.01t)}, when $h(t)$  decays, so does the energy. This is parallel to the uniform decay estimate obtained in Theorem \ref{mainresultthm}. Figure \ref{exp(-t)} shows a much faster decay due to a more rapidly decaying external input. In both figures, viscoelastic effect is quite evident.  Next, we present examples with non-decaying external manipulations and observe that the energy does not decay in contrast with the case when the manipulation were zero and only the viscoelastic effect were present. Note that these examples violate assumption \eqref{decayh} of Theorem \ref{mainresultthm}.  Figure \ref{5t} shows that, when $h(t)$ asymptotically approaches a constant, so does the average of its energy. And viscoelastic effect becomes lessen over time as in Figure \ref{sin}.

%\begin{figure}[!htb]
%	\centering
%	\includegraphics[scale=0.60]{5periodic.png}
%	\caption{
%		Energy of solution plotted against time for $g(t)=e^{-(t+1)}$ and $h(t)=\sin(t/5)$.
%	}\label{sin}
%\end{figure}
%\noindent Figure \ref{sin} shows that, when $h(t)$ is an oscillating function, energy does not tend to zero, its average seems to approach a constant level for large times. Moreover, the impact of viscoelasticity on intervals with decreasing energy diminishes over time.
%\begin{figure}[!htb]
%	\centering
%	\includegraphics[scale=0.60]{6poly.png}
%	\caption{
%		Energy of solution plotted against time for $g(t)=e^{-(t+1)}$ and $h(t)=\frac{5t}{t+1}$.
%	}\label{5t}
%\end{figure}

%\begin{figure}[!htb]
%	\centering
%	\includegraphics[scale=0.60]{sqrt.png}
%	\caption{
%		Energy of solution plotted against time for $g(t)=e^{-(t+1)}$ and $h(t)=\sqrt{t}$.
%	}\label{sqrt}
%\end{figure}

\noindent In Figure \ref{sqrt}, we observe  it is possible to pump energy from the boundary by using an increasing function as the external manipulation. Moreover, Neumann manipulation clearly seems to dominate the viscoelastic effect out here.
\subsubsection*{Case 2: $g(t)=(t+1)^{-10}$}
Here, we examine the scenario where $g(t)$ is a polynomially decaying function.

\begin{figure}[H]
	\centering
	\subfloat[Energy Plot, $h(t)=e^{-t}$, $g(t)=(t+1)^{-10}$. \label{new_g(t)_poly10_decay}]{%
		\resizebox*{7cm}{!}{\includegraphics{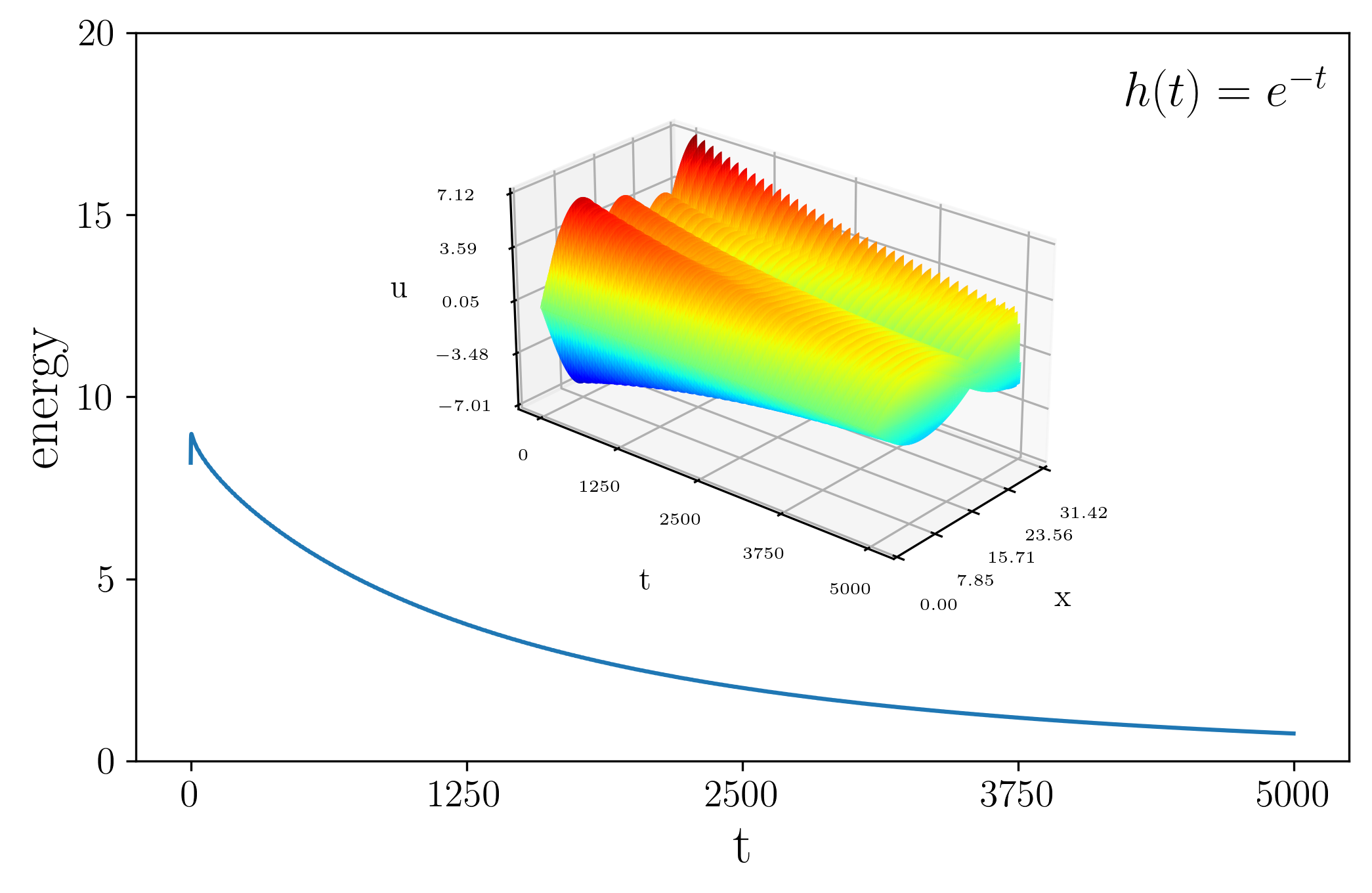}}}\hspace{5pt}
	\subfloat[Energy Plot, $h(t)=\sin(\frac{t}{5})$, $g(t)=(t+1)^{-10}$\label{g(t)_poly_10_sin_li}]{%
		\resizebox*{7cm}{!}{\includegraphics{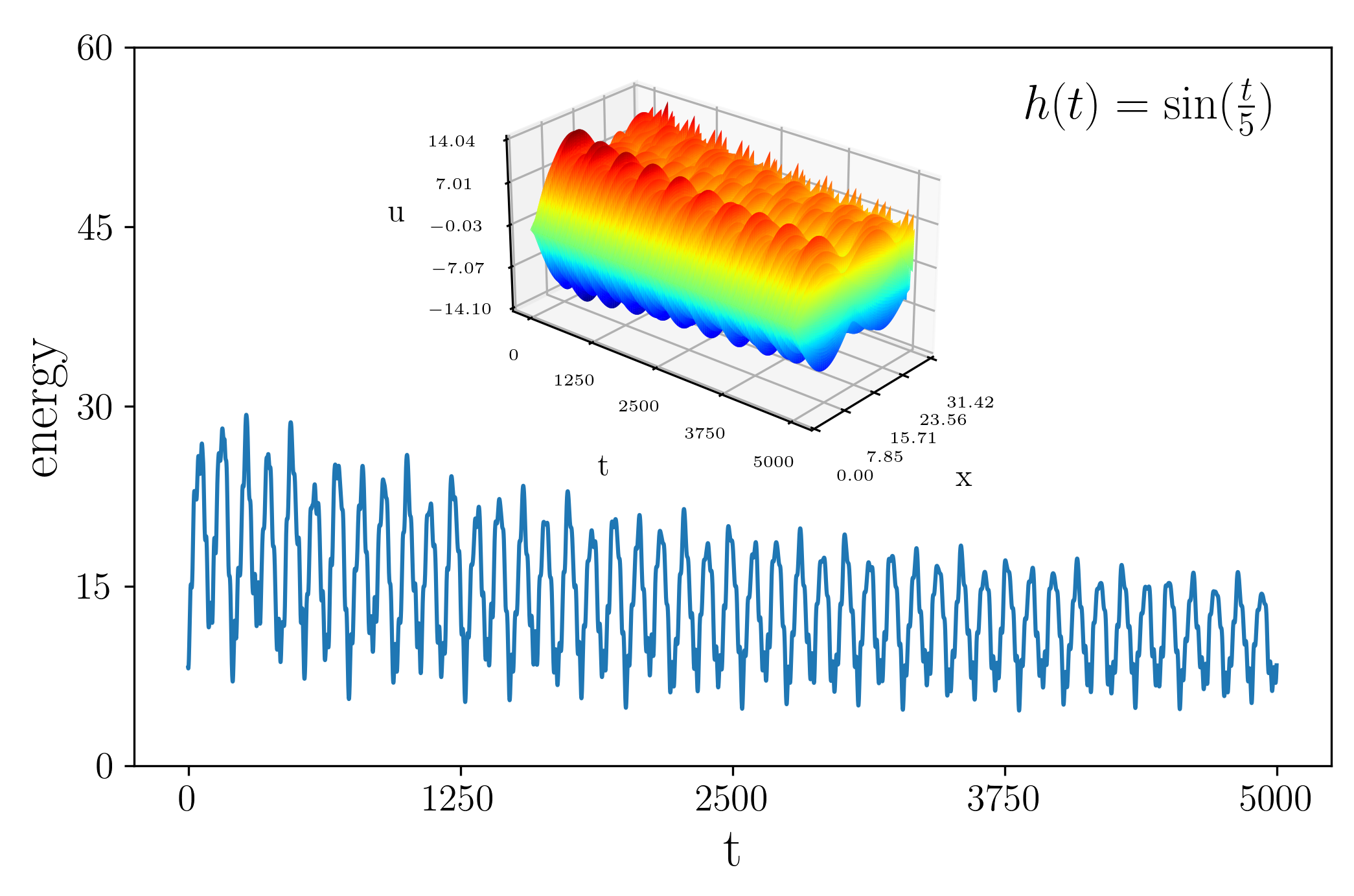}}}\\
	\subfloat[Energy Plot, $h(t)=\frac{5t}{t+1}$, $g(t)=(t+1)^{-10}$. \label{g(t)_poly_10_5t_li}]{%
		\resizebox*{7cm}{!}{\includegraphics{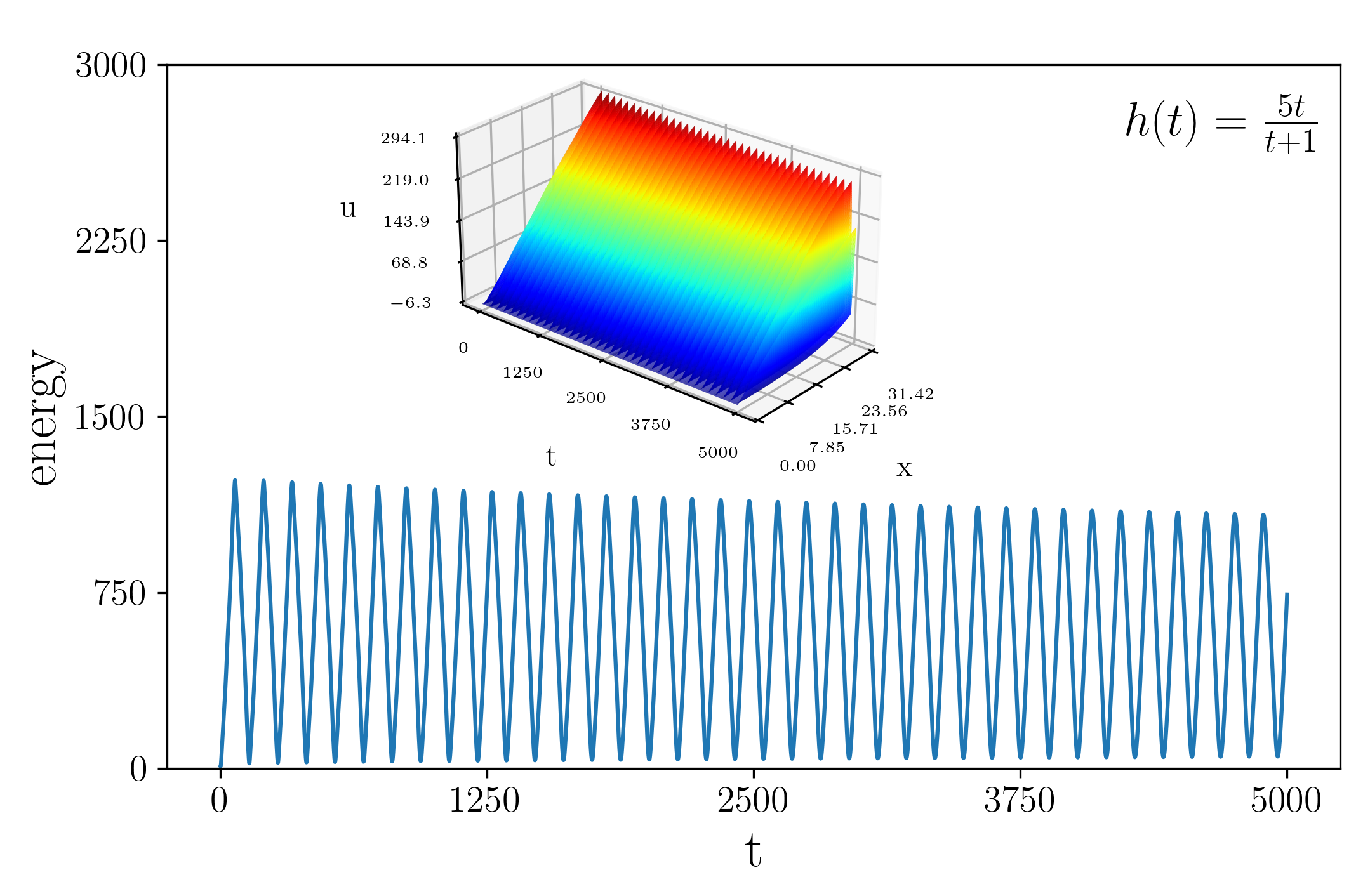}}}
	\subfloat[Energy Plot, $h(t)=\sqrt{t}$, $g(t)=(t+1)^{-10}$ .\label{g(t)_poly_10_sqrt}]{%
		\resizebox*{7cm}{!}{\includegraphics{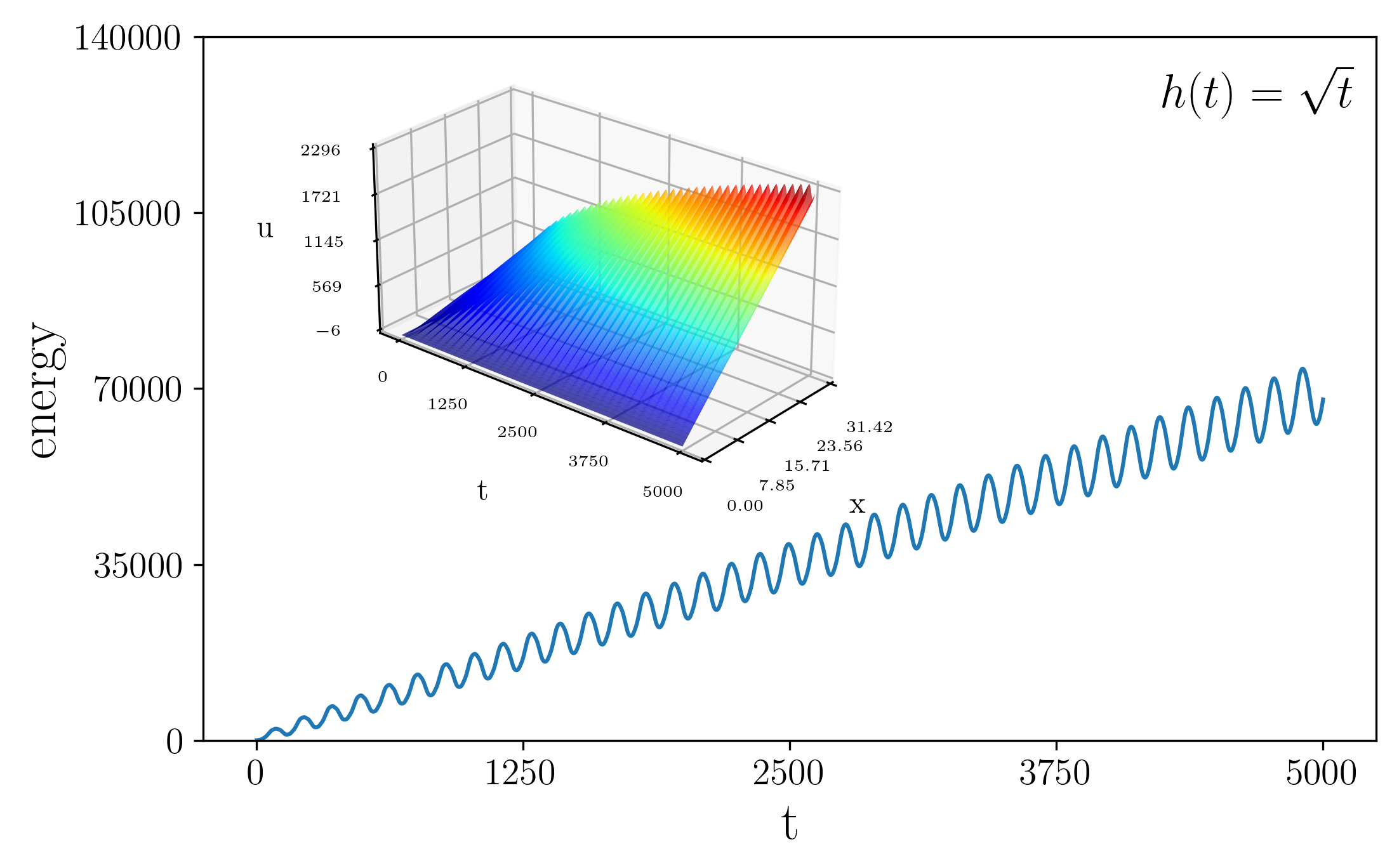}}}\\
	\caption{Viscoelastic waves subject to  Neumann manipulation and polynomially decaying relaxation.} \label{sample-figure6}
\end{figure}

%We also consider different relaxation functions  satisfying our assumptions (A1)-(A2) for the same model (\ref{nummodel}). 
%In the Figure \ref{differentg} and Figure \ref{differentg2}, we choose a ploynomially decaying $g(t)$, while $h(t)$ is still decaying at the same exponential rate. 

%\begin{figure}[!htb]
%	\centering
%	\includegraphics[scale=0.60]{newgtpoly10decay.png}
%	\caption{
%		Energy of solution plotted against time for $g(t)=(t+1)^{-10}$ and $h(t)=e^{-t}$ .
%	}\label{new_g(t)_poly10_decay}
%\end{figure}

\noindent Comparing Figure \ref{new_g(t)_poly10_decay} and Figure \ref{exp(-t)}, we observe that in Figure \ref{new_g(t)_poly10_decay} energy decays more slowly due to the fact that viscoelastic term is less effective because of the new choice of $g(t)$. Figure \ref{g(t)_poly_10_sin_li} - Figure \ref{g(t)_poly_10_sqrt} are analogues of Figure \ref{sin} - Figure \ref{sqrt} with the new choice of $g$.   There are certain similarities and differences to note: (i) oscillations in Figure \ref{g(t)_poly_10_sin_li} are more drastic than the ones in Figure \ref{sin} and also the average of energy has a decreasing character in the long run contrasting with the case of Figure \ref{g(t)_poly_10_sin_li}; (ii) the amplitude of oscillations in Figure \ref{5t} and Figure \ref{g(t)_poly_10_5t_li} both get smaller over time while the average of energy is stabilized around a certain level, however in the latter case these behaviors occur in much slower pace; (iii) Figure \ref{sqrt} shows a very similar long time behavior to that of Figure \ref{g(t)_poly_10_sqrt}, with the key difference being that oscillations fade in the former while they grow in the latter.

%\begin{figure}[!htb]
%	\centering
%	\includegraphics[scale=0.6]{newgtpoly10sinli.png}
%	\caption{
%		Energy of solution plotted against time for $h(t)=\sin(\frac{t}{5})$ and $g(t)=(t+1)^{-10}$.
%	}\label{g(t)_poly_10_sin_li}
%\end{figure}
%
%
%\begin{figure}[!htb]
%	\centering
%	\includegraphics[scale=0.6]{newgtpoly105tli.png}
%	\caption{
%		Energy of solution plotted against time for $h(t)=\frac{5t}{t+1}$ and $g(t)=(t+1)^{-10}$.
%	}\label{g(t)_poly_10_5t_li}
%\end{figure}
%
%
%
%\begin{figure}[!htb]
%	\centering
%	\includegraphics[scale=0.6]{newgtpoly10sqrt.png}
%	\caption{
%		Energy of solution plotted against time for $h(t)=\sqrt{t}$ and $g(t)=(t+1)^{-10}$.
%	}\label{g(t)_poly_10_sqrt}
%\end{figure}

\subsubsection*{Case 3: $g(t)=\frac{0.2}{\ln^2(t+2) (t+1)}$}
Finally, in this section, we examine the scenario where $g(t)$ is of a special logarithmic function, see Figure \ref{g(t)_ln_li}-Figure \ref{g(t)_ln_sqrt} below.  We repeat the experiments and find that the resulting behaviors are very similar to those of Case 1 perhaps with a slightly stronger stabilization effect here.

\begin{figure}[H]
	\centering
	\subfloat[Energy Plot, $h(t)=e^{-t}$, $g(t)=\frac{0.2}{\ln^2(t+2) (t+1)}$. \label{g(t)_ln_li}]{%
		\resizebox*{7cm}{!}{\includegraphics{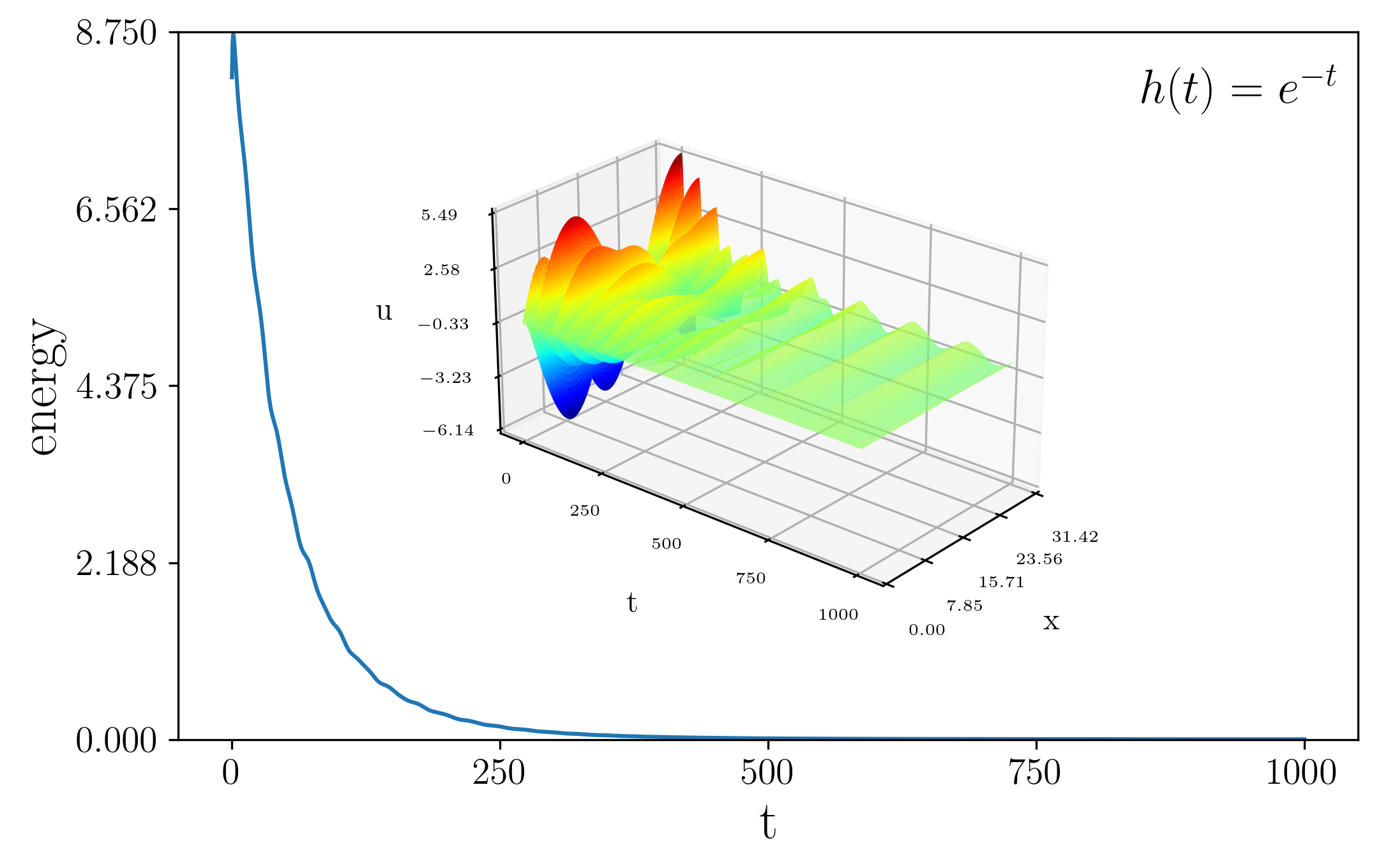}}}\hspace{5pt}
	\subfloat[Energy Plot, $h(t)=\sin(\frac{t}{5})$, $g(t)=\frac{0.2}{\ln^2(t+2) (t+1)}$\label{g(t)_ln_sin_li}]{%
		\resizebox*{7cm}{!}{\includegraphics{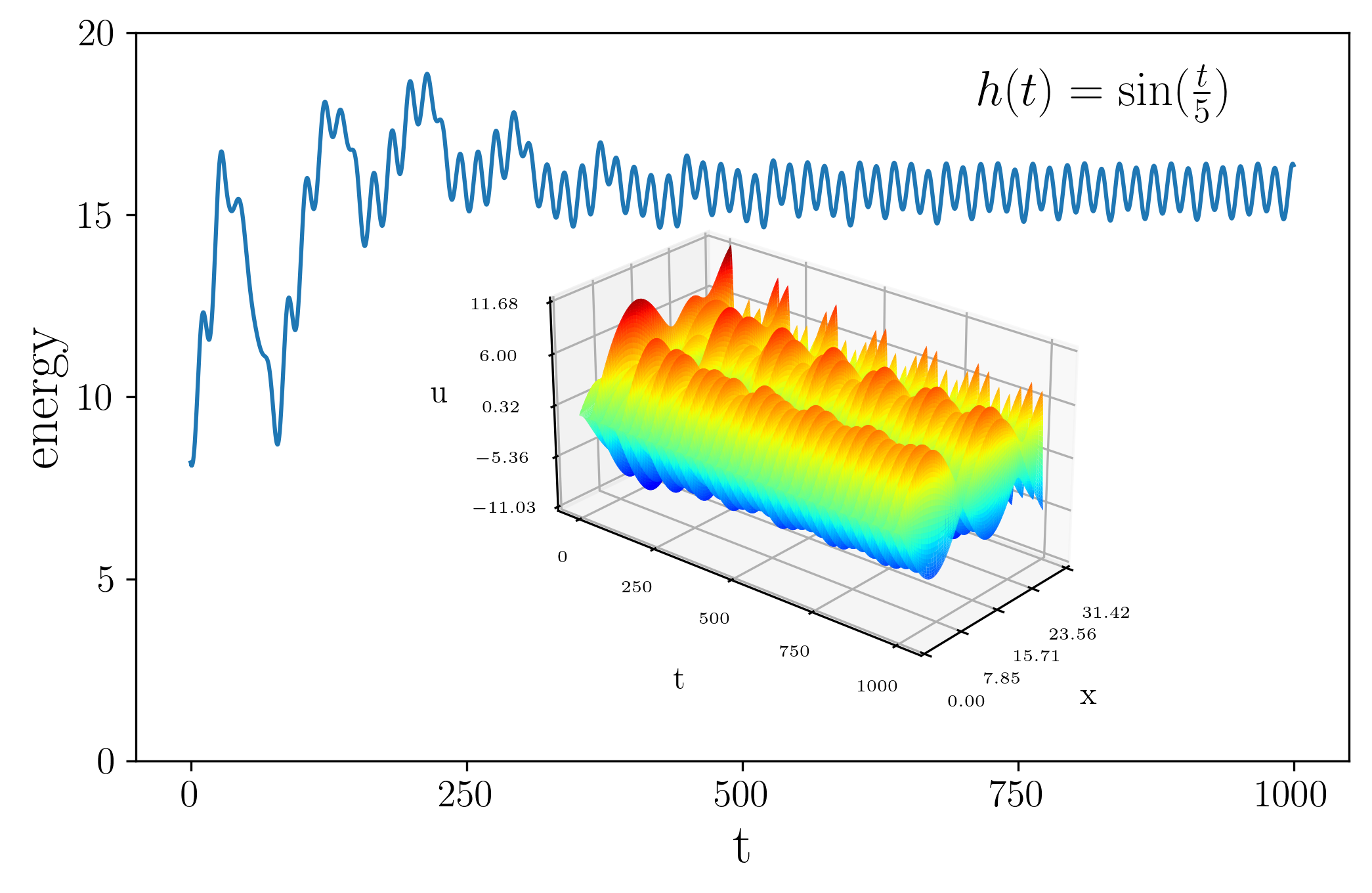}}}\\
	\subfloat[Energy Plot, $h(t)=\frac{5t}{t+1}$, $g(t)=\frac{0.2}{\ln^2(t+2) (t+1)}$. \label{g(t)_ln_5t_li}]{%
		\resizebox*{7cm}{!}{\includegraphics{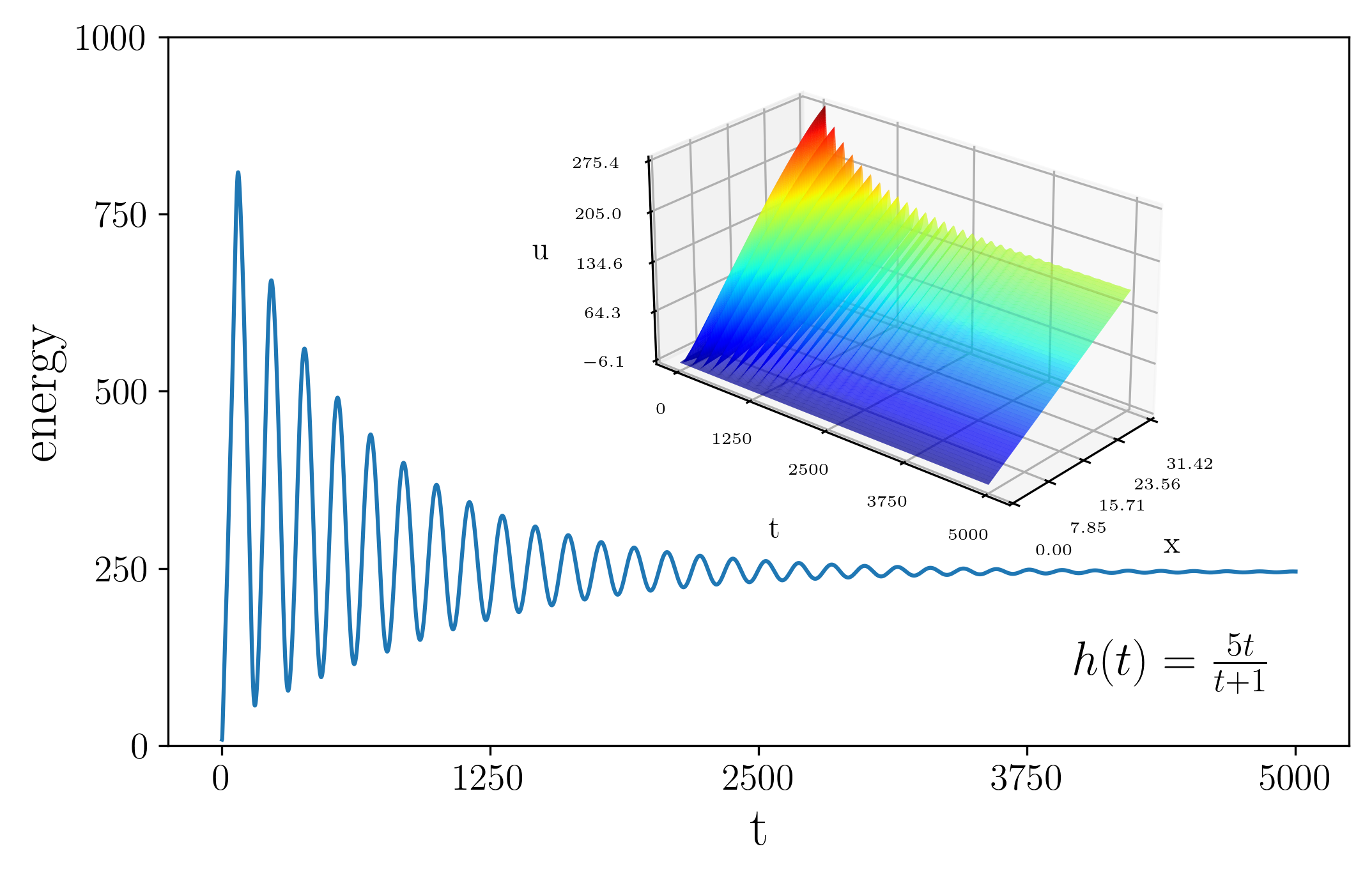}}}
	\subfloat[Energy Plot, $h(t)=\sqrt{t}$, $g(t)=\frac{0.2}{\ln^2(t+2) (t+1)}$ .\label{g(t)_ln_sqrt}]{%
		\resizebox*{7cm}{!}{\includegraphics{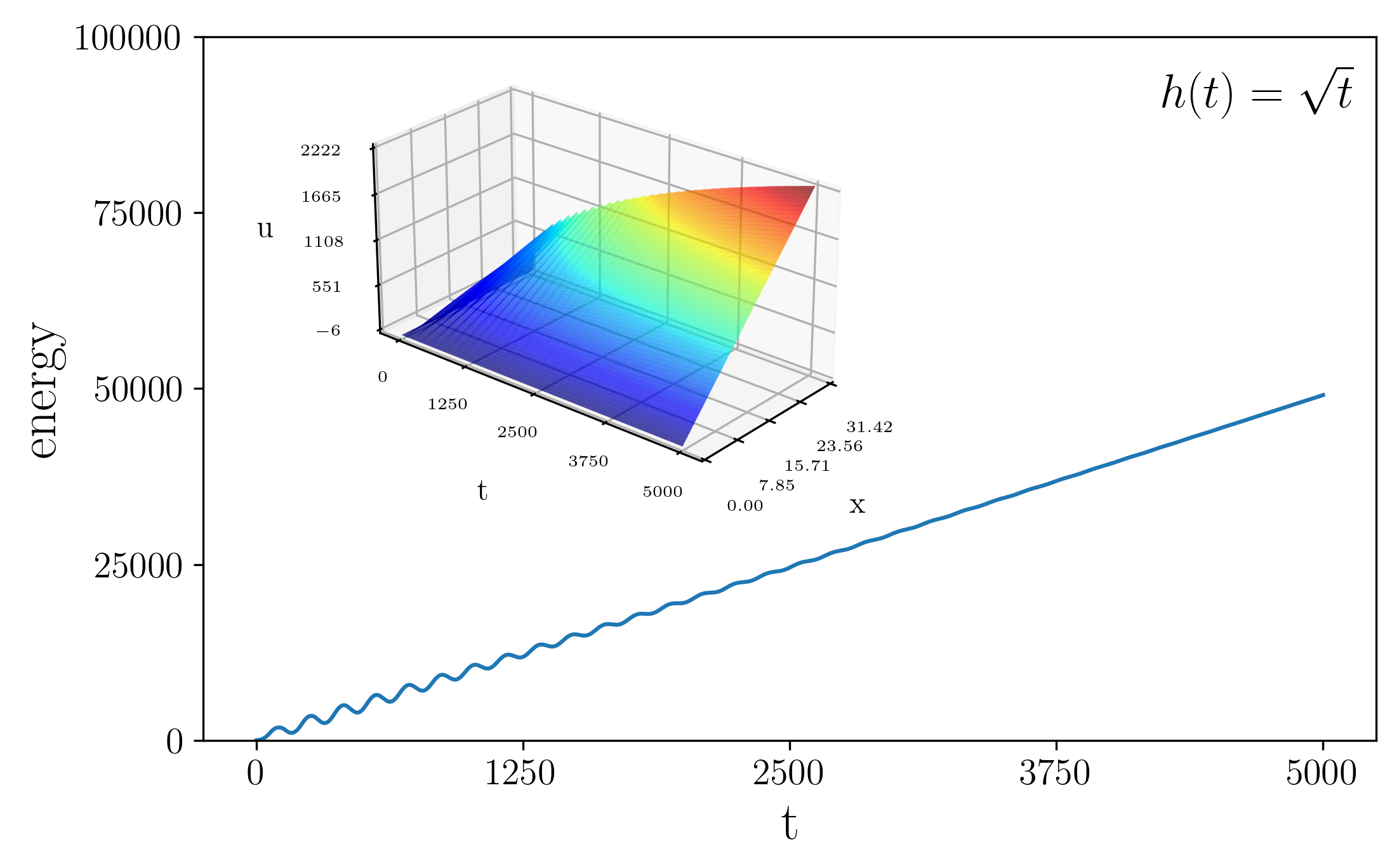}}}\\
	\caption{Viscoelastic waves subject to Neumann manipulation and logarithmic type relaxation.} \label{sample-figure7}
\end{figure}

%\begin{figure}[!htb]
%	\centering
%	\includegraphics[scale=0.6]{newgtlnli.png}
%	\caption{
%		Energy of solution plotted against time for $h(t)=e^{-t}$ and $g(t)=\frac{0.2}{\ln^2(t+2) (t+1)}$.
%	}\label{g(t)_ln_li}
%\end{figure}
%
%\begin{figure}[!htb]
%	\centering
%	\includegraphics[scale=0.6]{newgtlnsinli.png}
%	\caption{
%		Energy of solution plotted against time for $h(t)=\sin(\frac{t}{5})$ and $g(t)=\frac{0.2}{\ln^2(t+2) (t+1)}$.
%	}\label{g(t)_ln_sin_li}
%\end{figure}
%
%
%\begin{figure}[!htb]
%	\centering
%	\includegraphics[scale=0.6]{newgtln5tli.png}
%	\caption{
%		Energy of solution plotted against time for $h(t)=\frac{5t}{t+1}$ and $g(t)=\frac{0.2}{\ln^2(t+2) (t+1)}$.
%	}\label{g(t)_ln_5t_li}
%\end{figure}
%
%
%\begin{figure}[!htb]
%	\centering
%	\includegraphics[scale=0.6]{newgtlnsqrt.png}
%	\caption{
%		Energy of solution plotted against time for $h(t)=\sqrt{t}$ and $g(t)=\frac{0.2}{\ln^2(t+2) (t+1)}$.
%	}\label{g(t)_ln_sqrt}
%\end{figure}

\section{A few numerical experiments with rough initial data}
In the previous sections, we considered only smooth trigonometric initial data.  In this section, we consider the linear damped and viscoelastic wave equations with initial amplitude $u_0\in H^1(0,10\pi)$ and initial velocity $u_1\in L^2(0,10\pi)$. More precisely, we take $$u_0(x)=|x|,\quad u_1(x)=\left\{ 
\begin{array}{l l}
	1, \text{ for }\quad x\leq 5\pi\\
	0,  \text{ otherwise. }
\end{array}
\right.$$

We provide two numerical simulations below for each type of equation where the assumptions of respective theorems are satisfied and we observe the uniform decay properties implied by the respective main theorems.

\begin{figure}[H]
	\centering
	\subfloat[Energy Plot, $h(t)=e^{-t}$, $a(x)=e^{-x}$. \label{}]{%
		\resizebox*{7cm}{!}{\includegraphics{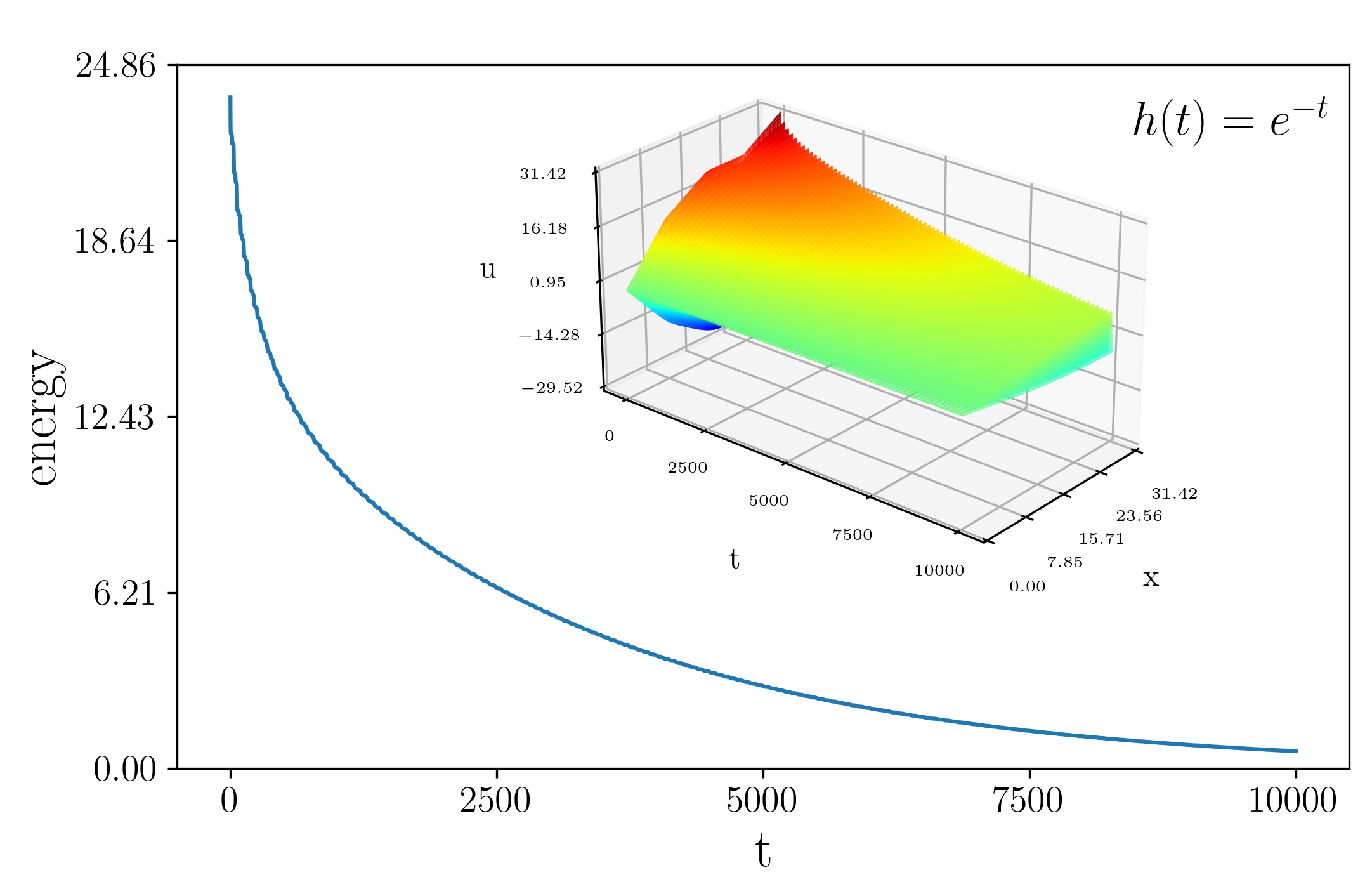}}}\hspace{5pt}
	\subfloat[Energy Plot, $h(t)=\sin(\frac{t}{5})$, $g(t)=e^{-(t+1)}$ \label{}]{%
		\resizebox*{7cm}{!}{\includegraphics{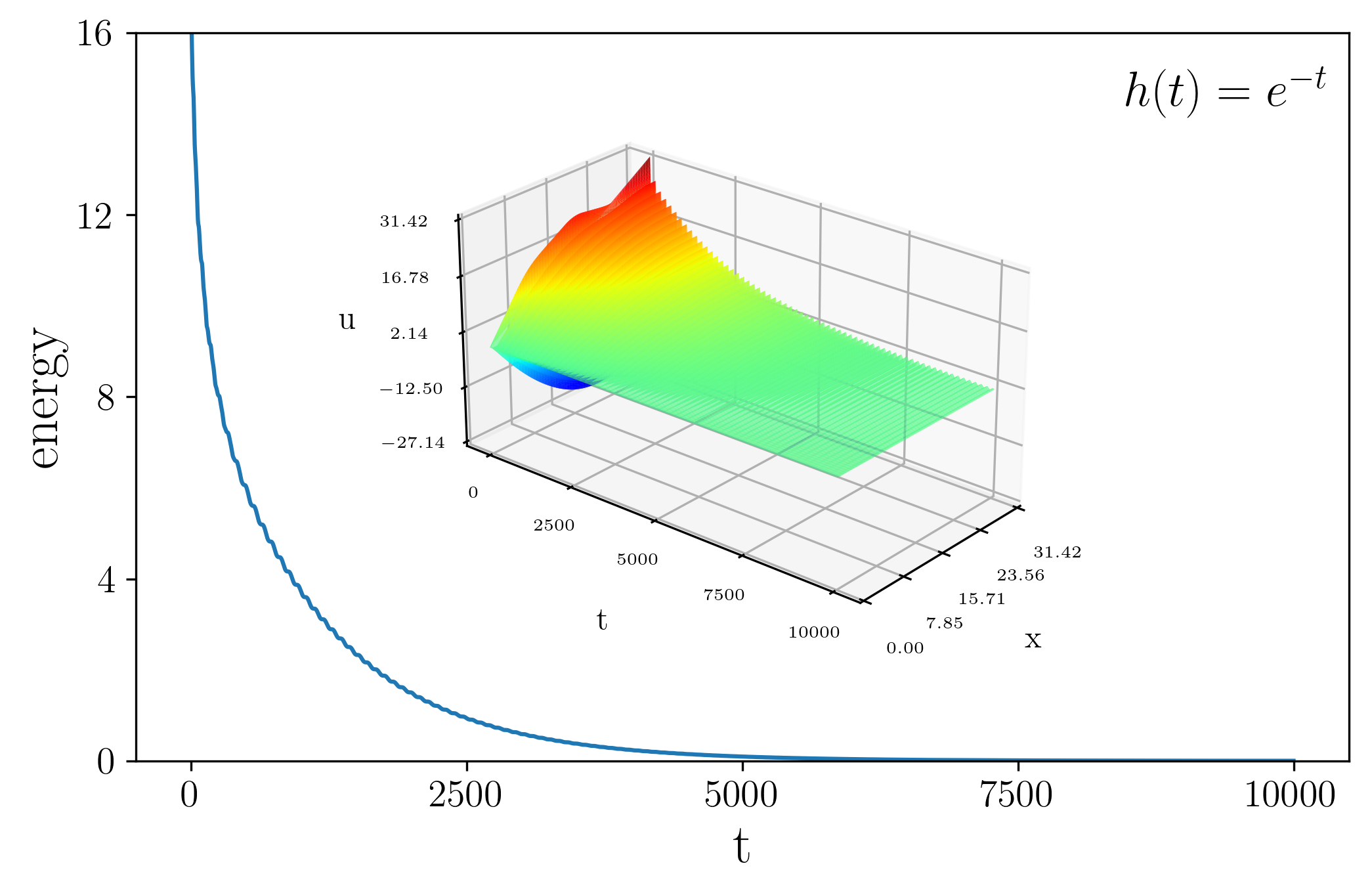}}}\\
	\caption{(a) Linear damped and (b) linear viscoelastic waves  with rough initial data} \label{sample-figure8}
\end{figure}

%\begin{figure}[!htb]
%	\centering
%	\includegraphics[scale=0.6]{i1absi2stepdamped.png}
%	\caption{\label{}
%		Energy of solution plotted against time for $h(t)=e^{-t}$ and $a(x)=e^{-x}$.
%	}
%\end{figure}
%
%\begin{figure}[!htb]
%	\centering
%	\includegraphics[scale=0.6]{i1absi2stepvisco.png}
%	\caption{\label{}
%		Energy of solution plotted against time for $g(t)=e^{-(t+1)}$ and $h(t)=e^{-t}$.
%	}
%\end{figure}

%\section*{Disclosure statement}
%No potential conflict of interest was reported by the authors.

\bibliographystyle{apacite}
\bibliography{references}

\begin{thebibliography}{}

\bibitem [\protect \citeauthoryear {%
Alabau-Boussouira%
, Brockett%
, Glass%
, Le~Rousseau%
\BCBL {}\ \BBA {} Zuazua%
}{%
Alabau-Boussouira%
\ \protect \BOthers {.}}{%
{\protect \APACyear {2012}}%
}]{%
ala10}
\APACinsertmetastar {%
ala10}%
\begin{APACrefauthors}%
Alabau-Boussouira, F.%
, Brockett, R.%
, Glass, O.%
, Le~Rousseau, J.%
\BCBL {}\ \BBA {} Zuazua, E.%
\end{APACrefauthors}%
\unskip\
\newblock
\APACrefYear{2012}.
\newblock
\APACrefbtitle {Control of partial differential equations} {Control of partial
  differential equations}\ (\BVOL\ 2048; P.~Cannarsa\ \BBA {} J\BHBI M.~Coron,
  \BEDS{}).
\newblock
\APACaddressPublisher{}{Springer, Heidelberg; Fondazione C.I.M.E., Florence}.
\newblock
\begin{APACrefURL} \url{https://doi.org/10.1007/978-3-642-27893-8}
  \end{APACrefURL}
\newblock
\APACrefnote{Lectures from the CIME Course held in Cetraro, July 19--23, 2010,
  Fondazione CIME/CIME Foundation Subseries}
\newblock
\begin{APACrefDOI} \doi{10.1007/978-3-642-27893-8} \end{APACrefDOI}
\PrintBackRefs{\CurrentBib}

\bibitem [\protect \citeauthoryear {%
Alabau-Boussouira%
\ \BBA {} Cannarsa%
}{%
Alabau-Boussouira%
\ \BBA {} Cannarsa%
}{%
{\protect \APACyear {2009}}%
}]{%
alabau2009}
\APACinsertmetastar {%
alabau2009}%
\begin{APACrefauthors}%
Alabau-Boussouira, F.%
\BCBT {}\ \BBA {} Cannarsa, P.%
\end{APACrefauthors}%
\unskip\
\newblock
\APACrefYearMonthDay{2009}{}{}.
\newblock
{\BBOQ}\APACrefatitle {A general method for proving sharp energy decay rates
  for memory-dissipative evolution equations} {A general method for proving
  sharp energy decay rates for memory-dissipative evolution equations}.{\BBCQ}
\newblock
\APACjournalVolNumPages{C. R. Math. Acad. Sci. Paris}{347}{15-16}{867--872}.
\newblock
\begin{APACrefURL} \url{https://doi.org/10.1016/j.crma.2009.05.011}
  \end{APACrefURL}
\newblock
\begin{APACrefDOI} \doi{10.1016/j.crma.2009.05.011} \end{APACrefDOI}
\PrintBackRefs{\CurrentBib}

\bibitem [\protect \citeauthoryear {%
Algharabli%
, Al-Mahdi%
\BCBL {}\ \BBA {} Messaoudi%
}{%
Algharabli%
\ \protect \BOthers {.}}{%
{\protect \APACyear {2019}}%
}]{%
suggested}
\APACinsertmetastar {%
suggested}%
\begin{APACrefauthors}%
Algharabli, M.%
, Al-Mahdi, A.%
\BCBL {}\ \BBA {} Messaoudi, S.%
\end{APACrefauthors}%
\unskip\
\newblock
\APACrefYearMonthDay{2019}{}{}.
\newblock
{\BBOQ}\APACrefatitle {General and Optimal Decay Result for a Viscoelastic
  Problem with Nonlinear Boundary Feedback} {General and optimal decay result
  for a viscoelastic problem with nonlinear boundary feedback}.{\BBCQ}
\newblock
\APACjournalVolNumPages{Journal of Dynamical and Control Systems}{25}{}{}.
\newblock
\begin{APACrefURL} \url{https://doi.org/10.1007/s10883-018-9422-y}
  \end{APACrefURL}
\newblock
\begin{APACrefDOI} \doi{10.1007/s10883-018-9422-y} \end{APACrefDOI}
\PrintBackRefs{\CurrentBib}

\bibitem [\protect \citeauthoryear {%
Belhannache%
, Algharabli%
\BCBL {}\ \BBA {} Messaoudi%
}{%
Belhannache%
\ \protect \BOthers {.}}{%
{\protect \APACyear {2020}}%
}]{%
farida2020}
\APACinsertmetastar {%
farida2020}%
\begin{APACrefauthors}%
Belhannache, F.%
, Algharabli, M.%
\BCBL {}\ \BBA {} Messaoudi, S.%
\end{APACrefauthors}%
\unskip\
\newblock
\APACrefYearMonthDay{2020}{}{}.
\newblock
{\BBOQ}\APACrefatitle {Asymptotic Stability for a Viscoelastic Equation with
  Nonlinear Damping and Very General Type of Relaxation Functions} {Asymptotic
  stability for a viscoelastic equation with nonlinear damping and very general
  type of relaxation functions}.{\BBCQ}
\newblock
\APACjournalVolNumPages{Journal of Dynamical and Control Systems}{26}{}{}.
\newblock
\begin{APACrefURL} \url{https://doi.org/10.1007/s10883-019-9429-z}
  \end{APACrefURL}
\newblock
\begin{APACrefDOI} \doi{10.1007/s10883-019-9429-z} \end{APACrefDOI}
\PrintBackRefs{\CurrentBib}

\bibitem [\protect \citeauthoryear {%
Boukhatem%
\ \BBA {} Benabderrahmane%
}{%
Boukhatem%
\ \BBA {} Benabderrahmane%
}{%
{\protect \APACyear {2014}}%
}]{%
baukhatem}
\APACinsertmetastar {%
baukhatem}%
\begin{APACrefauthors}%
Boukhatem, Y.%
\BCBT {}\ \BBA {} Benabderrahmane, B.%
\end{APACrefauthors}%
\unskip\
\newblock
\APACrefYearMonthDay{2014}{}{}.
\newblock
{\BBOQ}\APACrefatitle {Existence and decay of solutions for a viscoelastic wave
  equation with acoustic boundary conditions} {Existence and decay of solutions
  for a viscoelastic wave equation with acoustic boundary conditions}.{\BBCQ}
\newblock
\APACjournalVolNumPages{Nonlinear Analysis: Theory, Methods and
  Applications}{97}{}{191–209}.
\newblock
\begin{APACrefURL} \url{https://doi.org/10.1016/j.na.2013.11.019}
  \end{APACrefURL}
\newblock
\begin{APACrefDOI} \doi{10.1016/j.na.2013.11.019} \end{APACrefDOI}
\PrintBackRefs{\CurrentBib}

\bibitem [\protect \citeauthoryear {%
Brown%
, Du%
, Eruslu%
\BCBL {}\ \BBA {} Sayas%
}{%
Brown%
\ \protect \BOthers {.}}{%
{\protect \APACyear {2018}}%
}]{%
eruslu18}
\APACinsertmetastar {%
eruslu18}%
\begin{APACrefauthors}%
Brown, T\BPBI S.%
, Du, S.%
, Eruslu, H.%
\BCBL {}\ \BBA {} Sayas, F\BHBI J.%
\end{APACrefauthors}%
\unskip\
\newblock
\APACrefYearMonthDay{2018}{}{}.
\newblock
{\BBOQ}\APACrefatitle {Analysis of models for viscoelastic wave propagation}
  {Analysis of models for viscoelastic wave propagation}.{\BBCQ}
\newblock
\APACjournalVolNumPages{Appl. Math. Nonlinear Sci.}{3}{1}{55--96}.
\newblock
\begin{APACrefURL} \url{https://doi.org/10.21042/AMNS.2018.1.00006}
  \end{APACrefURL}
\newblock
\begin{APACrefDOI} \doi{10.21042/AMNS.2018.1.00006} \end{APACrefDOI}
\PrintBackRefs{\CurrentBib}

\bibitem [\protect \citeauthoryear {%
Cavalcanti%
, Domingos~Cavalcanti%
\BCBL {}\ \BBA {} Martinez%
}{%
Cavalcanti%
\ \protect \BOthers {.}}{%
{\protect \APACyear {2008}}%
}]{%
cavalcantidissipative}
\APACinsertmetastar {%
cavalcantidissipative}%
\begin{APACrefauthors}%
Cavalcanti, M.%
, Domingos~Cavalcanti, V.%
\BCBL {}\ \BBA {} Martinez, P.%
\end{APACrefauthors}%
\unskip\
\newblock
\APACrefYearMonthDay{2008}{}{}.
\newblock
{\BBOQ}\APACrefatitle {General decay rate estimates for viscoelastic
  dissipative systems} {General decay rate estimates for viscoelastic
  dissipative systems}.{\BBCQ}
\newblock
\APACjournalVolNumPages{Nonlinear Analysis: Theory, Methods and
  Applications}{68}{}{177-193}.
\newblock
\begin{APACrefURL} \url{https://doi.org/10.1016/j.na.2006.10.040}
  \end{APACrefURL}
\newblock
\begin{APACrefDOI} \doi{10.1016/j.na.2006.10.040} \end{APACrefDOI}
\PrintBackRefs{\CurrentBib}

\bibitem [\protect \citeauthoryear {%
Cox%
\ \BBA {} Zuazua%
}{%
Cox%
\ \BBA {} Zuazua%
}{%
{\protect \APACyear {1995}}%
}]{%
cox95}
\APACinsertmetastar {%
cox95}%
\begin{APACrefauthors}%
Cox, S.%
\BCBT {}\ \BBA {} Zuazua, E.%
\end{APACrefauthors}%
\unskip\
\newblock
\APACrefYearMonthDay{1995}{}{}.
\newblock
{\BBOQ}\APACrefatitle {The rate at which energy decays in a string damped at
  one end} {The rate at which energy decays in a string damped at one
  end}.{\BBCQ}
\newblock
\APACjournalVolNumPages{Indiana Univ. Math. J.}{44}{2}{545--573}.
\newblock
\begin{APACrefURL} \url{https://doi.org/10.1512/iumj.1995.44.2001}
  \end{APACrefURL}
\newblock
\begin{APACrefDOI} \doi{10.1512/iumj.1995.44.2001} \end{APACrefDOI}
\PrintBackRefs{\CurrentBib}

\bibitem [\protect \citeauthoryear {%
Dafermos%
}{%
Dafermos%
}{%
{\protect \APACyear {1970}}%
}]{%
dafermos1970}
\APACinsertmetastar {%
dafermos1970}%
\begin{APACrefauthors}%
Dafermos, C\BPBI M.%
\end{APACrefauthors}%
\unskip\
\newblock
\APACrefYearMonthDay{1970}{}{}.
\newblock
{\BBOQ}\APACrefatitle {Asymptotic stability in viscoelasticity} {Asymptotic
  stability in viscoelasticity}.{\BBCQ}
\newblock
\APACjournalVolNumPages{Arch. Rational Mech. Anal.}{37}{}{297--308}.
\newblock
\begin{APACrefURL} \url{https://doi.org/10.1007/BF00251609} \end{APACrefURL}
\newblock
\begin{APACrefDOI} \doi{10.1007/BF00251609} \end{APACrefDOI}
\PrintBackRefs{\CurrentBib}

\bibitem [\protect \citeauthoryear {%
Daoulatli%
}{%
Daoulatli%
}{%
{\protect \APACyear {2012}}%
}]{%
daou12}
\APACinsertmetastar {%
daou12}%
\begin{APACrefauthors}%
Daoulatli, M.%
\end{APACrefauthors}%
\unskip\
\newblock
\APACrefYearMonthDay{2012}{}{}.
\newblock
{\BBOQ}\APACrefatitle {Behaviors of the energy of solutions of the wave
  equation with damping and external force} {Behaviors of the energy of
  solutions of the wave equation with damping and external force}.{\BBCQ}
\newblock
\APACjournalVolNumPages{J. Math. Anal. Appl.}{389}{1}{205--225}.
\newblock
\begin{APACrefURL} \url{https://doi.org/10.1016/j.jmaa.2011.11.051}
  \end{APACrefURL}
\newblock
\begin{APACrefDOI} \doi{10.1016/j.jmaa.2011.11.051} \end{APACrefDOI}
\PrintBackRefs{\CurrentBib}

\bibitem [\protect \citeauthoryear {%
Dassios%
\ \BBA {} Zafiropoulos%
}{%
Dassios%
\ \BBA {} Zafiropoulos%
}{%
{\protect \APACyear {1990}}%
}]{%
dassios1992}
\APACinsertmetastar {%
dassios1992}%
\begin{APACrefauthors}%
Dassios, G.%
\BCBT {}\ \BBA {} Zafiropoulos, F.%
\end{APACrefauthors}%
\unskip\
\newblock
\APACrefYearMonthDay{1990}{}{}.
\newblock
{\BBOQ}\APACrefatitle {Equipartition of energy in linearized {$3$}-{D}
  viscoelasticity} {Equipartition of energy in linearized {$3$}-{D}
  viscoelasticity}.{\BBCQ}
\newblock
\APACjournalVolNumPages{Quart. Appl. Math.}{48}{4}{715--730}.
\newblock
\begin{APACrefURL} \url{https://doi.org/10.1090/qam/1079915} \end{APACrefURL}
\newblock
\begin{APACrefDOI} \doi{10.1090/qam/1079915} \end{APACrefDOI}
\PrintBackRefs{\CurrentBib}

\bibitem [\protect \citeauthoryear {%
Eruslu%
}{%
Eruslu%
}{%
{\protect \APACyear {2020}}%
}]{%
erusluphd}
\APACinsertmetastar {%
erusluphd}%
\begin{APACrefauthors}%
Eruslu, H.%
\end{APACrefauthors}%
\unskip\
\newblock
\APACrefYear{2020}.
\newblock
\APACrefbtitle {Analysis of {N}umerical {M}ethods for {T}ransient
  {V}iscoelastic {W}aves} {Analysis of {N}umerical {M}ethods for {T}ransient
  {V}iscoelastic {W}aves}.
\newblock
\APACaddressPublisher{}{ProQuest LLC, Ann Arbor, MI}.
\newblock
\APACrefnote{Thesis (Ph.D.)--University of Delaware}
\PrintBackRefs{\CurrentBib}

\bibitem [\protect \citeauthoryear {%
Fabrizio%
\ \BBA {} Polidoro%
}{%
Fabrizio%
\ \BBA {} Polidoro%
}{%
{\protect \APACyear {2002}}%
}]{%
Fab02}
\APACinsertmetastar {%
Fab02}%
\begin{APACrefauthors}%
Fabrizio, M.%
\BCBT {}\ \BBA {} Polidoro, S.%
\end{APACrefauthors}%
\unskip\
\newblock
\APACrefYearMonthDay{2002}{}{}.
\newblock
{\BBOQ}\APACrefatitle {Asymptotic decay for some differential systems with
  fading memory} {Asymptotic decay for some differential systems with fading
  memory}.{\BBCQ}
\newblock
\APACjournalVolNumPages{Appl. Anal.}{81}{6}{1245--1264}.
\newblock
\begin{APACrefURL} \url{https://doi.org/10.1080/0003681021000035588}
  \end{APACrefURL}
\newblock
\begin{APACrefDOI} \doi{10.1080/0003681021000035588} \end{APACrefDOI}
\PrintBackRefs{\CurrentBib}

\bibitem [\protect \citeauthoryear {%
Han%
\ \BBA {} Wang%
}{%
Han%
\ \BBA {} Wang%
}{%
{\protect \APACyear {2009}}%
}]{%
han2009}
\APACinsertmetastar {%
han2009}%
\begin{APACrefauthors}%
Han, X.%
\BCBT {}\ \BBA {} Wang, M.%
\end{APACrefauthors}%
\unskip\
\newblock
\APACrefYearMonthDay{2009}{}{}.
\newblock
{\BBOQ}\APACrefatitle {General decay of energy for a viscoelastic equation with
  nonlinear damping} {General decay of energy for a viscoelastic equation with
  nonlinear damping}.{\BBCQ}
\newblock
\APACjournalVolNumPages{Math. Methods Appl. Sci.}{32}{3}{346--358}.
\newblock
\begin{APACrefURL} \url{https://doi.org/10.1002/mma.1041} \end{APACrefURL}
\newblock
\begin{APACrefDOI} \doi{10.1002/mma.1041} \end{APACrefDOI}
\PrintBackRefs{\CurrentBib}

\bibitem [\protect \citeauthoryear {%
Hrusa%
}{%
Hrusa%
}{%
{\protect \APACyear {1985}}%
}]{%
zua85}
\APACinsertmetastar {%
zua85}%
\begin{APACrefauthors}%
Hrusa, W\BPBI J.%
\end{APACrefauthors}%
\unskip\
\newblock
\APACrefYearMonthDay{1985}{}{}.
\newblock
{\BBOQ}\APACrefatitle {Global existence and asymptotic stability for a
  semilinear hyperbolic {V}olterra equation with large initial data} {Global
  existence and asymptotic stability for a semilinear hyperbolic {V}olterra
  equation with large initial data}.{\BBCQ}
\newblock
\APACjournalVolNumPages{SIAM J. Math. Anal.}{16}{1}{110--134}.
\newblock
\begin{APACrefURL} \url{https://doi.org/10.1137/0516007} \end{APACrefURL}
\newblock
\begin{APACrefDOI} \doi{10.1137/0516007} \end{APACrefDOI}
\PrintBackRefs{\CurrentBib}

\bibitem [\protect \citeauthoryear {%
Komornik%
}{%
Komornik%
}{%
{\protect \APACyear {1994}}%
}]{%
kom94}
\APACinsertmetastar {%
kom94}%
\begin{APACrefauthors}%
Komornik, V.%
\end{APACrefauthors}%
\unskip\
\newblock
\APACrefYear{1994}.
\newblock
\APACrefbtitle {Exact controllability and stabilization} {Exact controllability
  and stabilization}.
\newblock
\APACaddressPublisher{}{Masson, Paris; John Wiley \& Sons, Ltd., Chichester}.
\newblock
\APACrefnote{The multiplier method}
\PrintBackRefs{\CurrentBib}

\bibitem [\protect \citeauthoryear {%
Lasiecka%
, Messaoudi%
\BCBL {}\ \BBA {} Mustafa%
}{%
Lasiecka%
\ \protect \BOthers {.}}{%
{\protect \APACyear {2013}}%
}]{%
Las13}
\APACinsertmetastar {%
Las13}%
\begin{APACrefauthors}%
Lasiecka, I.%
, Messaoudi, S\BPBI A.%
\BCBL {}\ \BBA {} Mustafa, M\BPBI I.%
\end{APACrefauthors}%
\unskip\
\newblock
\APACrefYearMonthDay{2013}{}{}.
\newblock
{\BBOQ}\APACrefatitle {Note on intrinsic decay rates for abstract wave
  equations with memory} {Note on intrinsic decay rates for abstract wave
  equations with memory}.{\BBCQ}
\newblock
\APACjournalVolNumPages{J. Math. Phys.}{54}{3}{031504, 18}.
\newblock
\begin{APACrefURL} \url{https://doi.org/10.1063/1.4793988} \end{APACrefURL}
\newblock
\begin{APACrefDOI} \doi{10.1063/1.4793988} \end{APACrefDOI}
\PrintBackRefs{\CurrentBib}

\bibitem [\protect \citeauthoryear {%
Liu%
}{%
Liu%
}{%
{\protect \APACyear {2009}}%
}]{%
liu2009-2}
\APACinsertmetastar {%
liu2009-2}%
\begin{APACrefauthors}%
Liu, W.%
\end{APACrefauthors}%
\unskip\
\newblock
\APACrefYearMonthDay{2009}{}{}.
\newblock
{\BBOQ}\APACrefatitle {General decay rate estimate for a viscoelastic equation
  with weakly nonlinear time-dependent dissipation and source terms} {General
  decay rate estimate for a viscoelastic equation with weakly nonlinear
  time-dependent dissipation and source terms}.{\BBCQ}
\newblock
\APACjournalVolNumPages{J. Math. Phys.}{50}{11}{113506, 17}.
\newblock
\begin{APACrefURL} \url{https://doi.org/10.1063/1.3254323} \end{APACrefURL}
\newblock
\begin{APACrefDOI} \doi{10.1063/1.3254323} \end{APACrefDOI}
\PrintBackRefs{\CurrentBib}

\bibitem [\protect \citeauthoryear {%
Liu%
}{%
Liu%
}{%
{\protect \APACyear {2012}}%
}]{%
liu2012}
\APACinsertmetastar {%
liu2012}%
\begin{APACrefauthors}%
Liu, W.%
\end{APACrefauthors}%
\unskip\
\newblock
\APACrefYearMonthDay{2012}{}{}.
\newblock
{\BBOQ}\APACrefatitle {General decay of the solution for a viscoelastic wave
  equation with a time-varying delay term in the internal feedback} {General
  decay of the solution for a viscoelastic wave equation with a time-varying
  delay term in the internal feedback}.{\BBCQ}
\newblock
\APACjournalVolNumPages{Journal of Mathematical Physics}{54}{}{}.
\newblock
\begin{APACrefDOI} \doi{10.1063/1.4799929} \end{APACrefDOI}
\PrintBackRefs{\CurrentBib}

\bibitem [\protect \citeauthoryear {%
Martinez%
}{%
Martinez%
}{%
{\protect \APACyear {2000}}%
}]{%
mart00}
\APACinsertmetastar {%
mart00}%
\begin{APACrefauthors}%
Martinez, P.%
\end{APACrefauthors}%
\unskip\
\newblock
\APACrefYearMonthDay{2000}{}{}.
\newblock
{\BBOQ}\APACrefatitle {Stabilization for the wave equation with {N}eumann
  boundary condition by a locally distributed damping} {Stabilization for the
  wave equation with {N}eumann boundary condition by a locally distributed
  damping}.{\BBCQ}
\newblock
\BIn{} \APACrefbtitle {Contr\^{o}le des syst\`emes gouvern\'{e}s par des
  \'{e}quations aux d\'{e}riv\'{e}es partielles ({N}ancy, 1999)} {Contr\^{o}le
  des syst\`emes gouvern\'{e}s par des \'{e}quations aux d\'{e}riv\'{e}es
  partielles ({N}ancy, 1999)}\ (\BVOL~8, \BPGS\ 119--136).
\newblock
\APACaddressPublisher{}{Soc. Math. Appl. Indust., Paris}.
\newblock
\begin{APACrefURL} \url{https://doi.org/10.1051/proc:2000009} \end{APACrefURL}
\newblock
\begin{APACrefDOI} \doi{10.1051/proc:2000009} \end{APACrefDOI}
\PrintBackRefs{\CurrentBib}

\bibitem [\protect \citeauthoryear {%
Messaoudi%
}{%
Messaoudi%
}{%
{\protect \APACyear {2008}}%
}]{%
messaoudi2008}
\APACinsertmetastar {%
messaoudi2008}%
\begin{APACrefauthors}%
Messaoudi, S\BPBI A.%
\end{APACrefauthors}%
\unskip\
\newblock
\APACrefYearMonthDay{2008}{}{}.
\newblock
{\BBOQ}\APACrefatitle {General decay of solutions of a viscoelastic equation}
  {General decay of solutions of a viscoelastic equation}.{\BBCQ}
\newblock
\APACjournalVolNumPages{J. Math. Anal. Appl.}{341}{2}{1457--1467}.
\newblock
\begin{APACrefURL} \url{https://doi.org/10.1016/j.jmaa.2007.11.048}
  \end{APACrefURL}
\newblock
\begin{APACrefDOI} \doi{10.1016/j.jmaa.2007.11.048} \end{APACrefDOI}
\PrintBackRefs{\CurrentBib}

\bibitem [\protect \citeauthoryear {%
Messaoudi%
\ \BBA {} Mustafa%
}{%
Messaoudi%
\ \BBA {} Mustafa%
}{%
{\protect \APACyear {2009}}%
}]{%
messaoudi2009}
\APACinsertmetastar {%
messaoudi2009}%
\begin{APACrefauthors}%
Messaoudi, S\BPBI A.%
\BCBT {}\ \BBA {} Mustafa, M\BPBI I.%
\end{APACrefauthors}%
\unskip\
\newblock
\APACrefYearMonthDay{2009}{}{}.
\newblock
{\BBOQ}\APACrefatitle {On the control of solutions of viscoelastic equations
  with boundary feedback} {On the control of solutions of viscoelastic
  equations with boundary feedback}.{\BBCQ}
\newblock
\APACjournalVolNumPages{Nonlinear Anal. Real World Appl.}{10}{5}{3132--3140}.
\newblock
\begin{APACrefURL} \url{https://doi.org/10.1016/j.nonrwa.2008.10.026}
  \end{APACrefURL}
\newblock
\begin{APACrefDOI} \doi{10.1016/j.nonrwa.2008.10.026} \end{APACrefDOI}
\PrintBackRefs{\CurrentBib}

\bibitem [\protect \citeauthoryear {%
Mu\~{n}oz Rivera%
}{%
Mu\~{n}oz Rivera%
}{%
{\protect \APACyear {1994}}%
}]{%
rivera1994}
\APACinsertmetastar {%
rivera1994}%
\begin{APACrefauthors}%
Mu\~{n}oz Rivera, J\BPBI E.%
\end{APACrefauthors}%
\unskip\
\newblock
\APACrefYearMonthDay{1994}{}{}.
\newblock
{\BBOQ}\APACrefatitle {Asymptotic behaviour in linear viscoelasticity}
  {Asymptotic behaviour in linear viscoelasticity}.{\BBCQ}
\newblock
\APACjournalVolNumPages{Quart. Appl. Math.}{52}{4}{629--648}.
\newblock
\begin{APACrefURL} \url{https://doi.org/10.1090/qam/1306041} \end{APACrefURL}
\newblock
\begin{APACrefDOI} \doi{10.1090/qam/1306041} \end{APACrefDOI}
\PrintBackRefs{\CurrentBib}

\bibitem [\protect \citeauthoryear {%
Mu\~{n}oz Rivera%
\ \BBA {} Cabanillas~Lapa%
}{%
Mu\~{n}oz Rivera%
\ \BBA {} Cabanillas~Lapa%
}{%
{\protect \APACyear {1996}}%
}]{%
rivera1996}
\APACinsertmetastar {%
rivera1996}%
\begin{APACrefauthors}%
Mu\~{n}oz Rivera, J\BPBI E.%
\BCBT {}\ \BBA {} Cabanillas~Lapa, E.%
\end{APACrefauthors}%
\unskip\
\newblock
\APACrefYearMonthDay{1996}{}{}.
\newblock
{\BBOQ}\APACrefatitle {Decay rates of solutions of an anisotropic inhomogeneous
  {$n$}-dimensional viscoelastic equation with polynomially decaying kernels}
  {Decay rates of solutions of an anisotropic inhomogeneous {$n$}-dimensional
  viscoelastic equation with polynomially decaying kernels}.{\BBCQ}
\newblock
\APACjournalVolNumPages{Comm. Math. Phys.}{177}{3}{583--602}.
\newblock
\begin{APACrefURL} \url{http://projecteuclid.org/euclid.cmp/1104286436}
  \end{APACrefURL}
\PrintBackRefs{\CurrentBib}

\bibitem [\protect \citeauthoryear {%
Mu\~{n}oz Rivera%
, Naso%
\BCBL {}\ \BBA {} Vegni%
}{%
Mu\~{n}oz Rivera%
\ \protect \BOthers {.}}{%
{\protect \APACyear {2003}}%
}]{%
rivera2003}
\APACinsertmetastar {%
rivera2003}%
\begin{APACrefauthors}%
Mu\~{n}oz Rivera, J\BPBI E.%
, Naso, M\BPBI G.%
\BCBL {}\ \BBA {} Vegni, F\BPBI M.%
\end{APACrefauthors}%
\unskip\
\newblock
\APACrefYearMonthDay{2003}{}{}.
\newblock
{\BBOQ}\APACrefatitle {Asymptotic behavior of the energy for a class of weakly
  dissipative second-order systems with memory} {Asymptotic behavior of the
  energy for a class of weakly dissipative second-order systems with
  memory}.{\BBCQ}
\newblock
\APACjournalVolNumPages{J. Math. Anal. Appl.}{286}{2}{692--704}.
\newblock
\begin{APACrefURL} \url{https://doi.org/10.1016/S0022-247X(03)00511-0}
  \end{APACrefURL}
\newblock
\begin{APACrefDOI} \doi{10.1016/S0022-247X(03)00511-0} \end{APACrefDOI}
\PrintBackRefs{\CurrentBib}

\bibitem [\protect \citeauthoryear {%
Mu\~{n}oz Rivera%
\ \BBA {} Peres~Salvatierra%
}{%
Mu\~{n}oz Rivera%
\ \BBA {} Peres~Salvatierra%
}{%
{\protect \APACyear {2001}}%
}]{%
rivera2001}
\APACinsertmetastar {%
rivera2001}%
\begin{APACrefauthors}%
Mu\~{n}oz Rivera, J\BPBI E.%
\BCBT {}\ \BBA {} Peres~Salvatierra, A.%
\end{APACrefauthors}%
\unskip\
\newblock
\APACrefYearMonthDay{2001}{}{}.
\newblock
{\BBOQ}\APACrefatitle {Asymptotic behaviour of the energy in partially
  viscoelastic materials} {Asymptotic behaviour of the energy in partially
  viscoelastic materials}.{\BBCQ}
\newblock
\APACjournalVolNumPages{Quart. Appl. Math.}{59}{3}{557--578}.
\newblock
\begin{APACrefURL} \url{https://doi.org/10.1090/qam/1848535} \end{APACrefURL}
\newblock
\begin{APACrefDOI} \doi{10.1090/qam/1848535} \end{APACrefDOI}
\PrintBackRefs{\CurrentBib}

\bibitem [\protect \citeauthoryear {%
Mustafa%
}{%
Mustafa%
}{%
{\protect \APACyear {2018}}%
}]{%
mustafa2017}
\APACinsertmetastar {%
mustafa2017}%
\begin{APACrefauthors}%
Mustafa, M\BPBI I.%
\end{APACrefauthors}%
\unskip\
\newblock
\APACrefYearMonthDay{2018}{}{}.
\newblock
{\BBOQ}\APACrefatitle {Optimal decay rates for the viscoelastic wave equation}
  {Optimal decay rates for the viscoelastic wave equation}.{\BBCQ}
\newblock
\APACjournalVolNumPages{Math. Methods Appl. Sci.}{41}{1}{192--204}.
\newblock
\begin{APACrefURL} \url{https://doi.org/10.1002/mma.4604} \end{APACrefURL}
\newblock
\begin{APACrefDOI} \doi{10.1002/mma.4604} \end{APACrefDOI}
\PrintBackRefs{\CurrentBib}

\bibitem [\protect \citeauthoryear {%
Mustafa%
\ \BBA {} Messaoudi%
}{%
Mustafa%
\ \BBA {} Messaoudi%
}{%
{\protect \APACyear {2012}}%
}]{%
mustafa-messaoudi2012}
\APACinsertmetastar {%
mustafa-messaoudi2012}%
\begin{APACrefauthors}%
Mustafa, M\BPBI I.%
\BCBT {}\ \BBA {} Messaoudi, S\BPBI A.%
\end{APACrefauthors}%
\unskip\
\newblock
\APACrefYearMonthDay{2012}{}{}.
\newblock
{\BBOQ}\APACrefatitle {General stability result for viscoelastic wave
  equations} {General stability result for viscoelastic wave equations}.{\BBCQ}
\newblock
\APACjournalVolNumPages{J. Math. Phys.}{53}{5}{053702, 14}.
\newblock
\begin{APACrefURL} \url{https://doi.org/10.1063/1.4711830} \end{APACrefURL}
\newblock
\begin{APACrefDOI} \doi{10.1063/1.4711830} \end{APACrefDOI}
\PrintBackRefs{\CurrentBib}

\bibitem [\protect \citeauthoryear {%
Nakao%
}{%
Nakao%
}{%
{\protect \APACyear {1997}}%
}]{%
nakao97}
\APACinsertmetastar {%
nakao97}%
\begin{APACrefauthors}%
Nakao, M.%
\end{APACrefauthors}%
\unskip\
\newblock
\APACrefYearMonthDay{1997}{}{}.
\newblock
{\BBOQ}\APACrefatitle {Decay of solutions of the wave equation with some
  localized dissipations} {Decay of solutions of the wave equation with some
  localized dissipations}.{\BBCQ}
\newblock
\BIn{} \APACrefbtitle {Proceedings of the {S}econd {W}orld {C}ongress of
  {N}onlinear {A}nalysts, {P}art 6 ({A}thens, 1996)} {Proceedings of the
  {S}econd {W}orld {C}ongress of {N}onlinear {A}nalysts, {P}art 6 ({A}thens,
  1996)}\ (\BVOL~30, \BPGS\ 3775--3784).
\newblock
\begin{APACrefURL} \url{https://doi.org/10.1016/S0362-546X(96)00348-3}
  \end{APACrefURL}
\newblock
\begin{APACrefDOI} \doi{10.1016/S0362-546X(96)00348-3} \end{APACrefDOI}
\PrintBackRefs{\CurrentBib}

\bibitem [\protect \citeauthoryear {%
Ogbiyele%
\ \BBA {} Arawomo%
}{%
Ogbiyele%
\ \BBA {} Arawomo%
}{%
{\protect \APACyear {2022}}%
}]{%
paul2022}
\APACinsertmetastar {%
paul2022}%
\begin{APACrefauthors}%
Ogbiyele, P\BPBI A.%
\BCBT {}\ \BBA {} Arawomo, P\BPBI O.%
\end{APACrefauthors}%
\unskip\
\newblock
\APACrefYearMonthDay{2022}{}{}.
\newblock
{\BBOQ}\APACrefatitle {Energy decay for a viscoelastic wave equation with
  space-time potential in Rn} {Energy decay for a viscoelastic wave equation
  with space-time potential in rn}.{\BBCQ}
\newblock
\APACjournalVolNumPages{Journal of Mathematical Analysis and
  Applications}{505}{2}{125529}.
\newblock
\begin{APACrefURL}
  \url{https://www.sciencedirect.com/science/article/pii/S0022247X21006089}
  \end{APACrefURL}
\newblock
\begin{APACrefDOI} \doi{https://doi.org/10.1016/j.jmaa.2021.125529}
  \end{APACrefDOI}
\PrintBackRefs{\CurrentBib}

\bibitem [\protect \citeauthoryear {%
\"{O}zsar\i%
}{%
\"{O}zsar\i%
}{%
{\protect \APACyear {2010}}%
}]{%
OThesis}
\APACinsertmetastar {%
OThesis}%
\begin{APACrefauthors}%
\"{O}zsar\i, T.%
\end{APACrefauthors}%
\unskip\
\newblock
\APACrefYear{2010}.
\newblock
\APACrefbtitle {Stabilization of {N}onlinear {S}chrodinger {E}quation with
  {I}nhomogeneous {D}irichlet {B}oundary {C}ontrol} {Stabilization of
  {N}onlinear {S}chrodinger {E}quation with {I}nhomogeneous {D}irichlet
  {B}oundary {C}ontrol}.
\newblock
\APACaddressPublisher{}{ProQuest LLC, Ann Arbor, MI}.
\newblock
\APACrefnote{Thesis (Ph.D.)--University of Virginia}
\PrintBackRefs{\CurrentBib}

\bibitem [\protect \citeauthoryear {%
Pata%
}{%
Pata%
}{%
{\protect \APACyear {2009}}%
}]{%
Vit09}
\APACinsertmetastar {%
Vit09}%
\begin{APACrefauthors}%
Pata, V.%
\end{APACrefauthors}%
\unskip\
\newblock
\APACrefYearMonthDay{2009}{}{}.
\newblock
{\BBOQ}\APACrefatitle {Stability and exponential stability in linear
  viscoelasticity} {Stability and exponential stability in linear
  viscoelasticity}.{\BBCQ}
\newblock
\APACjournalVolNumPages{Milan J. Math.}{77}{}{333--360}.
\newblock
\begin{APACrefURL} \url{https://doi.org/10.1007/s00032-009-0098-3}
  \end{APACrefURL}
\newblock
\begin{APACrefDOI} \doi{10.1007/s00032-009-0098-3} \end{APACrefDOI}
\PrintBackRefs{\CurrentBib}

\bibitem [\protect \citeauthoryear {%
Pazy%
}{%
Pazy%
}{%
{\protect \APACyear {1983}}%
}]{%
pazy}
\APACinsertmetastar {%
pazy}%
\begin{APACrefauthors}%
Pazy, A.%
\end{APACrefauthors}%
\unskip\
\newblock
\APACrefYear{1983}.
\newblock
\APACrefbtitle {Semigroups of Linear Operators and Applications to Partial
  Differential Equations} {Semigroups of linear operators and applications to
  partial differential equations}.
\newblock
\APACaddressPublisher{}{Springer}.
\newblock
\begin{APACrefURL} \url{https://books.google.com.tr/books?id=80XYPwAACAAJ}
  \end{APACrefURL}
\PrintBackRefs{\CurrentBib}

\bibitem [\protect \citeauthoryear {%
Pitts%
\ \BBA {} Rammaha%
}{%
Pitts%
\ \BBA {} Rammaha%
}{%
{\protect \APACyear {2002}}%
}]{%
pitts2002}
\APACinsertmetastar {%
pitts2002}%
\begin{APACrefauthors}%
Pitts, D\BPBI R.%
\BCBT {}\ \BBA {} Rammaha, M\BPBI A.%
\end{APACrefauthors}%
\unskip\
\newblock
\APACrefYearMonthDay{2002}{}{}.
\newblock
{\BBOQ}\APACrefatitle {Global existence and non-existence theorems for
  nonlinear wave equations} {Global existence and non-existence theorems for
  nonlinear wave equations}.{\BBCQ}
\newblock
\APACjournalVolNumPages{Indiana Univ. Math. J.}{51}{6}{1479--1509}.
\newblock
\begin{APACrefURL} \url{https://doi.org/10.1512/iumj.2002.51.2215}
  \end{APACrefURL}
\newblock
\begin{APACrefDOI} \doi{10.1512/iumj.2002.51.2215} \end{APACrefDOI}
\PrintBackRefs{\CurrentBib}

\bibitem [\protect \citeauthoryear {%
Rauch%
}{%
Rauch%
}{%
{\protect \APACyear {1976}}%
}]{%
rauch1976}
\APACinsertmetastar {%
rauch1976}%
\begin{APACrefauthors}%
Rauch, J.%
\end{APACrefauthors}%
\unskip\
\newblock
\APACrefYearMonthDay{1976}{}{}.
\newblock
{\BBOQ}\APACrefatitle {Qualitative behavior of dissipative wave equations on
  bounded domains} {Qualitative behavior of dissipative wave equations on
  bounded domains}.{\BBCQ}
\newblock
\APACjournalVolNumPages{Archive for Rational Mechanics and Analysis}{62}{}{}.
\newblock
\begin{APACrefDOI} \doi{10.1007/BF00251857} \end{APACrefDOI}
\PrintBackRefs{\CurrentBib}

\bibitem [\protect \citeauthoryear {%
Rauch%
\ \BBA {} Taylor%
}{%
Rauch%
\ \BBA {} Taylor%
}{%
{\protect \APACyear {1974}}%
}]{%
rauch74}
\APACinsertmetastar {%
rauch74}%
\begin{APACrefauthors}%
Rauch, J.%
\BCBT {}\ \BBA {} Taylor, M.%
\end{APACrefauthors}%
\unskip\
\newblock
\APACrefYearMonthDay{1974}{}{}.
\newblock
{\BBOQ}\APACrefatitle {Exponential decay of solutions to hyperbolic equations
  in bounded domains} {Exponential decay of solutions to hyperbolic equations
  in bounded domains}.{\BBCQ}
\newblock
\APACjournalVolNumPages{Indiana Univ. Math. J.}{24}{}{79--86}.
\newblock
\begin{APACrefURL} \url{https://doi.org/10.1512/iumj.1974.24.24004}
  \end{APACrefURL}
\newblock
\begin{APACrefDOI} \doi{10.1512/iumj.1974.24.24004} \end{APACrefDOI}
\PrintBackRefs{\CurrentBib}

\bibitem [\protect \citeauthoryear {%
Shubov%
, Martin%
, Dauer%
\BCBL {}\ \BBA {} Belinskiy%
}{%
Shubov%
\ \protect \BOthers {.}}{%
{\protect \APACyear {1997}}%
}]{%
shu97}
\APACinsertmetastar {%
shu97}%
\begin{APACrefauthors}%
Shubov, M\BPBI A.%
, Martin, C\BPBI F.%
, Dauer, J\BPBI P.%
\BCBL {}\ \BBA {} Belinskiy, B\BPBI P.%
\end{APACrefauthors}%
\unskip\
\newblock
\APACrefYearMonthDay{1997}{}{}.
\newblock
{\BBOQ}\APACrefatitle {Exact controllability of the damped wave equation}
  {Exact controllability of the damped wave equation}.{\BBCQ}
\newblock
\APACjournalVolNumPages{SIAM J. Control Optim.}{35}{5}{1773--1789}.
\newblock
\begin{APACrefURL} \url{https://doi.org/10.1137/S0363012996291616}
  \end{APACrefURL}
\newblock
\begin{APACrefDOI} \doi{10.1137/S0363012996291616} \end{APACrefDOI}
\PrintBackRefs{\CurrentBib}

\bibitem [\protect \citeauthoryear {%
Tcheugou\'{e}~T\'{e}bou%
}{%
Tcheugou\'{e}~T\'{e}bou%
}{%
{\protect \APACyear {1998}}%
}]{%
teb98}
\APACinsertmetastar {%
teb98}%
\begin{APACrefauthors}%
Tcheugou\'{e}~T\'{e}bou, L\BPBI R.%
\end{APACrefauthors}%
\unskip\
\newblock
\APACrefYearMonthDay{1998}{}{}.
\newblock
{\BBOQ}\APACrefatitle {Well-posedness and energy decay estimates for the damped
  wave equation with {$L^r$} localizing coefficient} {Well-posedness and energy
  decay estimates for the damped wave equation with {$L^r$} localizing
  coefficient}.{\BBCQ}
\newblock
\APACjournalVolNumPages{Comm. Partial Differential
  Equations}{23}{9-10}{1839--1855}.
\newblock
\begin{APACrefURL} \url{https://doi.org/10.1080/03605309808821403}
  \end{APACrefURL}
\newblock
\begin{APACrefDOI} \doi{10.1080/03605309808821403} \end{APACrefDOI}
\PrintBackRefs{\CurrentBib}

\bibitem [\protect \citeauthoryear {%
Tebou%
}{%
Tebou%
}{%
{\protect \APACyear {1998}}%
}]{%
Tebeu1998}
\APACinsertmetastar {%
Tebeu1998}%
\begin{APACrefauthors}%
Tebou, L.%
\end{APACrefauthors}%
\unskip\
\newblock
\APACrefYearMonthDay{1998}{}{}.
\newblock
{\BBOQ}\APACrefatitle {Well-posedness and energy decay estimates for the damped
  wave equation with Lr localizing coefficient} {Well-posedness and energy
  decay estimates for the damped wave equation with lr localizing
  coefficient}.{\BBCQ}
\newblock
\APACjournalVolNumPages{Communications in Partial Differential
  Equations}{23}{}{513-529}.
\newblock
\begin{APACrefURL} \url{https://doi.org/10.1080/03605309808821403}
  \end{APACrefURL}
\newblock
\begin{APACrefDOI} \doi{10.1080/03605309808821403} \end{APACrefDOI}
\PrintBackRefs{\CurrentBib}

\bibitem [\protect \citeauthoryear {%
Temam%
}{%
Temam%
}{%
{\protect \APACyear {1991}}%
}]{%
temam}
\APACinsertmetastar {%
temam}%
\begin{APACrefauthors}%
Temam, R.%
\end{APACrefauthors}%
\unskip\
\newblock
\APACrefYearMonthDay{1991}{}{}.
\newblock
{\BBOQ}\APACrefatitle {Infinite-dimensional dynamical systems in mechanics an
  physics / Roger Temam} {Infinite-dimensional dynamical systems in mechanics
  an physics / roger temam}.{\BBCQ}
\newblock
\APACjournalVolNumPages{SERBIULA (sistema Librum 2.0)}{}{}{}.
\PrintBackRefs{\CurrentBib}

\bibitem [\protect \citeauthoryear {%
Yu%
, Yadong%
\BCBL {}\ \BBA {} Di%
}{%
Yu%
\ \protect \BOthers {.}}{%
{\protect \APACyear {2020}}%
}]{%
jiali2020}
\APACinsertmetastar {%
jiali2020}%
\begin{APACrefauthors}%
Yu, J.%
, Yadong, S.%
\BCBL {}\ \BBA {} Di, H.%
\end{APACrefauthors}%
\unskip\
\newblock
\APACrefYearMonthDay{2020}{}{}.
\newblock
{\BBOQ}\APACrefatitle {Global existence, nonexistence, and decay of solutions
  for a viscoelastic wave equation with nonlinear boundary damping and source
  terms} {Global existence, nonexistence, and decay of solutions for a
  viscoelastic wave equation with nonlinear boundary damping and source
  terms}.{\BBCQ}
\newblock
\APACjournalVolNumPages{Journal of Mathematical Physics}{61}{}{071503}.
\newblock
\begin{APACrefURL} \url{https://doi.org/10.1063/5.0012614} \end{APACrefURL}
\newblock
\begin{APACrefDOI} \doi{10.1063/5.0012614} \end{APACrefDOI}
\PrintBackRefs{\CurrentBib}

\bibitem [\protect \citeauthoryear {%
Özsarı%
}{%
Özsarı%
}{%
{\protect \APACyear {2012}}%
}]{%
O2012}
\APACinsertmetastar {%
O2012}%
\begin{APACrefauthors}%
Özsarı, T.%
\end{APACrefauthors}%
\unskip\
\newblock
\APACrefYearMonthDay{2012}{}{}.
\newblock
{\BBOQ}\APACrefatitle {Weakly-damped focusing nonlinear {S}chrödinger
  equations with {D}irichlet control} {Weakly-damped focusing nonlinear
  {S}chrödinger equations with {D}irichlet control}.{\BBCQ}
\newblock
\APACjournalVolNumPages{Journal of Mathematical Analysis and
  Applications}{389}{}{84–97}.
\newblock
\begin{APACrefURL} \url{https://doi.org/10.1016/j.jmaa.2011.11.053}
  \end{APACrefURL}
\newblock
\begin{APACrefDOI} \doi{10.1016/j.jmaa.2011.11.053} \end{APACrefDOI}
\PrintBackRefs{\CurrentBib}

\bibitem [\protect \citeauthoryear {%
Özsarı%
}{%
Özsarı%
}{%
{\protect \APACyear {2013}}%
}]{%
O2013}
\APACinsertmetastar {%
O2013}%
\begin{APACrefauthors}%
Özsarı, T.%
\end{APACrefauthors}%
\unskip\
\newblock
\APACrefYearMonthDay{2013}{}{}.
\newblock
{\BBOQ}\APACrefatitle {Global existence and open loop exponential stabilization
  of weak solutions for nonlinear {S}chrödinger equations with localized
  external {N}eumann manipulation} {Global existence and open loop exponential
  stabilization of weak solutions for nonlinear {S}chrödinger equations with
  localized external {N}eumann manipulation}.{\BBCQ}
\newblock
\APACjournalVolNumPages{Nonlinear Analysis: Theory, Methods and
  Applications}{80}{}{179–193}.
\newblock
\begin{APACrefURL} \url{https://doi.org/10.1016/j.na.2012.10.006}
  \end{APACrefURL}
\newblock
\begin{APACrefDOI} \doi{10.1016/j.na.2012.10.006} \end{APACrefDOI}
\PrintBackRefs{\CurrentBib}

\bibitem [\protect \citeauthoryear {%
Özsarı%
}{%
Özsarı%
}{%
{\protect \APACyear {2015}}%
}]{%
O2015}
\APACinsertmetastar {%
O2015}%
\begin{APACrefauthors}%
Özsarı, T.%
\end{APACrefauthors}%
\unskip\
\newblock
\APACrefYearMonthDay{2015}{}{}.
\newblock
{\BBOQ}\APACrefatitle {Well-posedness for nonlinear {S}chrödinger equations
  with boundary forces in low dimensions by {S}trichartz estimates}
  {Well-posedness for nonlinear {S}chrödinger equations with boundary forces
  in low dimensions by {S}trichartz estimates}.{\BBCQ}
\newblock
\APACjournalVolNumPages{Journal of Mathematical Analysis and
  Applications}{424}{}{}.
\newblock
\begin{APACrefURL} \url{https://doi.org/10.1016/j.jmaa.2014.11.034}
  \end{APACrefURL}
\newblock
\begin{APACrefDOI} \doi{10.1016/j.jmaa.2014.11.034} \end{APACrefDOI}
\PrintBackRefs{\CurrentBib}

\bibitem [\protect \citeauthoryear {%
Özsarı%
, Kalantarov%
\BCBL {}\ \BBA {} Lasiecka%
}{%
Özsarı%
\ \protect \BOthers {.}}{%
{\protect \APACyear {2011}}%
}]{%
OKL2011}
\APACinsertmetastar {%
OKL2011}%
\begin{APACrefauthors}%
Özsarı, T.%
, Kalantarov, V.%
\BCBL {}\ \BBA {} Lasiecka, I.%
\end{APACrefauthors}%
\unskip\
\newblock
\APACrefYearMonthDay{2011}{}{}.
\newblock
{\BBOQ}\APACrefatitle {Uniform decay rates for the energy of weakly damped
  defocusing semilinear {S}chrödinger equations with inhomogeneous {D}irichlet
  boundary control} {Uniform decay rates for the energy of weakly damped
  defocusing semilinear {S}chrödinger equations with inhomogeneous {D}irichlet
  boundary control}.{\BBCQ}
\newblock
\APACjournalVolNumPages{Journal of Differential Equations}{251}{}{1841-1863}.
\newblock
\begin{APACrefURL} \url{https://doi.org/10.1016/j.jde.2011.04.003}
  \end{APACrefURL}
\newblock
\begin{APACrefDOI} \doi{10.1016/j.jde.2011.04.003} \end{APACrefDOI}
\PrintBackRefs{\CurrentBib}

\end{thebibliography}
\end{document}